\documentclass[11 pt, reqno]{amsart}
\usepackage{amsmath,amssymb,mathrsfs,graphicx,bbm}
\usepackage{amsmath}
\usepackage{amsfonts}
\usepackage{latexsym}
\usepackage{amsthm}
\usepackage{amssymb,amscd}
\usepackage{xargs}
\usepackage{tikz}
\usepackage{stmaryrd}
\usepackage{pgf,tikz}
\usepackage{fancybox}
\usepackage{graphicx}
\usepackage{stmaryrd}
\usepackage{color}
\usepackage{enumitem}
\usepackage{mathpazo}
\usepackage[T1]{fontenc}
\usepackage{microtype}
\usepackage{cases}

\usepackage[colorinlistoftodos]{todonotes}
 \presetkeys{todonotes}%
{inline,backgroundcolor=gray!20,bordercolor=gray!30}{}
\tikzset{/tikz/notestyleraw/.append style={text=black}}

\usepackage{fullpage}

\newtheorem{thm}{Theorem}[section]

\newtheorem{lem}[thm]{Lemma}
\newtheorem{defn}[thm]{Definition}
\newtheorem{prop}[thm]{Proposition}

\newtheorem{cor}[thm]{Corollary}

\newtheorem{rmk}[]{Remark}

\newcommand{\be}{\begin{eqnarray}}
\newcommand{\ee}{\end{eqnarray}}
\newcommand{\ben}{\begin{eqnarray*}}
\newcommand{\een}{\end{eqnarray*}}
\newcommand{\beq}{\begin{equation}}
\newcommand{\eeq}{\end{equation}}
\newcommand{\beal}{\begin{aligned}}
\newcommand{\enal}{\end{aligned}}

\newcommand{\eps}{\varepsilon}

\newcommand{\T}{\mathbb{T}}
\newcommand{\R}{\mathbb{R}}
\newcommand{\Q}{\mathbb{Q}}

\newcommand{\N}{\mathbb{N}}

\newcommand{\Z}{\mathbb{Z}}

\newcommand{\om}{\omega}
\newcommand{\dt}{\delta}

\newcommand{\cD}{\mathcal{D}}

\newcommand{\cM}{\mathcal{M}}

\newcommand{\wt}{\widetilde }

\usepackage{boxedminipage}

\usepackage[unicode=true]{hyperref}
\hypersetup{
     colorlinks,
     linkcolor={black!10!blue},
     linkbordercolor = {black!100!blue},
     citecolor={red}
}

\usepackage{enumitem}
\makeatletter
\def\namedlabel#1#2{\begingroup
    #2%
    \def\@currentlabel{#2}%
    \phantomsection\label{#1}\endgroup
}
\makeatother

\title{\textsc{Convergence Rate of Birkhoff average for  toral quasi-periodic rotations And Applications}}
\author[S. Tu]{Son N.T. Tu$^\dagger$}
\address[S. Tu]{$^\dagger$Department of Mathematics, Baylor University\\
Waco, Texas 76706, USA}
\email{son\_tu@baylor.edu}

\author[J. Zhang]{Jianlu Zhang}
\address[J. Zhang]{State Key Laborotary of Mathematical Sciences,Academy of Mathematics and systems science\\
Chinese Academy of Sciences, Beijing 100190, China}
\email{jellychung1987@gmail.com}

\author[S. Zhu]{Siyao Zhu}
\address[S. Zhu]{State Key Laborotary of Mathematical Sciences,Academy of Mathematics and systems science\\
Chinese Academy of Sciences, Beijing 100190, China}
\email{zsydtc@amss.ac.cn}

\date{\today}
\thanks{$\dagger$Corresponding author}
\begin{document}

 \begin{abstract} 
In this paper, we establish quantitative Denjoy--Koksma type estimates for higher-dimensional quasi-periodic torus rotations. For Diophantine frequency vectors, we establish quantitative estimates on the discrepancy between Birkhoff averages and spatial averages for observables with various Besov-type regularities. 
The Besov framework is naturally adapted to the small-divisor structure of the problem, allowing a sharp analysis of resonant frequencies and yielding optimal rates throughout the entire parameter range. These estimates yield, to the best of our knowledge, the sharpest currently available convergence rates for H\"older continuous observables.
As applications, we obtain substantially improved quantitative homogenization results for Hamilton--Jacobi equations in spatially quasi-periodic settings, as well as nearly optimal statistical regularity estimates for invariant measures under perturbations. 
 \end{abstract}

\maketitle

\numberwithin{equation}{section}

\section{Introduction}

Let $\T^n := \R^n/\Z^n$ with $n\geq 2$ denote the standard flat $n$-dimensional torus, equipped with the Euclidean metric $|\cdot| $. 
Any observable $\phi:\T^n\to \R$ can be identified with a function on $\R^n$ that is $1$-periodic in each coordinate direction. 
We consider the translation flow $\rho_\omega^t:\T^n\to\T^n$ maps $\rho_\omega^t(x)=x+t\omega$, which is the solution flow of the ordinary differential equation
\begin{equation}\label{eq:odeIntro}
    \dot{x}= \omega \in\R^n, \quad x(0)=x_0    
\end{equation}
with the frequency $\omega\in\R^n$. 
If $\omega$ is non-resonant, then each orbit of \eqref{eq:odeIntro} is dense on $\T^n$, and the flow \eqref{eq:odeIntro} is uniquely ergodic on $\T^n$ with invariant measure \(m\). Birkhoff's ergodic theorem yields, for every $\phi\in C(\T^n)$, uniformly in \(x\in\T^n\),
\begin{equation}\label{eq:mainest}
    \lim_{T\to \infty}	\underbrace{\frac1T\int_0^T \phi(\rho_\om^t(x)) dt}_{\text{Birkhoff average}} = \cM(\phi) := \underbrace{\int_{\T^n} \phi(y) dm(y)}_{\text{spatial average}}. 
\end{equation}
It is natural to seek quantitative convergence rates for Birkhoff averages under suitable arithmetic assumptions, especially for low-regularity observables, in contrast with existing results for smoother functions \cite{KleinLiuMelo2021, hu_polynomial_2025}. Except for the study of quantitative rates, many topics of independent interest technically related, e.g., the ergodicity of skew product \cite{KleinLiuMelo2021}, quantitative homogenization of Hamilton--Jacobi equations \cite{hu_polynomial_2025, mitake_tran_yu_2019_rate_of_convergence_ARMA}, quantitative perturbations of effective dynamics \cite{bolotin2003dynpde, galatolo_quantitative_2022, sorrentino_statistical_2026, armstrong_error_2014} and fractal geometry
 \cite{naito_fractal_1996, rynne_fractal_1998}. 

The first goal of this paper is to determine the convergence rate in \eqref{eq:mainest}. 
We develop a Besov-space approach to study \eqref{eq:mainest}, obtaining rates that are optimal to the best of our knowledge (a discrete analog of \eqref{eq:mainest} is also included). Such a result directly applies to H\"older spaces, and the convergence exponent is explicitly in terms of the regularity of the observables. We also construct examples that show that the exponents are sharp or nearly sharp in all different cases.

Our second goal is to apply the improved convergence rates to two important problems: the quantitative quasi-periodic homogenization of Hamilton-Jacobi equations and statistical regularity of invariant measures under perturbations. The former is derived from the study of correctors for almost periodic Hamilton-Jacobi equations and proposed in a series of works \cite{LPV,lions_correctors_2003,lions_homogenization_2005}, see also \cite{Ishii_almost_periodic_1999} for a qualitative result. The latter was initiated 
in \cite{baladi2014linear}, where a general linear response for perturbative dynamical systems was questioned.  These two topics motivate our thorough study of the convergence rate of \eqref{eq:mainest}, see \cite{hu_polynomial_2025,sorrentino_statistical_2026} for our recent works. In this paper, we sharpen the upper bound for the convergence rate of quasi-periodic homogenization and, for a prototype problem proposed 
in \cite{lions_correctors_2003}, establish an optimal lower bound.
Besides, we derive a sharp exponent of statistical regularity for invariant measures under perturbations of the dynamical system \eqref{eq:odeIntro}, with examples showing the optimality of the exponent in a two-dimensional setting.

\subsection{Literature} \label{subSec:literature}

In the discrete setting, a quantitative rate was first obtained by the so-called Denjoy-- Koksma inequality.  Herman \cite[Chapter VI.3]{Herman1979French} proved the rate for functions of bounded variation:
\begin{equation}\label{eq:d-k}
    \Big\| \frac{1}{N}\sum_{j=0}^{N-1} \phi(x+j\omega) - \int_{\mathbb{T}} \phi(x)\,dx \Big\|_{L^\infty}
    \leq \frac{{\rm Var}(\phi)}{N}, \qquad \phi \in \mathrm{BV}(\mathbb{T}), \omega\in \R\backslash \Q, 
\end{equation}
where $N\in\Z_+$ satisfies $|\om-M/N|\leq 1/N^2$ for some integer $M$. Such an inequality can be interpreted as a discrete analogue of our consideration\footnote{Such a setting is quite like a sectional treatment for the system $(\T^2,m,\rho_\om^k)$ with $\om=(\om_1,1)\in\T^2$ non-resonant and $k\in\Z$. See also Corollary \ref{cor:discreteRate} for more details}. 
Moreover, due to the Gottschalk-Hedlund theorem (see \cite[Section I.8.10]{katznelson_introduction_2004}), such an $O(1/{N})$ rate is in fact optimal. A high dimensional generalization was later obtained in \cite{KleinLiuMelo2021} for H\"older continuous observables $f \in C^{0,\alpha}(\mathbb{T}^n,\R)$ and Diophantine frequencies $\omega \in \mathbb{R}^n$, for which a much worse rate was obtained. Nonetheless, a low dimensional example was contructed in \cite{KleinLiuMelo2021}, which shows that for H\"older continuous observables and almost every $\om\in\R\backslash\Q$, the rate $\mathcal{O}({\ln^{3\alpha} N}/{N^\alpha})$ is nearly optimal.
We also note that J. C. Yoccoz constructed an analytic function 
\(\phi\in C^\omega(\T^2,\mathbb{R})\) whose discrete Birkhoff averages \eqref{eq:d-k} can converge arbitrarily slowly for Liouville frequencies \(\omega\in\mathbb{R}^2\); see \cite[Appendix~1]{yoccoz_1995}.

For continuous observables, when $\omega\in\mathbb{Q}^n$, the optimal rate of \eqref{eq:mainest} is $\mathcal{O}(T^{-1})$, see \cite[Lemma~4.2]{mitake_tran_yu_2019_rate_of_convergence_ARMA} and \cite{Neukamm2017,tu_2018_rate_asymptotic}. The same rate was obtained in \cite[Proposition~2.8]{hu_polynomial_2025} for $\omega\in \mathcal{D}(\sigma,C_\omega,n)$ (defined in \eqref{eq:diophan}) and $f\in H^{\sigma+\frac{n}{2}+\varepsilon}(\T^n)$ (see Subsection~\ref{subSection:FunctionSpaces} for the definition of $H^s(\T^n)$). For H\"older observables $f\in C^{0,\alpha}(\T^n)$ with $n\geq 2$, several convergence rates were established in \cite[Proposition~2.12]{hu_polynomial_2025} using the metric theory of Diophantine approximation and other arithmetic tools,  which are far from being optimal. 
Additionally, we would like to mention that the convergence rates studied here are closely related to average convergence rates for almost-periodic functions; see \cite{besicovitch_almost_1954, rynne_fractal_1998, naito_fractal_1996} for further details.  As extensions, fractal dimensions were also estimated there.

\subsubsection{Convergence rates for quasi-periodic homogenization of first-order Hamilton--Jacobi equations}
\label{subsubSec:homogenization}

Periodic homogenization of Hamilton--Jacobi equations was initiated in \cite{LPV}. The perturbed test function method \cite{evans_perturbed_1989,evans_homogenization_1992},
combined with discount approximation, yielded the first rate
\(\mathcal{O}(\varepsilon^{1/3})\) in \cite{capuzzo-dolcetta_rate_2001}. A conditional and the optimal rate \(\mathcal{O}(\varepsilon)\) were established in
\cite{mitake_tran_yu_2019_rate_of_convergence_ARMA} and
\cite{tran_optimal_2025}, respectively.
Quantitative periodic homogenization has advanced significantly in recent
years; see, among many others,
\cite{han_rate_2023,han_quantitative_2025,hanTu_quantitative_2025,mitake_system_rate_2025,mitake_quantitative_2026,mitake_quantitative_2025,ding_quantitative_2026}
and the references therein. Earlier nearly optimal rates include
\cite{jing_effective_2020,tu_2018_rate_asymptotic}. Beyond one dimension, sharp examples for the nonconvex setting were recently studied in \cite{jang_optimal_2026}.

By contrast, the quasi-periodic case remains far less understood. Qualitative
results were established in
\cite{Ishii_almost_periodic_1999,lions_homogenization_2005, tran_hamilton-jacobi_2021}; see also
\cite{shen_convergence_2015} for a rate in the elliptic setting. For
Hamiltonians with finite-range spatial dependence, a rate was obtained in
\cite{armstrong_error_2014}. In this paper, a modulus-based rate for
the almost-periodic Hamiltonians was derived by following the techniques in \cite{capuzzo-dolcetta_rate_2001}. See also \cite{souganidis_stochastic_1999,rezakhanlou_homogenization_2000,caffarelli_rates_2010} for related qualitative and quantitative results for homogenization of stochastic settings.
In \cite{hu_polynomial_2025}, an algebraic convergence rate was established in one dimension under additional assumptions, motivating us to seek an improved rate in the present work.

\subsubsection{Statistical Regularity of Invariant measures with respect the perturbations} \label{subsubSec:statis-regu}
For uniformly hyperbolic systems, the statistical regularity is well explored, with results on Lipschitz regularity (the so-called \textbf{linear response}) and, in some cases, differentiability with respect to perturbations can be available (see \cite{baladi2014linear, galatolo_quadratic_2020}). In such systems, it is also known that generic minimal measures is uniquely ergodic and supported on a single hyperbolic periodic orbit \cite{contreras_ground_2016}. However, for general dynamical systems, like Hamiltonian systems, uniform hyperbolicity is not available, so it's still open whether the uniqueness result in \cite{contreras_ground_2016} still holds.

In this paper, the system \eqref{eq:odeIntro} is totally neutral, so small perturbations of it cannot be uniformly hyperbolic. For general systems without hyperbolicity, the statistical regularity was discussed only in \cite{galatolo_quantitative_2022, sorrentino_statistical_2026}. Nonetheless, making use of the sharp estimate of \eqref{eq:mainest}, we can show that the invariant measure is H\"older continuous with respect to the perturbations (under the $W_1$ Wasserstein distance). By proposing examples, we show that the H\"older exponent is optimal in the case $n=2$. Such examples also indicate that linear response is impossible for perturbations of  \eqref{eq:odeIntro}. By the same strategy in \cite{sorrentino_statistical_2026}, such a H\"older regularity can be similarly considered for nearly integrable Hamiltonian systems.

\subsection{Setting and preliminary} We denote by $|k| = |k|_2$, the standard Euclidean norm on $\R^2$. Sometimes we sill use also $|k|_\infty = \max\{k_i: i=1,\ldots,n\}$ if $k = (k_1,\ldots, k_n)\in \R^n$. 
We say that \(\omega\in\mathbb R^n\) is Diophantine of {\bf index} \(\sigma>0\) if there exists \(C_\omega>0\) such that
\begin{equation}\label{eq:diophan}
	| k\cdot \omega | \geq C_\omega |k|^{-\sigma} 
	\quad \text{for all } k \in \mathbb{Z}^n \backslash \{0\}.   
\end{equation}
We say that \(\omega\in \mathcal{D}(\sigma,C_\omega,n)\) has {\bf exact} Diophantine index \(\sigma\) if there exists \(k_j\in\mathbb Z^n\) with \(|k_j|\to\infty\) such that
\begin{equation}\label{eq:exactDio}
	C_\omega |k_j|^{-\sigma}\le |k_j\cdot\omega|
	\leq \widehat C_\omega |k_j|^{-\sigma}.
\end{equation}
%
We denote by $\mathcal{D}(\sigma, C_\omega, n)$ the set of all Diophantine frequencies satisfying \eqref{eq:diophan}. As is proved in \cite{JurgenKAM1999}, $\sigma\geq n-1$. For $\sigma>n-1$, almost every $\om\in\R^n$ belongs to $\cD(\sigma, C_\om,n)$ for a suitably small $C_\om>0$. If $\sigma=n-1$, then $\mathcal{D}(n-1, C_\om,n)$ always has measure zero, but the Hausdorff dimension of it could tend to $n$ as $C_\om\rightarrow 0_+$.

We refer to Subsection \ref{subSection:FunctionSpaces} for some definition of the H\"older space $C^{k,\alpha}(\T^n)$ with $k\in \N$ and $\alpha\in (0,1]$. If $f\in L^1(\mathbb{T}^n)$, its Fourier transform $\widehat{f}:\mathbb{Z}^n\to\mathbb{C}$ is defined by $\widehat{f}(k) = \int_{\mathbb{T}^n} f(x)e^{-2\pi i k \cdot x}\;dx$ for $k \in \mathbb{Z}^n$. 
  %
We briefly recall the Besov spaces \(B^s_{p,q}(\T^n)\), which will be used to study convergence rates; see Section \ref{sec:Prelim}. For \(f\in C^{0,\alpha}(\T^n)\), \(0<\alpha\leq1\), the Littlewood--Paley decomposition
\begin{equation*}
    f=\sum_{j=-1}^\infty \Delta_j f,
\qquad
\widehat{\Delta_j f}(k)=\varphi_j(k)\widehat f(k
),    
\end{equation*}
converges absolutely and uniformly, where \(\varphi_j\) is localized to frequencies of size \(2^j\); see \cite{triebel_theory_1983,muscalu_classical_2013,stein_singular_1970,adams_function_1996}. Compared to the direct Fourier series as in \cite[Proposition 2.8]{hu_polynomial_2025}, this decomposition is available under much weaker regularity assumptions. 
The Besov space $B^s_{p,q}(\T^n)$ with $s\in \R$ and $1\leq p,q\leq \infty$, is defined as the space of all distributions $f$ such that
\begin{equation*}
\begin{aligned}
    & \Vert f\Vert_{B^s_{p,q}(\T^n)} = \left(\sum_{j=-1}^\infty (2^{js}\Vert \Delta_j f\Vert_{L^p(\T^n)})^q\right)^{1/q}  && q < \infty, \\
    & \Vert f\Vert_{B^s_{p,\infty}(\T^n)} = \sup_{j\geq -1} 2^{js}\Vert \Delta_j f\Vert_{L^p(\T^n)}  && q = \infty . 
\end{aligned}
\end{equation*}
For completeness, we introduce the Wiener algebra and Bessel potential spaces, along with the corresponding results, in Appendix \ref{a0}. \medskip

If $f$ is regular enough (H\"older is enough), then $f = \sum_{j=-1}^\infty \Delta_j f$ uniformly and absolutely, hence
\begin{equation*}
 \left|\frac{1}{T}\int_0^T f(x+\omega t)\;dt - \int_{\T^n} f(y)\;dy\right|  
     \leq \sum_{j=-1}^\infty \left|\int_0^T \Delta_j f(\omega t)\;dt\right|,
\end{equation*}
thus the proof reduces to get an upper bound for each term $\int_0^T \Delta_j f(\omega t)\;dt$, which is done in Lemma \ref{lem:EstBlock}. To estimate this term, the sharp dyadic counting in Lemma~\ref{lem:keyThm1}
transform the Diophantine condition into quantitative information on the distribution of small divisors. In this way, the Besov formulation identifies the natural regularity scale to obtain optimal estimates as in our Theorem \ref{thm:BesovHolderCounting}. 
 In particular, the results apply to the
  H\"older class \(C^{0,\alpha}(\T^n)\), which coincides with \(B^\alpha_{\infty,\infty}(\T^n)\) for \(\alpha\in(0,1)\) (see \cite[Lemma 8.6]{muscalu_classical_2013}), and more generally
 \begin{equation}\label{eq:embedHolder}
 \begin{aligned}
      & B^\alpha_{\infty, 1}(\mathbb{T}^n) \hookrightarrow  C^{0,\alpha}(\mathbb{T}^n) \hookrightarrow B^{\alpha-\varepsilon}_{\infty, 1}(\mathbb{T}^n),\quad \forall\  \alpha\in(0,1], \eps\in (0,\alpha) \\ 
      & C^{k,\alpha}(\T^n) = B^s_{\infty,\infty}(\T^n), \quad s = k+\alpha, k\in \N, \alpha\in ( 0,1), 
      \qquad\qquad 
      C^{0,1}(\T^n) \subsetneq B^{1}_{\infty,\infty}(\T^n).
 \end{aligned}
 \end{equation}
 These embedding results are neat and precise, so they make a sharp estimate of the convergence rate for Birkhoff averages possible. After that, they are then used in quantitative homogenization for quasi-periodic Hamilton--Jacobi equations and in the study of statistical stability of invariant measures; see Theorems~\ref{thm:homogenization} and~\ref{thm:st-regu}.

\subsection{Main results} 
Our first result establishes a general convergence rate for \eqref{eq:mainest} when the observable belongs to a Besov space.
We consider $f\in B^s_{p,q}(\T^n)$ with $1\leq p,q \leq \infty$. 
The following assumption $s>n/p$ is needed if $q\in (1,\infty]$ to guarantee that $B^s_{p,q}(\T^n)\subset C(\T^n)$, while if $q=1$, the relaxed condition $s\geq n/p$ is enough to make $B^s_{p,1}(\T^n)\subset C(\T^n)$. 

\begin{thm}[Besov Space]\label{thm:BesovHolderCounting}

Let $n\geq 2$, $\omega\in \mathcal{D}(\sigma, C_\omega,n)$, $s>0$ and $x\in \T^n$. Let 
    \begin{equation*}
        f\in B^s_{p,q}(\T^n), \qquad 1\leq p,q \leq \infty, \qquad s>\frac{n}{p}.
    \end{equation*}
    There exists a positive constant $C=C(\sigma, s, p,q,n)>0$ such that
    \begin{subnumcases}{\displaystyle
        \left|
            \frac{1}{T}\int_0^T f(x+\omega t)\;dt
            -
            \int_{\T^n}f(y)\;dy
        \right|
    \leq
    \frac{C}{C_\omega}
    \Vert f\Vert_{B^s_{p,q}(\T^n)}}
        T^{-\frac{s-n/p}{\sigma-n/p}},
            & $\displaystyle \tfrac{n}{p}<s<\sigma$,
            \label{eq:BesovAllPQa}\\
        T^{-1}(\log T)^{1-1/q},
            & $\displaystyle \tfrac{n}{p}<s=\sigma$,
            \label{eq:BesovAllPQb}\\
        T^{-1},
            & $\displaystyle s>\max\{\sigma, \tfrac{n}{p}\}$.
            \label{eq:BesovAllPQc}
    \end{subnumcases}

     Furthermore, for the endpoint $q=1$, the supercritical case \eqref{eq:BesovAllPQc} can be relaxed to $s\geq \max \{\sigma, \frac{n}{p}\}$:
    \begin{equation}
    \begin{aligned}
     &  
            \left|\frac{1}{T} \int_0^T  f(x+\omega t)\;dt - \int_{\T^n} f(y)\;dy \right| 
            \leq 
            \frac{C(\sigma,p)}{C_\omega}
            \Vert f\Vert_{{B^s_{p,1}}(\T^n)} \cdot \frac{1}{T},              && s\geq \max \{\sigma, \tfrac{n}{p}\}. 
            \label{eq:BesovALemmaP1}
	   \end{aligned}
     \end{equation}
\end{thm}

An immediate consequence is that, by taking $p=\infty$ in \eqref{eq:BesovAllPQa}--\eqref{eq:BesovAllPQc}, the dimension $n$ disappears from the estimate, yielding the sharp rate. 
In particular, when $p=q=\infty$, the Besov space coincides with the corresponding H\"older space, and our estimate improves the previously known rates in the literature to the optimal rate, whose sharpness is established in Theorem \ref{thm2:examples}.

\begin{cor}[H\"older space]\label{cor:Holder} Let $n\geq 2$, $\omega\in \mathcal{D}(\sigma, C_\omega, n)$, and $x\in \T^n$. If $f\in C^{k,\alpha}(\T^n)$ for $k\in \N$ and $\alpha\in(0,1]$ then
\begin{align}\label{eq:corHolder}        
      \left|
            \frac{1}{T}\int_0^T f(x+\omega t)\;dt - \int_{\T^n} f(y)\;dy 
        \right| 
        \leq 
        \frac{C(\sigma,k+\alpha)}{C_\omega} \Vert f\Vert_{C^{k,\alpha}(\T^n)}  
        \begin{cases}
        T^{-1}
            & k+\alpha > \sigma, \\[1mm]
             T^{-1}\log(T)
            & k+\alpha = \sigma, \\[1mm]
        T^{-\frac{k+\alpha}{\sigma}} 
            & k+\alpha < \sigma.
               \end{cases}
\end{align}
\end{cor}

The logarithmic factor in \eqref{eq:corHolder} is new and exists in all dimensions. Moreover, when \(k+\alpha=\sigma\), the rate \(\mathcal{O}(T^{-1}\log T)\) is {\bf optimal} in view of \eqref{eq:exam-besov} later on. This rate improves the $\mathcal{O}(T^{-1}\log^3 T)$ rate in \cite[Theorem 1]{KleinLiuMelo2021} for $n=2$, $\omega\in\mathcal{D}(1,C_\omega,2)$ and $f\in C^{0,1}(\T^n)\subsetneq B^1_{\infty,\infty}(\T^n)$. Notice that any $f\in B^1_{\infty,\infty}(\T^n)$ satisfies $|f(x+h)-2f(x)+f(x-h)|\leq C|h|$ (see \cite[Chapter~V]{stein_singular_1970}).  We also record results for the related spaces
$L^{s+\frac np+\varepsilon,p}(\T^n)
\subseteq A^s(\T^n)
\subseteq B^s_{\infty,1}(\T^n)
\subset B^s_{\infty,\infty}(\T^n)$, where \(L^{s,p}\) is the Bessel potential space and \(A^s\) the weighted Wiener algebra (see Appendix~\ref{a0} and Proposition~\ref{prop:BirkhoffRateBesselWiener} for the corresponding rates). Anyway, we are able to show the optimality of the estimate in \eqref{eq:corHolder}:

\begin{thm}[Sharpness] \label{thm2:examples}
Let $n\geq 2$, $\omega\in \mathcal{D}(\sigma, C_\omega,n)$. 
\begin{itemize}
    \item[(i)] \emph{(subcritical)} If $\omega$ has an exact Diophantine index $\sigma$, as defined in \eqref{eq:exactDio}, 
    then for every $0<s<\sigma$, there exists $f\in B^s_{\infty,\infty}(\T^n)$, a constant $C=C(\sigma, s, \omega)$, and a sequence $T_N\to\infty$ as $N\to\infty$ such that 
    \begin{align}\label{eq:exam-besovsalpha}
        \left|\frac{1}{T_N}\int_0^{T_N} f(\omega t)\,dt - \int_{\T^n} f(x)\;dx \right|
        \geq C\left(\frac{1}{T_N}\right)^{s/\sigma}. 
    \end{align}
    
    \item[(ii)] \emph{(critical)} If $\sigma = n-1$ then 
    there exists $f\in B^{n-1}_{\infty,\infty}(\T^n)$, a constant $C=C(\omega, n)$, and a sequence $T_N\to\infty$ as $N\to\infty$ such that 
    \begin{align}\label{eq:exam-besov}
        \left|\frac{1}{T_N}\int_0^{T_N} f(\omega t)\,dt - \int_{\T^n} f(x)\;dx \right|
        \geq C\cdot\frac{\log(T_N)}{T_N}. 
    \end{align}

    \item[(iii)] \emph{(supercritical)}
    For almost every $\om\in\R^n$, there exists a function $f\in B_{\infty,\infty}^{s}(\T^n)$ for all $s>0$, and a sequence $T_j\to \infty$ as $j\to \infty$ such that
        \begin{align}\label{eq:besovAexample}
            \left|
                \frac{1}{T_j}\int_0^{T_j} f(\omega t)\;dt
                -
                \int_{\T^n} f(y)\;dy
            \right|
            \geq C(\omega)\cdot \frac{1}{T_j}. 
        \end{align}

    \item[(iv)] \emph{(subcritical)}
    For almost every  $\omega\in\R^n$, there exists $f\in B_{\infty,\infty}^{\alpha}(\T^n)$ with $\alpha\in (0,1)$ and a sequence $T_j\rightarrow \infty$ as $j\rightarrow+\infty$, such that 
    \begin{equation}\label{eq:NearlyOptimal2D}
    C_1(\om)\left(\frac{1}{T_j}\right)^\alpha \leq          \left|
            \frac{1}{T_j}\int_0^{T_j} f(\omega t)\;dt - \int_{\T^n} f(x)\;dx 
        \right| 
         \leq  C_2(\om,\eps)\left(\frac{1}{T_j}\right)^{\frac{\alpha}{n-1+\varepsilon}} 
    \end{equation}
        for all $\varepsilon>0$. Consequently, the subcritical rate $\mathcal{O}(T^{-\frac{k+\alpha}{\sigma}})$ in \eqref{eq:corHolder} is nearly optimal for $n=2$.
\end{itemize}
\end{thm}
We note that \eqref{eq:exam-besovsalpha} and \eqref{eq:exam-besov}, concern a measure-zero set of frequencies, whereas (iii) and (iv) hold for almost every $\omega\in\mathbb{R}^n$. A typical example of exact Diophantine frequency for \eqref{eq:exam-besovsalpha} is $\omega = (1,a,a^2,\ldots, a^{n-1})$ where $a$ is an algebraic irrational of degree $n$. The constructions of examples for \eqref{eq:exam-besovsalpha} and \eqref{eq:exam-besov} have appeared in various contexts (see \cite{moshchevitin_example_2023, moshchevitin_nonrecursiveness_1991} and references therein). To the best of our knowledge, such constructions have not previously been used to obtain optimal convergence rates for Birkhoff averages in Besov spaces as in our Theorem \ref{thm2:examples}-(i), (ii). 
The examples in Theorem \ref{thm2:examples}-(iv) are essentially two-dimensional, (relying on the Diophantine approximation), which gives the only known lower bound for \(C^{0,\alpha}(\T^n)\), \(\alpha\in(0,1)\), \(n\geq2\) in literature.\medskip

As an analogue  of Theorem \ref{thm:BesovHolderCounting}, a convergence rate of Birkhoff avarages for Besov functions in the discrete setting can be drawn (see
Proposition \ref{prop:DiscreteBesov} for more details). Here we present the following Corollary for H\"older observables:

\begin{cor}[Discrete analogue]\label{cor:discreteRate} 
Let $n\geq 1$, $\omega \in \R^n$ be such that $(\omega,1) \in \mathcal{D}(\sigma, C_\omega, n+1)$, $x\in \T^n$.
For $f\in C^{k,\alpha}(\T^n)$ where $k\in \N$ and $0<\alpha \leq 1$, 
\begin{align}\label{eq:corHolderDiscrete}        
      \left|
            \frac{1}{N}\sum_{\ell=0}^{N-1}  f(x+\ell\omega )
            - 
            \int_{\T^n} f(y)\;dy 
    \right| 
        \leq 
        \frac{C(\sigma,k+\alpha)}{C_\omega} \Vert f\Vert_{C^{k,\alpha}(\T^n)}  
        \begin{cases}
        N^{-1}
            & k+\alpha > \sigma, \\[1mm]
             N^{-1}\log(N)
            & k+\alpha = \sigma, \\[1mm]
        N^{-\frac{k+\alpha}{\sigma}} 
            & k+\alpha < \sigma.
               \end{cases}
\end{align}
\end{cor}

The discrete case has been the main focus of earlier works and is often studied under the framework of Denjoy--Koksma type inequalities, going back to Herman \cite[Chapter VI.3]{Herman1979French}; see also the recent work \cite{KleinLiuMelo2021}. 
Our estimate \eqref{eq:corHolderDiscrete} yields, under the present assumptions, a sharper rate than those previously available in this direction.

\subsection{Applications} 

Our first application concerns convergence rates for $u^\varepsilon$, the solution of \eqref{eq:HJ-intro}, to $u$, the solution of \eqref{eq:intro:C}, in the homogenization of first-order Hamilton--Jacobi equations with quasi-periodic potentials\cite{Ishii_almost_periodic_1999, tran_hamilton-jacobi_2021}. See Subsection \ref{subsubSec:homogenization} for background and more references. The following theorem yields the optimal lower bound, improving \cite{hu_polynomial_2025}.

Let $u_0\in \mathrm{Lip}(\R)\cap L^\infty(\R)$, the space of Lipschitz continuous function on $\R$, we consider
\begin{equation}\label{eq:HJ-intro}
\begin{cases}
u_t^\varepsilon+\dfrac12|Du^\varepsilon|^2
    -V\left(\dfrac{x}{\varepsilon}\right)=0,
    & (x,t)\in\mathbb R\times(0,\infty),\\
u^\varepsilon(x,0)=u_0(x),
    &x\in\mathbb R,
\end{cases}
\end{equation}
where the potential is quasi-periodic: $V(x)=f(\omega x)$ with $
f\in C(\mathbb T^n)$ and $\omega\in\mathbb R^n$ is \emph{non-resonant}, i.e. $k\cdot \omega \neq 0$ for all $k\in \Z^n\backslash \{0\}$. If $\omega$ satisfies a Diophantine condition  as in  \eqref{eq:diophan}, then as \(\varepsilon\to0\), the solution \(u^\varepsilon\) converges to $u$, the solution of an effective Hamilton--Jacobi equation \cite{Ishii_almost_periodic_1999, tran_hamilton-jacobi_2021}, given by
\begin{equation}\label{eq:intro:C}
\begin{cases}
\begin{aligned}
    u_t+ \overline{H}(Du) &= 0 & &\quad\text{in}\;\mathbb{R}\times(0,\infty),\\
    u(x,0) &= u_0(x) & &\quad\text{on}\;\mathbb{R},
\end{aligned}
\end{cases} 
\end{equation}
where the \emph{effective Hamiltonian} $\overline{H}: p\in\R\rightarrow\R$  is the unique constant such that 
\begin{equation}\label{eq:CPdelta}
    \overline{H}(p) - \delta 
    \leq 
    \frac{1}{2}|p+Dv_{p,\delta}|^2 - V(x)
    \leq \overline{H}(p)+\delta \qquad\text{in}\;\mathbb{R}
\end{equation}
can be solved by a viscosity solution $v_{p,\delta}$ (\emph{approximated corrector}) for any $\delta>0$.
As observed in \cite{hu_polynomial_2025}, if an exact corrector exists ($\delta=0$), the lower bound for the homogenization error is controlled entirely by the sublinearity of such a corrector. In the simplest special case (corresponding to $0$-level set of $\overline{H}$, see \cite{hu_polynomial_2025} or Section \ref{Sec:AppHomogenization}), the relevant quantity takes the form
\begin{equation*}
    \frac{v_{p,0}(T)}{T}
 =
 \frac{1}{T}\int_0^T \sqrt{2f(\omega t)}\,dt
 -\int_{\mathbb T^n} \sqrt{2f(y)}\,dy. 
\end{equation*}
Thus, a key ingredient in estimating the homogenization error is to obtain quantitative convergence rates for Birkhoff averages.

\begin{thm}[Homogenization Rate]\label{thm:homogenization} 
Let \(n\geq 2\), $\varepsilon \in(0,1)$, \(\omega\in \mathcal{D}(\sigma,C_\omega,n)\) and assume
\begin{equation}\label{eq:asum}
    u_0\in C^{0,1}(\mathbb{R})\cap L^\infty(\R),  
        \qquad 
    V(x) = f(\omega x)
        \qquad 
    \text{with}\; f\in C(\T^n;\R)\;\text{and}\; x\in \R. 
\end{equation}
There exists a uniform constant $C:=C(\sigma, C_\om, u_0)>0$ such that for all $t\geq 0$, 
\begin{align}\label{eq:generalQuasi}
    u^\varepsilon(x,t) - u(x,t)\geq -C
    \begin{cases}
    \begin{aligned}
        &t^{1-\frac{1}{2\sigma}} \varepsilon^{\frac{1}{2\sigma}}   && f\in C^{0,1}(\T^n)\\ 
        &t^{1-\frac{1}{\sigma}} \varepsilon^{\frac{1}{\sigma}}|\log(\varepsilon)|    &&  f\in C^2(\T^n)\;\text{has only non-degenerate minima}. 
    \end{aligned}
    \end{cases}
\end{align}
Here we say \(f\in C^2(\T^n)\) has a nondegenerate minimum at \({\bf y}\) if \(D^2f({\bf y})\) is positive definite.
If, in addition, $\overline{H}\in C^{1,\beta}(\R)$, then for $f\in C^{0,1}(\T^n)$,
\begin{equation}\label{eq:generalQuasiUpper}
    u^\varepsilon(x,t) - u(x,t)\leq C
    \begin{cases}
    	    	t              & \qquad t \leq \varepsilon, \\
    	    	t^{1-\frac{\beta}{2\sigma(1+\beta)}} \varepsilon ^\frac{\beta}{2\sigma(1+\beta)}
			      	    	& \qquad t>\varepsilon.     	    	
    	    \end{cases}        
\end{equation}
Consequently, for almost every \(\omega\in \mathbb{R}^n\), the rates in \eqref{eq:generalQuasi} and \eqref{eq:generalQuasiUpper} hold with an arbitrary loss $\delta>0$, corresponding to the choice $\sigma \approx n-1+\delta$. 
\end{thm}

The proof is based on optimal control theory, exploiting fine properties of optimal trajectories, conservation of energy, and sharp quantitative ergodic estimates. This approach was initiated in \cite{mitake_tran_yu_2019_rate_of_convergence_ARMA} for the one-dimensional single-scale setting, extended to the multiscale case with the optimal rate by the first author in \cite{tu_2018_rate_asymptotic}, and later adapted to the quasi-periodic setting in \cite{hu_polynomial_2025}.

\begin{rmk}\label{rmk:prototypeCPDE} 
The following prototype presents a generic quasi-periodic example, which possesses a unique nondegenerate minimum:
    \begin{align}\label{eq:f}
        f({\bf y}) = \big( n - \sin(2\pi y_1) - \ldots - \sin(2\pi y_n)\big)^\gamma, \qquad {\bf y}=(y_1,\ldots, y_n)\in \T^n, \gamma>0. 
    \end{align} 
This prototype was introduced in \cite{lions_correctors_2003} to illustrate difficulties in constructing sublinear correctors for \eqref{eq:CPdelta}.
Here $f({\bf y})$ is Lipschitz only when \(\gamma\geq 1\). 
    For $n=2$, $\omega\in \mathcal{D}(1,C_\omega,2)$, Theorem \ref{thm:homogenization} gives, 
    \begin{align*}
      u^\varepsilon(x,1) - u(x,1) \geq -C
        \begin{cases}
             \varepsilon &\quad \gamma>1,\\
             \varepsilon |\log(\varepsilon)| &\quad \gamma=1,\\
             \varepsilon^\gamma &\quad \gamma<1. 
         \end{cases}
     \end{align*} 
    This sharpens \cite[Theorem 1.1]{hu_polynomial_2025}, lowering the critical \(\gamma\) from \(2\) to \(1\).
    For \eqref{eq:f}, \(\overline H(p)\) is explicit and belongs to \(C^{1,\frac12-\frac1\gamma}(\R)\) only for \(\gamma>2\).
\end{rmk}

The other application concerns the statistical regularity of invariant measures with respect to the perturbations. These problems were originally studied for hyperbolic systems \cite{baladi2014linear,galatolo_quadratic_2020}, and recently attracted consideration for Hamiltonian systems in the context of Aubry-Mather theory; see 
\cite{bolotin2003dynpde, sorrentino_statistical_2026} and the references therein for related developments and earlier works.

\begin{thm}[Statistical Regularity]\label{thm:st-regu}
	Let $V:\T^n\times [-1,1]\to \R^n$ be a Lipschitz vector field such that $V(x,0)\equiv \omega \in \mathcal{D}(\sigma, C_\omega, n)$ for $x\in \T^n$. 
	Assume
	\begin{equation}\label{ODEfor t}
		\dot{x} = V(x,\dt), \quad x\in\mathbb{T}^n
	\end{equation}
	is a parametrized ODE.
	There exists a uniform constant $C(\sigma, C_\om, \|V\|_{C^{0,1}})>0$, such that 	
	the $1$-Wasserstein distance satisfies
\begin{align}\label{eq:stat-regu-upper}
	\mathcal{W}_1(\mu_\delta,\mu_0)
	:=
	\sup_{\substack{f\in {\rm Lip}(\T^n,\R), \\ \| \nabla f	\|_{L^\infty}\leq 1}}
	\left|
	\int_{\mathbb{T}^n} f\,d\mu_\delta
	-
	\int_{\mathbb{T}^n} f\,d\mu_0
	\right|
	\leq
	C
	\begin{cases}
		\delta^{\frac{1}{1+\sigma}}, & \qquad \sigma>1,\\
		\delta^{\frac{1}{2}}|\log(\delta		)|, & \qquad \sigma=1,
	\end{cases}
\end{align}
	for any probability measure $\mu_\dt\in\mathbb P(\T^n,\R)$ 
which is {\bf invariant} with respect to the flow of \eqref{ODEfor t}.

The estimate \eqref{eq:stat-regu-upper} is also nearly optimal for $n=2$. Precisely,  there exists a sequence of vector fields $V(x,\dt_j)$ and a sequence of invariant measures $\mu_{\dt_j}$ associated with it ($\dt_j\rightarrow 0$ as $j\rightarrow+\infty$), such that
\begin{equation}\label{eq:stat-regu-lower}
	\mathcal{W}_1(\mu_{\delta_j},\mu_0)
	\geq
	C'
	\begin{cases}
		\delta_j^{\frac{1}{1+r}}, & \qquad \sigma_*(\om)>r>1,\\
		\delta_j^{\frac{1}{2}},   & \qquad \sigma_*(\om)=1,
	\end{cases}
\end{equation}
	where $C':=C'(\sigma, C_\om) >0$ is a uniform constant, and
	\begin{equation}\label{eq:defn-diophantine-type}
			\sigma_*(\om):=\inf\{\sigma>0\ |\ \om\in\cD(\sigma, C_\sigma,2)\;  \text{ for some }C_\sigma>0\}. 
	\end{equation}
\end{thm}

\begin{rmk}
In view of \eqref{eq:defn-diophantine-type}, if
\(\omega\in\mathcal D(\sigma,C_\omega,2)\), then $1\leq \sigma_*(\omega)\leq \sigma$. If \(\sigma_*(\omega)<\infty\) then for every
\(\varepsilon>0\) there exists \(C_\varepsilon>0\) such that $\omega\in\mathcal D(\sigma_*(\omega)+\varepsilon,C_\varepsilon,2)$.
Thus \eqref{eq:stat-regu-upper} is nearly optimal by means of \eqref{eq:stat-regu-lower} and such properties of $\sigma_*(\om)$.

We note that \eqref{ODEfor t} can be replaced by a Hamiltonian ODE to study statistical regularity of \textbf{Mather measures} under cohomology and potential perturbations; see \cite[Sections 4 and 5]{sorrentino_statistical_2026}. 
\end{rmk}

\subsection*{Organization of the Paper}
In Section~\ref{sec:Prelim}, we recall various notions of function spaces, especially Besov spaces, together with several embedding and inclusion results among them. Section~\ref{Sec:Rate} is devoted to the proofs of the main results of this paper. In Section \ref{a1}, we prove the discrete analog of the rates, including Corollary \ref{cor:discreteRate}. 
In Section~\ref{Sec:AppHomogenization}, we present applications of these results to improved convergence rates in the homogenization of Hamilton--Jacobi equations with quasi-periodic potentials. In Section \ref{Sec:AppStabilityMeasures}, we present applications to the statistical regularity of invariant measures. 
We included results for Bessel potentials and Wiener algebras in the Appendix.

\section{Preliminaries} \label{sec:Prelim}
Throughout the paper, with $\xi = (\xi_1,\ldots,\xi_n)\in \R^n$, we write $|\xi| = |\xi|_2 = \sqrt{\xi_1^2+\ldots+\xi_n^2}$, while $|\xi|_\infty = \max\{|\xi_1,\ldots, \xi_n|\}$. 

\subsection{Diophantine Approximations}\label{subsection:Diophantine}

For an irrational $\vartheta$, its continuous fraction is denoted by $\vartheta = [a_0;a_1,a_2,\ldots]$ where $a_0\in \Z$ and $a_i\in \N$ for $i=1,2,\ldots$, such that 
\begin{align*}
	\vartheta = a_0 + \frac{1}{a_1 + \frac{1}{a_2+\ldots}}. 
\end{align*}
We summarize several facts from the classical theory; see, for example, \cite{cassels_1957,schmidt_diophantine_1980}.
We define the sequence $\{(P_n/Q_n)\}_{n=1}^\infty$ in the following way:
    \begin{equation}\label{eq:DiophantinePQ}
    \begin{aligned}
        P        _{-1} &= 1, &P_0 &= a_0, &\qquad P_{k} &= a_k P_{k-1} + P_{k-2},\\
        Q_{-1} &= 0, &Q_0 &= 1, &\qquad Q_{k} &= a_k Q_{k-1} + Q_{k-2}.
    \end{aligned}    
    \end{equation}
    Then $P_k/Q_k$'s are called Diophantine approximations of $\vartheta$, and 
    \begin{equation}\label{eq:Diophantineestimates}
    \frac{1}{Q_j(Q_j + Q_{j+1})} <    \left|\vartheta - \frac{P_k}{Q_k}\right| < \frac{1}{Q_kQ_{k+1}} < \frac{1}{Q_k^2} \qquad\text{for}\;k=1,2,\ldots. 
    \end{equation}
    In particular, if $\vartheta$ is badly approximable, we have 
    \begin{equation}
        \left|\vartheta - \frac{p}{q}\right| 
        > \frac{C_\vartheta}{q^2} \qquad \text{for all integers}\; p,q\neq 0.
    \end{equation}
    An equivalent criterion is that $\vartheta$ is badly approximable if and only if $1\leq a_k \leq M_\vartheta$ for all $k=1,2,\ldots$. 
    In other words, if $\vartheta$ is badly approximable, the sequence $(P_k,Q_k)$ satisfies
    \begin{align*}
        \frac{C_\vartheta}{Q_k^2}	< \left|\vartheta - \frac{P_k}{Q_k}\right| < \frac{1}{Q_k^2}, \qquad k=1,2,\ldots. 
    \end{align*}
We want to point out that any frequency $\om:=(1,\vartheta)\in\R^2$ with $\vartheta$ badly approximable has to be contained in $\cD(1,C_\vartheta, 2)$.
We state a continued-fraction analogue of the Borel–Cantelli lemma.

\begin{thm}[Borel-Bernstein Theorem, {\cite[Theorem 1.1.1]{Bugeaud-2004ApproximationAlgebraicNumbers}}] \label{thm:borelbernstein}
For each irrational number \(\vartheta \in (0,1)\), let $\vartheta=[0;a_1(\vartheta),a_2(\vartheta),\ldots]$
denote its continued fraction expansion, and let
\(\varphi:\mathbb{N}\to(0,\infty)\) be a given function.
\begin{itemize}
\item[(i)] If $\sum_{i=1}^\infty \frac{1}{\varphi(n)} = \infty$ then for almost every $\vartheta\in [0,1]$, one has $a_i(\vartheta) \geq \varphi(i)$ for infinitely many $i\in \N$.
\item[(ii)] If $\sum_{i=1}^\infty \frac{1}{\varphi(n)} < \infty$ then for almost every $\vartheta\in [0,1]$, the inequality $a_i(\vartheta) \geq \varphi(i)$ holds only for finitely many $i\in \N$.
\end{itemize}
\end{thm}

\begin{defn}[Badly Approximable]\label{defn:badly-approx}
Let \(m\in \mathbb{N}\) and \(\omega\in \mathbb{R}^m\). Let $\phi:(0,+\infty)\to (0,1]$ be a non-increasing real-valued function. We say that \(\alpha\) is a
\(\phi\)-badly approximable vector if
\begin{equation*}
    \|\xi \cdot \omega\|_\Z
    \geq
    \frac{\phi(|\xi|_\infty)}{|\xi|_\infty^m}
    \qquad
    \text{for all}\; \xi\in \mathbb{Z}^m\setminus\{0\},
\end{equation*}
where $\Vert \xi\cdot \omega\Vert_{\mathbb{Z}} := \min_{\xi_0\in \mathbb{Z}}|\xi_0 + \xi\cdot \omega|$.
If \(\phi\) can be chosen to be constant, we say that \(\omega\) is badly approximable.
\end{defn}

\begin{thm}[{\cite[Theorem 1.2]{fregoli_sums_2019}}] \label{thm:freg2019thm12}
Let \(\theta \in \mathbb{R}^m\) be a
\(\phi\)-badly approximable vector. Then for all \(Q\geq 1\), 
\begin{equation}\label{eq:freg2019thm12}
    \sum_{ \substack{q \in \mathbb{Z}^m\setminus\{0\} \\ |q|_\infty\leq Q }}
    \frac{1}{\Vert q\cdot \theta\Vert _\Z}
    \leq 
    C(m) 
    \left( 
        Q^m\log Q + \frac{Q^m}{\phi(Q)}
    \right). 
\end{equation}
\end{thm}

The next lemma relates the Diophantine index \eqref{eq:diophan} to the discrete analog \cite{KleinLiuMelo2021}.

\begin{lem}[Discrete Diophantine Estimate for the Normalized Frequency]\label{lem:disandcont} Let $\omega\in \mathcal{D}(\sigma, C_\omega, n)$ be such that $\omega_n > 0$ and $\omega_n = \max\{|\omega_i|:1\leq i\leq n\}$. We define
\begin{equation*}
	\theta = \omega_n^{-1}(\omega_1,\ldots, \omega_{n-1})\qquad\text{so that}\qquad \omega = \omega_n (\theta, 1), \qquad |\theta|_\infty \leq 1.
\end{equation*}
Then
\begin{equation}\label{eq:DiscreteDiophantine}
	\Vert q\cdot \theta\Vert_\Z 
	\geq \frac{C_{n,\omega}}{|q|_\infty^\sigma} 
	\geq \frac{C_{n,\omega}}{|q|^\sigma}, 
	\qquad q\in \Z^{n-1}\backslash \{0\}, \qquad C_{n,\omega} = \frac{ C_\omega}{n^\sigma|\omega|_\infty }. 
\end{equation}
\end{lem}
\begin{proof}
For every $q\in \Z^{n-1}\backslash \{0\}$, there exists $p \in \Z$ such that 
\begin{equation}\label{eq:onsss}
    |q\cdot \theta + p| =  \Vert q\cdot \theta\Vert_{\Z} = |(q,p)\cdot (\theta,1)| = \frac{|(q,p)\cdot\omega|}{|\omega_n|} \geq \frac{C_\omega}{|\omega_n|.|(q,p)|^\sigma}.
\end{equation}
We compute, thanks to $|\theta|_\infty\leq 1$ and $|q| \leq (n-1)|q|_\infty$, that
\begin{align*}
    |p + q\cdot \theta| \leq \frac{1}{2} \quad\Longrightarrow\quad 
    |p| \leq |q\cdot \theta| + \frac{1}{2} \leq (n-1)|q|_\infty + \frac{1}{2} \leq \left(n-\frac{1}{2}\right)|q|_\infty.
\end{align*}
As a consequence, we have 
\begin{equation}\label{eq:estpp}
	|(q,p)| = (|p|^2 + |q|^2)^{1/2} \leq \left(\left(n-\frac{1}{2}\right)^2 + (n-1)\right)^{1/2}|q|_\infty \leq n|q|_\infty.
\end{equation}
Therefore \eqref{eq:DiscreteDiophantine} follows from \eqref{eq:onsss}.
\end{proof}

\begin{thm}[Dirichlet Theorem on Linear forms {\cite{davenport_dirichlets_1970}}]\label{thm:DirichletThm} \quad 
\begin{itemize}
\item[(i)] Let $\theta\in \R^{d}$ be nonresonant. For every $Q\geq 1$ there exist
$p\in\Z$ and $q\in\Z^d\backslash\{0\}$ such that
\begin{equation*}
	|q|_\infty\leq Q,
	\qquad
	0<\Vert q\cdot\theta\Vert_\Z =|p+q\cdot\theta|\leq Q^{-d}.
\end{equation*}
\item[(ii)] Let $\omega\in \R^{n}$ be nonresonant. For every $Q\geq 1$ there exists $\xi\in \Z^n \backslash \{0\}$ such that
\begin{equation*}
	|\xi| \leq nQ,
	\qquad
	0<|\xi\cdot\omega| \leq \frac{|\omega|_\infty}{Q^{n-1}}.
\end{equation*}
\end{itemize}
\end{thm}
\begin{proof} For (i), let $N=\lfloor Q\rfloor$. 
Let us consider $\mathcal{Q} = \{q=(q_1,\ldots, q_{d}) \in \{0, 1,\ldots, N\}\} \subset\Z^{n-1}$. 
There is a total of $\#\mathcal{Q} = (N+1)^{d}$ such vectors, hence there is a total of $(N+1)^{d}$ positive real numbers in $[0,1)$ of the form $\{ \Vert q\cdot \theta\Vert_\Z =  q\cdot \theta \;(\mathrm{mod}\;1): q\in \mathcal{Q}\}$.
Hence, there exist distinct vectors $m,m' \in \mathcal{Q}$ such that, for $q=m'-m\in \Z^d\backslash \{0\}$ then 
\begin{equation*}
	\Vert q\cdot \theta\Vert_\Z = 	\Vert (m'-m)\cdot \theta\Vert_\Z \leq \frac{1}{(N+1)^{d}} \leq \frac{1}{Q^d}, \qquad |q|_\infty \leq N \leq Q. 
\end{equation*}

For (ii), without loss of generality, we assume $\omega_n > 0$ and $\omega_n = \max\{|\omega_i|:1\leq i\leq n\}$. We define $	\theta = \omega_n^{-1}(\omega_1,\ldots, \omega_{n-1})$ so that $\omega = \omega_n (\theta, 1)$ and $|\theta|_\infty \leq 1$ as in Lemma \ref{lem:disandcont}. Let $d=n-1$, $p\in \Z$ and $q\in \Z^{n-1}$ from part (i), and $\xi = (q,p)\in \Z^n\backslash \{0\}$. Then by \eqref{eq:estpp} we have
\begin{equation*}
	|\xi| \leq n |q|_\infty \leq nQ, \qquad |\xi\cdot \omega| = \omega_n |p+\theta\cdot q| = \Vert q\cdot \theta\Vert_Z.  
\end{equation*}
We deduce the conclusion from (i).
\end{proof}

\begin{cor}\label{cor:se} If $\omega\in \mathcal{D}(n-1,C_\omega, n)$ then $\omega$ has exact Diophantine index $n-1$. Furthermore, there exists constants $C(n, \omega, C_\omega)$ and a sequence $k_j\in \Z^n$ such that $|k_j\to \infty|$ and
\begin{equation*}
	c_1 |k_{j-1}| \leq |k_j| \leq c_2|k_{j-1}|, 
	\qquad 
		0< k_{j+1}\cdot\omega < k_j\cdot \omega,
	\qquad\text{and}\qquad 
		\frac{C_\omega}{|k_j|^{n-1}} < k_j\cdot \omega \leq \frac{\widehat{C}_\omega}{|k_j|^{n-1}}
\end{equation*}
where $c_1,c_2 > 1$ depend on $n, |\omega|_\infty, C_\omega$, and $\widehat{C}_\omega = |\omega|_\infty n^{n-1}$. 
\end{cor}
\begin{proof} By the definition of $\omega\in \mathcal{D}(n-1,C_\omega,n)$, the lower bound $C_\omega |k|^{-\sigma}\le |k\cdot\omega|$ holds for every $k\in\Z^n$, while, by Theorem \ref{thm:DirichletThm}, for every $Q_j>1$ there exists $k_j\in\Z^n\backslash \{0\}$ satisfying the corresponding upper bound:
\begin{equation*}
	|k_j| \leq nQ_j, \qquad 
	 C_\omega |k_j|^{-(n-1)} < |k_j\cdot \omega| \leq |\omega|_\infty Q_{j}^{-(n-1)}  \leq \widehat{C}_\omega |k_j|^{-(n-1)}.	
\end{equation*}
In particular, this yields
\begin{equation*}
	c Q_j \leq	|k_j| \leq nQ_j, \qquad \text{where}\qquad c= \left(C_\omega |\omega|_\infty^{-1}\right)^{\frac{1}{n-1}}. 
\end{equation*}
By replacing $k_j$ with $-k_j$ if necessary, we may assume $k_j\cdot \omega > 0$. Let $Q_j = R^j$ for some $R$ to be chosen. Then
\begin{equation*}
	 c_1:= cn^{-1} R \leq \frac{|k_{j+1}|}{|k_j|} \leq  c_2:= nc^{-1}R, 
	 \qquad
	 \frac{\delta_{j+1}}{\delta_j} 
	\leq 
		\frac{|\omega|_\infty}{Q_{j+1}^{n-1}} \cdot \frac{|k_j|^{n-1}}{C_\omega}
	\leq 
		\frac{|\omega|_\infty}{R^{n-1}} \cdot \frac{n^{n-1}}{C_\omega} < 1
\end{equation*}
provided $R$ is sufficiently large. Moreover, we may choose $R$ so that $c_1,c_2>1$.
\end{proof}

\subsection{Some function spaces}\label{subSection:FunctionSpaces}
    \subsubsection*{H\"older continuous functions}
    We write \(C^0(\R^n)=C(\R^n)\) for the space of continuous functions, and
\(C^k(\R^n)\) for functions whose derivatives up to order \(k\)$(\in\N)$ are continuous.
For \(\alpha\in(0,1]\), the H\"older space \(C^{0,\alpha}(\R)\) consists of
continuous functions \(f\) such that
\begin{equation*}
	[f]_{C^{0,\alpha}(\T^n)}=\sup_{x\in \T^n} \sup_{h\in \R^n\backslash \{0\}} \frac{|f(x+h)-f(x)|}{|h|^\alpha} <\infty. 
\end{equation*}
The space $C^{0,\alpha}(\T^n,\R)$ is a Banach space with norm $\|f\|_{C^{0,\alpha}(\T^n)}
:= \|f\|_{L^\infty(\T^n)}
+ [f]_{C^{0,\alpha}(\T^n)}$. In particular, $C^{0,1}(\T^n,\R)$ is the set of all Lipschitz continuous functions. 
For $k\in \N$ and $\alpha\in (0,1]$, the space $C^{k,\alpha}(\T^n)$ is the space of $C^{k}(\T^n)$ function such that its $k$-derivative $D^{\beta} f\in C^{0,\alpha}(\T^n)$ for all multi-index $|\beta|=k$. Similarly, $C^{k,\alpha}(\T^n)$ is a Banach space endowed with the norm 
\begin{align*}
	\Vert f\Vert_{C^{k,\alpha}(\T^n)} = \sum_{j=0}^k \sum_{|\beta|=j} \Vert D^\beta f \Vert_{L^\infty(\T^n)} + \sum_{|\beta|=k} [D^\beta f]_{C^{0,\alpha}(\T^n)},
\end{align*}
where 
$\beta = (\beta_1,\ldots, \beta_n) \in (\mathbb{Z}_{\geq 0})^n$ is a multi-index, with $|\beta| = \beta_1+\ldots+\beta_n$ and $\beta! = \beta_1\beta_2\ldots \beta_n$. We use the convention that the endpoint case \(k=1,\alpha=0\) is included in
the Lipschitz class, since $C^1(\mathbb T^n)\subset C^{0,1}(\mathbb T^n)$.
Thus the case \(k=1,\alpha=0\) is contained in the Lipschitz case \(k=0,\alpha=1\).

    \subsubsection*{Sobolev space $H^s(\T^n)$ and $W^{k,p}(\T^n)$}
    Let $\mathcal{D}'(\mathbb{T}^n)$\label{def:SobolevHs} be the space of distributions on $\mathbb{T}^n$. 
    The Sobolev space $H^{s}(\mathbb{T}^n)$ is defined by 
    \begin{equation*}
        H^{s}(\mathbb{T}^n) = \left\lbrace f\in \mathcal{D}'(\mathbb{T}^n): (1+|\kappa|^2)^{\frac{s}{2}}|\widehat{f}(\kappa)| \in \ell^2(\mathbb{Z}^n) \right\rbrace,   
    \end{equation*}
    with the norm $\Vert f\Vert_{H^s(\T^n)} = \left(\sum_{\kappa\in \mathbb{Z}^n} (1+|\kappa|^2)^{\frac{s}{2}} |\widehat{f}(\kappa)|^2\right)^{1/2}$.

    For $k\in \mathbb{N}$, let 
    \begin{equation*}
        W^{k,p}(\mathbb{T}^n) = \left\lbrace f\in L^p(\mathbb{T}^n): \partial^\alpha f\in \ell^p(\mathbb{Z}^n)\;\text{for}\;|\alpha|\leq k \right\rbrace, 
    \end{equation*}
    with the norm $\Vert f\Vert_{W^{k,p}(\mathbb{T}^n)} = \sum_{|\alpha|\leq k} \Vert D^\alpha f\Vert_{L^p(\mathbb{T}^n)}$, where $\partial^{\alpha}f$ is the distributional derivative of $f$. It is known that $H^k(\T^n) = W^{k,2}(\T^n)$ if $k\in \N$.

\subsubsection*{The Besov space $W^s_{p,q}(\T^n)$} We construct a dyadic partition of unity as follows. Let $\chi \in C^\infty_c(\mathbb{R})$ be a smooth bump function such that, nonincreasing in $|s|$, with
\begin{equation*}
    \chi = 1\;\text{on}\;\left[-\frac{3}{4}, \frac{3}{4}\right], \qquad \mathrm{supp}(\chi) = \left[-\frac{4}{3}, \frac{4}{3}\right]. 
\end{equation*}
For $s>0$, we define
\begin{align}\label{eq:PaleyLittlewood-1}
    \varphi_{-1}(s) = \chi(s) ,
    \qquad 
    \varphi_0(s) = \chi\left(\frac{s}{2}\right) - \chi(s)  \geq 0,
    \qquad 
    \varphi_j(s) = \varphi_0(2^{-j}s) \geq 0, \quad j\geq 1. 
\end{align}
Then 
\begin{align}
    \mathrm{supp}\;\varphi_{-1} = \left[0, \tfrac{4}{3}\right], 		
    \quad 
    \mathrm{supp}\;\varphi_j = A_j := \left[\tfrac{3}{4}\cdot 2^j, \tfrac{8}{3}\cdot 2^j\right], \quad j\geq 0 \quad\text{and}\quad \sum_{j=-1}^\infty \varphi_j(s) = 1. \label{eq:annulus}
\end{align}
For convenience, we define for $j\geq 0$ the lattices
\begin{align}\label{eq:annulusZ}
	\mathcal{A}_j:= A_j\cap \Z^n. 
\end{align}
For $\xi \in \mathbb{R}^n$, we will, by a slight abuse of notation, write $\varphi_j(\xi) = \varphi_j(|\xi|)$. If $f\in L^1(\T^n)$, the Littlewood-Paley projections $\Delta_j f\in C^\infty(\T^n)$ are defined by 
\begin{align}\label{eq:PaleyLittlewood-2}
    \Delta_j f(x) 
    = \sum_{ \xi\in \mathcal{A}_j} \varphi_{j}(\xi) \widehat{f}(\xi)e^{2\pi i \xi \cdot x}, \qquad x\in \T^n. 
\end{align}
If $f\in L^1(\T^n)\cap L^2(\T^n)$, by Parseval's Theorem $\Vert f\Vert_{L^2(\T^n)}^2 = \sum_{\xi\in \Z^n} |\widehat{f}(\xi)|^2$, thus $f(x) = \sum_{k=-1}^\infty \Delta_kf(x)$ in $L^2(\T^n)$.   
Another basis result in \cite{stein_singular_1970, muscalu_classical_2013} is that if $f\in \mathcal{D}'(\T^n)$ then $\sum_{j=-1}^\infty \Delta_j f = f$ in $\mathcal{D}'(\T^n)$.

\begin{lem} Let $K_j:\T^n\to \R$ be the inverse Fourier transform (on $\T^n$) of $\varphi_j$, namely 
\begin{equation*}
    K_j(x) = \sum_{\xi\in \Z^n} \varphi_j(\xi) e^{2\pi i \xi \cdot x}, \qquad x\in \T^n,
\end{equation*}
where the sum is finite since $\mathrm{supp}(\varphi_j) \subset \left[\frac{3}{4}\cdot 2^{j}, \frac{8}{3}\cdot 2^j\right]$. 
    There exists $C(n)$ such that $\Vert K_j\Vert_{L^1(\T^n)} \leq C(n)$ for all $j\geq -1$. 
\end{lem}
\begin{proof} Let $ K_0(x) = \varphi_0^\vee(x) = \int_{\R^n}  \varphi_0(\xi) e^{2\pi i\xi \cdot x}\;d\xi$ for $x\in \T^n$,
then $\varphi_0^\vee \in\mathcal{S}(\R^n)$, the Schwartz class. 
There exists $C_{n+1}$ such that $|\varphi_0^\vee(x)|\leq C_{n+1}(1+|x|)^{-n-1}$ for $x\in \R^n$. Then, by the Poisson summation formula
\begin{equation}\label{eq:KjL1}
    K_j(x) = \sum_{\xi\in \Z^n} 2^{nj} \varphi_0^\vee(2^j(x+\xi))
\end{equation}
we obtain that $\Vert K_j\Vert_{L^1(\T^n)} \leq C_{n+1}$ uniformly in $j$.
\end{proof}

\subsection{Some embedding results}

\begin{lem}[Periodic Bernstein's inequality]\label{lem:BernsteinPaley} Fix $\varphi_j$ be the Paley-Littlewood decomposition as in \eqref{eq:PaleyLittlewood-1}. There exists a constant $C = C(n)$ such that,
for $1\leq p\leq r\leq \infty$, the decomposition $\Delta_j f$ defined in \eqref{eq:PaleyLittlewood-2} satisfies
\begin{equation*}
	    \|\Delta_j f\|_{L^r(\T^n)}
    \leq C \cdot 2^{j\left(\frac{n}{p}-\frac{n}{r}\right)}\|\Delta_j f\|_{L^p(\T^n)}, \qquad j \geq -1. 
\end{equation*}
\end{lem}
\begin{proof} The proof follows from Bernstein's inequality on $\R^n$, see, e.g.,
\cite[Lemma 2.1 and Proposition 2.20]{hajer_bahouri_fourier_nodate}.
\end{proof}

\begin{lem}[Besov Embedding] \label{lem:besovEmbeddingGeneral}
Let $s>0$. 
    \begin{itemize}
        \item[(a)] We have $B^s_{p_2,q}(\T^n) \hookrightarrow B^s_{p_1,q}(\T^n)$ if $p_1\leq p_2$. 

        \item[(b)]
        For $s>0$ and $p_1,p_2,q\in [1,\infty]$, we have 
        \begin{align}\label{eq:besov-basic}
            B^{s_1}_{p_1,q}(\T^n) \hookrightarrow B^{s_2}_{p_2,q}(\T^n) \qquad\text{if}\qquad s_2<s_1, s_2 - \frac{n}{p_2} \leq s_1-\frac{n}{p_1}, p_1\leq p_2. 
        \end{align}
        As a consequence, for $\sigma\geq 0$ and $1\leq p\leq 2$ we have
    \begin{equation}\label{eq:besov2}
         B^{\sigma+n}_{1,1}(\T^n)
        \hookrightarrow
    B^{\sigma+n/p}_{p,1}(\T^n) 
        \hookrightarrow
    B^{\sigma+n/2}_{2,1}(\T^n) 
    \hookrightarrow
    B^{\sigma}_{\infty,1}(\T^n) 
    \hookrightarrow
    B^{\sigma}_{\infty,\infty}(\T^n) .
    \end{equation}
        
        \item[(c)] For $1\leq p\leq \infty$, we have
        \begin{equation}\label{eq:embedd3}
        A^s(\T^n) \subset B^{s}_{\infty,1}(\T^n)\subset B^{s}_{p,1}(\T^n)\subset B^{s}_{1,1}(\T^n).
    \end{equation}
    with $\Vert f\Vert_{B^s_{1,1}(\T^n)}
        \leq 
        \Vert f\Vert_{B^s_{p,1}(\T^n)}
        \leq 
        \Vert f\Vert_{B^s_{\infty,1}(\T^n)} \leq \left(\frac{4}{3}\right)^s \Vert f\Vert_{A^s(\T^n)}$, where $A^s$ is the Wiener Algebra, defined in \eqref{eq:AsWienerAlgebra}.
    \end{itemize}
\end{lem}

\begin{proof}
{\it (a).} 
On a finite measure space, if $p_1\leq p_2$ then $\Vert \cdot \Vert_{L^{p_1}(\T^n)} \leq \Vert \cdot \Vert_{L^{p_2}(\T^n)}$. Therefore, the conclusion follows from the definition of $B^s_{p,q}(\T^n)$.

\medskip
        
\noindent
{\it (b).} 
The result \eqref{eq:besov-basic} follows from \cite{triebel_theory_1983}, which implies the first four embeddings in \eqref{eq:besov2}. The last embedding is obvious since $\Vert f\Vert_{B^\sigma_{\infty,\infty}(\T^n)} \leq \Vert f\Vert_{B^\sigma_{\infty,1}(\T^n)}$.

\medskip
        
\noindent
{\it (c).} 
For $\xi \in \Z^n $ then $\widehat{\Delta_k f}(\xi) = \varphi_k(\xi) \widehat{f}(\xi)$. 
We have $\Vert f\Vert_{B^s_{\infty,1}(\T^n)} 
    = 
    \sum_{k=-1}^\infty 2^{ks}\Vert \Delta_k f\Vert_{L^\infty(\T^n)}$. 
On the annulus $\mathcal{A}_k = \left[\frac{3}{4}\cdot 2^k, \frac{8}{3}\cdot 2^k\right] $ we have $\left(\frac{8}{3}\right)^{s}\cdot 2^{ks} \geq |\xi|^s \geq \left(\frac{3}{4}\right)^{s}\cdot 2^{ks}$. Therefore
\begin{align*}
\frac{1}{\pi C_\omega} \left(\frac{3}{4}\right)^{s} \sum_{|\xi| \in \mathcal{A}_k}2^{ks}|\widehat{f}(\xi)|
\leq 
    \frac{1}{\pi C_\omega}
    \sum_{|\xi| \in \mathcal{A}_k} |\xi|^s     \big|\widehat{f}(\xi)\big|    \leq  \frac{1}{\pi C_\omega} \left(\frac{8}{3}\right)^{s} \sum_{|\xi| \in \mathcal{A}_k}2^{ks}|\widehat{f}(\xi)|.
\end{align*}
From $\mathrm{supp}\varphi_k = \mathcal{A}_k$ we deduce that 
    \begin{align*}
        \left(\frac{3}{4}\right)^{s}\cdot 2^{ks} \varphi_k(\xi)
        \leq 
        |\xi|^{s}\varphi_k(\xi) \leq \left(\frac{8}{3}\right)^s\cdot 2^{ks} \varphi_k(\xi).
    \end{align*}
    Thus
    \begin{equation*}
        \left(\frac{3}{4}\right)^s\left(\sum_{k=-1}^\infty 2^{ks}\varphi_k(\xi)|\widehat{f}(\xi)|\right)
        \leq
        \underbrace{
            \sum_{k=-1}^\infty |\xi|^s \varphi_k(\xi) |\widehat{f}(\xi)| 
        }_{|\xi|^s|\widehat{f}(\xi)|}
        \leq \left(\frac{8}{3}\right)^s \left(\sum_{k=-1}^\infty 2^{ks}\varphi_k(\xi)|\widehat{f}(\xi)|\right). 
    \end{equation*}
We also have, for $k\geq 0$ that
\begin{equation*}
\Vert \Delta_k f(x)\Vert_{L^1(\T^n)} \leq             \Vert \Delta_k f(x)\Vert_{L^\infty(\T^n)} \leq
    \sum_{ \xi\in \mathcal{A}_k} \varphi_k(\xi) |\widehat{f}(\xi)|.
\end{equation*}
Therefor, for any $1\leq p\leq \infty$ then
\begin{align*}
    & \sum_{k=-1}^\infty 2^{ks} \Vert \Delta_k f(x)\Vert_{L^1(\T^n)}
    \leq 
    \sum_{k=-1}^\infty 2^{ks} \Vert \Delta_k f(x)\Vert_{L^p(\T^n)}
    \leq 
    \sum_{k=-1}^\infty 2^{ks} \Vert \Delta_k f(x)\Vert_{L^\infty(\T^n)} \\
    &\quad  \leq
    \sum_{k=-1}^\infty  \sum_{ \xi\in \mathcal{A}_k} 2^{ks}\varphi_k(\xi) |\widehat{f}(\xi)| 
    = \sum_{\xi\in \Z^n} 2^{ks}\varphi_k(\xi) |\widehat{f}(\xi)|
    \leq \left(\frac{4}{3}\right)^s\sum_{\xi\in \Z^n} |\xi|^s |\widehat{f}(\xi)|. 
\end{align*}
In other words, we have $\Vert f\Vert_{B^s_{1,1}(\T^n)}
        \leq 
        \Vert f\Vert_{B^s_{p,1}(\T^n)}
        \leq 
        \Vert f\Vert_{B^s_{\infty,1}(\T^n)} \leq \left(\frac{4}{3}\right)^s \Vert f\Vert_{A^s(\T^n)}$.         
\end{proof}

The following Lemma gives a characterization of H\"older space $C^{0,\alpha}(\T^n)$ and the Besov space $B^\alpha_{\infty,\infty}(\T^n)$ (see \cite{stein_singular_1970}, \cite[Lemma 8.6]{muscalu_classical_2013}). 

\begin{lem}[Correspondence between Besov and H\"older spaces] \label{lem:BesovHolderEmbed} \quad 
\begin{itemize}
    \item[(a)] For $s>0$, we have
    \begin{align}\label{eq:Csalpha}
        C^{[s],\alpha}(\T^n) = B^s_{\infty,\infty}(\T^n), \qquad s = [s]+\alpha, \alpha\in ( 0,1). 
    \end{align}
    In particular, $C^{0,\alpha}(\T^n) = B^\alpha_{\infty, \infty}(\T^n)$ for $\alpha<1$, and $C^{0,1}(\T^n) \subsetneq B^1_{\infty, \infty}(\T^n)$. Moreover, we have
    \begin{align}\label{eq:estBesovLessHolder}
            \Vert f\Vert_{B^\alpha_{\infty, \infty}(\T^n)} \leq C(n)\Vert f\Vert_{C^{0,\alpha}(\T^n)}. 
    \end{align}
    
    \item[(b)] For all $\varepsilon>0$ and $\alpha\in(0,1)$, we have $B^\alpha_{\infty, 1}(\mathbb{T}^n) \hookrightarrow B^{\alpha}_{\infty,\infty}(\T^n) = C^{0,\alpha}(\mathbb{T}^n) \hookrightarrow B^{\alpha-\varepsilon}_{\infty, 1}(\mathbb{T}^n)$.

    \item[(c)] 
    If $f\in B^\alpha_{\infty, \infty}(\T^n)$ for $\alpha>0$, or if $f\in B^0_{\infty, 1}(\T^n)$ then 
    $\sum_{j=-1}^N \Delta_j f(x) \to f(x) $ as $N\to \infty$
    absolutely and uniformly in $C(\T^n)$.

    \end{itemize}
\end{lem}

\begin{proof}
    We recall that $f\in B^\alpha_{\infty, \infty}(\T^n)$ if and only if
    \begin{equation*}
         \sup_{j\geq -1} 2^{j\alpha} \Vert \Delta_j f\Vert_{L^\infty(\T^n)} \leq C:=\Vert f\Vert_{B^\alpha_{\infty, \infty}(\T^n)}. 
    \end{equation*}

\noindent
{\it (a).} We refer to \cite{stein_singular_1970} for \eqref{eq:Csalpha}. 
The fact that for $\alpha\in (0,1)$ then $B^{\alpha}_{\infty,\infty}(\T^n) = C^{0,\alpha}(\mathbb{T}^n)$ follows from \cite[Lemma 8.6]{muscalu_classical_2013}. 
        \medskip

\noindent
{\it (b).} If $f\in C^{0,\alpha}(\T^n) = B^{\alpha}_{\infty, \infty}(\T^n)$ then for $j\in \{-1,0,1,\ldots\}$ we have $2^{j\alpha} \Vert \Delta_j f\Vert_{L^\infty(\T^n)} \leq \Vert f\Vert_{B^\alpha_{\infty, \infty}(\T^n)}$. 
        Therefore
        \begin{align*}
            \Vert f\Vert_{B^{\alpha-\varepsilon}_{\infty,1}(\T^n)} =  \sum_{j=-1}^\infty 2^{j(\alpha-\varepsilon)} \Vert  \Delta _jf\Vert_{L^\infty(\T^n)} 
            &= 
            \sum_{j=-1}^\infty 2^{-j\varepsilon} 2^{j\alpha} \Vert  \Delta _jf\Vert_{L^\infty(\T^n)} \\
            &\leq  \Vert f\Vert_{B^\alpha_{\infty,\infty}(\T^n)} \sum_{j=-1}^\infty 2^{-j\varepsilon}  = C(\varepsilon)  \Vert f\Vert_{B^\alpha_{\infty,\infty}(\T^n)}. 
        \end{align*}
        Thus $B^\alpha_{\infty,\infty}(\T^n) = C^{0,\alpha}(\T^n) \hookrightarrow B^{\alpha-\varepsilon}_{\infty,1}(\T^n)$ for any $\varepsilon>0$ small enough. 
\medskip

\noindent
{\it (c).} We have $\Vert \Delta_j f\Vert_{L^\infty(\T^n)} \leq 2^{-j\alpha} \Vert f\Vert_{B^\alpha_{\infty, \infty}(\T^n)}$ for $j\geq -1$, hence
            \begin{align*}
                \sum_{j=-1}^\infty \Vert \Delta_j f\Vert_{L^\infty(\T^n)} \leq 
            \Vert f\Vert_{B^\alpha_{\infty, \infty}(\T^n)} \sum_{j=-1}^\infty 2^{-j\alpha} < \infty.
            \end{align*}
Therefore, the Littlewood--Paley series $\sum_{j=-1}^\infty \Delta_j f$ converges absolutely and uniformly in \(C(\mathbb T^n)\). Since the Littlewood--Paley series converges to \(f\) in distributions, the
uniform limit must coincide with \(f\).
\end{proof}

\section{Rate of convergence of the Birkhoff average}\label{Sec:Rate}

\subsection{Proof of Theorem \ref{thm:BesovHolderCounting}: Rate of convergence for Besov space}
\label{subSec:thm11}

We state the following elementary key lemma, which will serve as a main ingredient in the proof of Theorem~\ref{thm:BesovHolderCounting}. 

\begin{lem}\label{lem:keyThm1}
Let $n\geq 2$, $\omega\in \mathcal{D}(\sigma, C_\omega,n)$ and let $\mathcal{A}_j$ be defined as in \eqref{eq:annulusZ}. 
\begin{itemize}
\item[(i)] If $p\in (1,\infty)$, we have
\begin{align*}
	\left(	\sum_{\xi\in \mathcal{A}_j} \frac{1}{|\xi\cdot \omega| ^p } \right)^{1/p} \leq  \frac{C_p}{C_\omega}\left(\frac{16}{3}\right)^\sigma  2^{j\sigma}, \qquad C_p=\left(2 \sum_{m=1}^\infty \frac{1}{m^p}\right)^{1/p}. 
\end{align*}
\item[(ii)] If $p=1$, we have for $j\in \N$ that
\begin{equation}\label{eq:reciprocalSump1}
	\sum_{\xi\in \mathcal{A}_j} \frac{1}{|\xi\cdot \omega|} 
	\leq  
    C(n,\sigma)\left((1+j)2^{j(n-1)} + \frac{1}{C_\omega}2^{j\sigma}\right).
\end{equation}
\end{itemize}
\end{lem}

\begin{proof}
  Recall that $\mathcal{A}_j := \left[\frac{3}{4}\cdot 2^j, \frac{8}{3}\cdot 2^j\right] \cap \Z^n$ for $ j\geq 0$. For $j\geq 0$ we define
\begin{equation*}
    S_j = \{\xi\cdot \omega: \xi\in \mathcal{A}_j\}    = S_j^+ \cup S_j^-, \qquad\text{where}\; 
    S^+_j = \{x\in S_j: x>0\}, S^-_j = \{x\in S_j: x<0\}. 
\end{equation*}
Let us define
\begin{equation}\label{eq:delta_j}
    \delta_j = C_\omega \left( \frac{3}{16}\right)^\sigma 2^{-j\sigma}. 
\end{equation}
For any $x\in S_j$, we have $x=\xi\cdot\omega$ for $\xi\in \mathcal{A}_j$, and thus $|x| = |\xi\cdot\omega| \geq C_\omega |\xi|^{-\sigma}  \geq C_\omega\left(\frac{3}{8}\right)^\sigma \cdot 2^{-j\sigma} \geq \delta_j$. 
Furthermore, for any distinct $x,y\in S_j$, by the same argument, we have $x=\xi_1\cdot \omega, y = \xi_2\cdot\omega$ for $\xi_1,\xi_2 \in \mathcal{A}_j$, which means $\xi_1\neq \xi_2$ and $0< |\xi_1-\xi_2|\leq \frac{16}{3}\cdot 2^{j}$. 
Therefore, using the Diophantine condition for $\xi_1-\xi_2 \in \Z^n\backslash \{0\}$ we obtain
\begin{equation*}
    |x-y| = |(\xi_1-\xi_2)\cdot \omega| \geq \frac{C_\omega}{|\xi_1-\xi_2|^\sigma} \geq  C_\omega\left(\frac{3}{16}\right)^\sigma\cdot 2^{-j\sigma} = \delta_j. 
\end{equation*}
We have $|x|\geq \delta_j, |x-y|\geq \delta_j$ for all $x, y\in S_j$ with $x\neq y$.
Let $\overline{K},\underline{K}$ be size of $S^+_j, S^-_j$, respectively. We can order elements of $S^+_j = \{x_k: k=1,\ldots, \overline{K}\}$ and $S^-_j = \{y_k: k=1,\ldots, \underline{K}\}$, we have
\begin{align*}
    x_{\overline{K}} > \ldots >  x_1 > 0 &\qquad\text{with}\qquad 
    x_m \geq +m\delta_j &&\quad \text{for all}\;x_m\in S^+_j \\
    y_{\underline{K}} <\ldots < y_1 < 0 &\qquad\text{with}\qquad 
    y_m\leq -m\delta_j &&\quad \text{for all}\;y_m\in S^-_j.    
\end{align*}
{\it (i).} For $p\in (1,\infty)$ we compute
\begin{align*}
	\left( \sum_{\xi\in \mathcal{A}_j} \frac{1}{|\xi\cdot\omega|^p} \right)^{1/p} 
	\leq 
		\left( \sum_{x\in S_j} \frac{1}{|x|^p} \right)^{1/p} 
	\leq 
		\left( \sum_{m=1}^\infty \frac{2}{(m\delta_j	)^p} \right)^{1/p}  = \frac{1}{\delta_j} \left(\sum_{m=1}^\infty \frac{2}{m^p}\right)^{1/p}. 	
\end{align*}
{\it (ii).} For $p=1$, the simple packing argument above yields (see \cite[Proof of Lemma 3.3, p.123]{KuipersNiederreiter1974})
\begin{equation}\label{eq:notsharp}
    \sum_{\xi\in \mathcal{A}_j} \frac{1}{|\xi\cdot \omega|}
    \leq
    C (1+j)2^{j\sigma},
    \qquad C = \frac{(2n+5)}{C_\omega}\left(\frac{16}{3}\right)^\sigma.
\end{equation}
However, the estimate \eqref{eq:notsharp} is not sharp. We use \cite[Theorem 1.2]{fregoli_sums_2019} to obtain a sharper estimate. Let $\omega=(\omega_1,\ldots, \omega_n)$. Without loss of generality we can assume $\omega_n = \max\{|\omega_i|:1\leq i\leq n\}$. We write
\begin{equation*}
    \theta = \left(\frac{\omega_1}{\omega_n}, \ldots, \frac{\omega_{n-1}}{\omega_n}\right) \in \R^{n-1} \qquad\text{so that}\qquad \omega = \omega_n(\theta,1). 
\end{equation*}
Then the Diophantine condition for $\omega\in \mathcal{D}(\sigma, C_\omega, n)$ transfers to a $\phi$-discrete Diophantine condition for $\theta\in \R^{n-1}$ (see Definition \ref{defn:badly-approx}) by
\begin{equation*}
    \Vert q\cdot \theta\Vert_{\Z} \geq \frac{\phi(|q|_\infty)}{|q|_\infty^{n-1}}, \qquad \text{where}\qquad 
        \phi(t) = \frac{C_\omega}{n |\omega_n| t^{\sigma - (n-1)}}. 
\end{equation*}
Indeed, for any $q = (q_1,\ldots, q_{n-1}) \in \Z^{n-1}\backslash\{0\}$, we choose $\xi_n \in \Z$ such that 
\begin{equation}\label{eq:rel}
    \Vert q\cdot \theta\Vert_{\Z} = |q\cdot \theta + \xi_n| =  |(q,\xi_n)\cdot (\theta,1)| = \frac{|\xi\cdot\omega|}{|\omega_n|} \geq \frac{C_\omega}{|\omega_n|.|\xi|^\sigma},
\end{equation}
where $\xi = (q, \xi_n)$. We compute, thanks to $|\theta_i|\leq 1$, that
\begin{align*}
    |\xi_n +q\cdot \theta| \leq \frac{1}{2} &\qquad\Longrightarrow\qquad 
    |\xi_n| \leq |q\cdot \theta| + \frac{1}{2} \leq (n-1)|q|_\infty + \frac{1}{2} \leq \left(n-\frac{1}{2}\right)|q|_\infty\\
    &\qquad\Longrightarrow\qquad 
	|\xi| = (|\xi_n|^2 + |q|_2^2)^{1/2} 
	\leq \left(\left(n-\frac{1}{2}\right)^2 + (n-1)\right)^{1/2}|q|_\infty \leq n|q|_\infty.
\end{align*}
Then we have 
\begin{equation*}
    \Vert q\cdot \theta\Vert_{\Z} \geq \frac{C_\omega}{|\omega_n|}\cdot \frac{1}{|\xi|^\sigma} \geq \frac{C_\omega}{n |\omega_n|}\cdot \frac{1}{|q|_\infty^\sigma} = 
    \frac{\phi(|q|_\infty)}{|q|_\infty^{n-1}}, \qquad 
    \phi(t) = \frac{C_\omega}{n |\omega_n| t^{\sigma - (n-1)}}. 
\end{equation*}
Since $\sigma \geq n-1$, the function $t\mapsto \phi(t)$ is non-decreasing on $(0,\infty)\to (0,\infty)$. By Definition \ref{defn:badly-approx}, $\theta$ is $\phi$-badly approximable. We apply \cite[Theorem 1.2]{fregoli_sums_2019}, namely \eqref{eq:freg2019thm12} with $m=n-1$, to deduce that
\begin{equation*}
    \sum_{ \substack{q \in \mathbb{Z}^{n-1}\setminus\{0\} \\ |q|_\infty\leq Q }}
    \frac{1}{\Vert q\cdot \theta\Vert _\Z}
    \leq 
    C(n) 
    \left( 
        Q^{n-1}\log Q + \frac{Q^{n-1}}{\phi(Q)}
    \right). 
\end{equation*}
If we write $\xi \in \Z^n$ as $\xi = (q,\xi_n) \in \Z^{n-1}\times \Z$, then (thanks to \eqref{eq:rel})
\begin{equation*}
    |\xi\cdot \omega| = |\omega_n||(q,\xi_n)\cdot(\theta,1)| = |\omega_n|\cdot |q\cdot\theta + \xi_n| \geq |\omega_n|\cdot\min_{k\in \Z} |q\cdot \theta + k| = |\omega_n|\cdot \Vert q\cdot \theta\Vert_\Z. 
\end{equation*}
Furthermore, if $\xi = (q,\xi_n)$ then $|\xi|\leq R$ implies that $|q|_\infty\leq R$. Therefore
\begin{align*}
    \sum_{\xi\in \mathcal{A}_j} \frac{1}{|\xi\cdot \omega|} 
    & \leq 
    \sum_{ \substack{ \xi\in \Z^n\backslash \{0\} \\ |\xi| \leq \frac{8}{3}\cdot 2^j }} \frac{1}{|\xi\cdot \omega|} 
    \leq \frac{1}{|\omega_n|} \sum_{ \substack{ \xi\in \Z^n\backslash \{0\} \\ |q|_\infty \leq \frac{8}{3}\cdot 2^j }}  \frac{1}{\Vert q\cdot\theta\Vert_\Z} \\
    & \leq 
    C(n)
    \left(
        Q^{n-1}\log Q+\frac{Q^{n-1}}{\phi(Q)}
    \right)
    \leq
    C(n,\sigma)\left(
        (1+j)\cdot 2^{j(n-1)}
            +
        \frac{1}{C_{\omega}}2^{j\sigma}
    \right).
\end{align*}
This is the desired conclusion \eqref{eq:reciprocalSump1}. 
\end{proof}

\begin{rmk}\label{rmk:simplep1} 
Estimate \eqref{eq:notsharp} can be used to obtain the following estimates (by the same procedure of the proof of Theorem \ref{thm:BesovHolderCounting}):
    \begin{equation*}
        f\in B^s_{1,q}(\T^n),\qquad 1\leq q\leq \infty, \qquad s>n,
    \end{equation*}
    then there exists $C= C(n, \sigma, s, C_\omega)$ such that 
    \begin{equation}\label{eq:subOptimal}
     \left|\frac{1}{T}\int_0^T f(x+\omega t)\;dt - \int_{\T^n} f(y)\;dy\right|  
     \leq C \Vert f\Vert_{B^s_{1,q}(\T^n)}
     \begin{cases}
     \begin{aligned}
         & (T^{-1}\log T)^{\frac{s-n}{\sigma-n}}  && n < s < \sigma, \\
         & T^{-1}(\log T)^{2-1/q}  &&  n < s = \sigma, \\
         & T^{-1}, && s>\max\{\sigma, n\}.
     \end{aligned}    
     \end{cases}
\end{equation}
This estimate \eqref{eq:subOptimal} is not as sharp as the conclusion of Theorem \ref{thm:BesovHolderCounting}: that is why \eqref{eq:reciprocalSump1} is crucial in getting a sharper estimate for \(f\in B^s_{1,q}(\T^n)\) for all $s\geq \max\{\sigma, n\}$. 
\end{rmk}

\begin{lem}\label{lem:EstBlock} Let $n\geq 2$, $p\in [1,\infty]$, $\omega\in \mathcal{D}(\sigma, C_\omega, n)$, and $f\in C(\T^n)$. 
\begin{itemize}
    \item[(i)] If $p\in (1,\infty]$, then for $j\geq -1$ we have
    \begin{equation}\label{eq:est}
        \left|\int_0^T \Delta_j f(\omega t) \;dt \right|
        \leq 
        C \cdot 2^{j\sigma}\Vert \Delta_j f\Vert_{L^p(\T^n)},
        \qquad \text{where}\qquad 
        C = \frac{C_p}{\pi C_\omega}\left(\frac{16}{3}\right)^\sigma.
    \end{equation}
    Here $C_p$ is the constant from Lemma \ref{lem:keyThm1}. If $p\geq 2$, the constant $C_p$ can be chosen as $C_2$ in \eqref{eq:est}.

    \item[(ii)] If $p=1$, then 
    \begin{equation}\label{eq:estp1}
        \left|\int_0^T \Delta_j f(\omega t) \;dt \right|
        \leq C(n,\sigma)
        \left((1+j)2^{j(n-1)} + \frac{1}{C_\omega}2^{j\sigma}\right)  \Vert \Delta_j f\Vert_{L^1(\T^n)}.
    \end{equation}
\end{itemize}
\end{lem}

\begin{proof}
From $\Delta_j f(x) = \sum_{\xi\in \mathcal{A}_j} \widehat{\Delta_j f}(\xi) e^{2\pi i \xi \cdot x}$ for $x\in \T^n$, we have
\begin{align*}
    \left| \int_0^T \Delta_j f(\omega t)\;dt \right| 
    = 
    \left| 
    \sum_{\xi\in \mathcal{A}_j} \widehat{\Delta_j f}(\xi) \frac{e^{2\pi i \xi \cdot \omega T}-1}{2\pi i \xi\cdot \omega}
    \right| 
    \leq 
    \frac{1}{\pi} \sum_{\xi\in \mathcal{A}_j} |\widehat{\Delta_j f}(\xi)| \cdot \frac{1}{|\xi\cdot\omega|}.
\end{align*}
Recall that $\delta_j = C_\omega \left( \frac{3}{16}\right)^\sigma 2^{-j\sigma}$ from \eqref{eq:delta_j}.
\medskip

\noindent 
\underline{\it Case 1. $p\in (1,2]$.}
Using H\"older's inequality with conjugate exponents $p\in (1,2]$ and $p'\in [2,\infty)$, along with the Hausdorff--Young inequality and Lemma \ref{lem:keyThm1}, we deduce that
\begin{equation}\label{eq:keyEstimateA}
\begin{aligned}
    & \frac{1}{\pi}\sum_{\xi\in \mathcal{A}_j} |\widehat{\Delta_j f}(\xi)|\cdot \frac{1}{|\xi\cdot \omega|} 
    \leq 
        \frac{1}{\pi } \left(\sum_{\xi\in \mathcal{A}_j} |\widehat{\Delta_j f}(\xi)|^{p'} \right)^{1/p'}  \left( \sum_{\xi\in \mathcal{A}_j} \frac{1}{|\xi\cdot\omega|^p} \right)^{1/p}\\
    &\qquad \qquad 
        \leq
        \frac{1}{\pi} \Vert \widehat{\Delta_j f} \Vert_{\ell^{p'}(\Z^n)} \frac{C_p}{\delta_j} \leq \frac{C_p}{\pi } \Vert \Delta_j f\Vert_{L^p(\T^n)}  \frac{1}{\delta_j}   
    = \frac{C_p}{\pi C_\omega} \left(\frac{16}{3}\right)^\sigma  2^{j\sigma} \Vert \Delta_j f\Vert_{L^p(\T^n)}. 
\end{aligned}
\end{equation}
\noindent 
\underline{\it Case 2. $p\in(2,\infty]$.} By \eqref{eq:keyEstimateA} with $p=2$ and the fact that $L^p(\T^n)\hookrightarrow L^2(\T^n)$ for $p>2$ we obtain 
\begin{equation*}
\begin{aligned}
	\frac{1}{\pi} \sum_{\xi\in \mathcal{A}_j} |\widehat{\Delta_j f}(\xi)|\cdot \frac{1}{|\xi\cdot \omega|} 
	& \leq 
	\frac{C_2}{\pi C_\omega}\left(\frac{16}{3}\right)^\sigma 2^{j\sigma} \Vert\Delta_j f\Vert_{L^2(\T^n)} 
	\leq 
		\frac{C_2}{\pi C_\omega}\left(\frac{16}{3}\right)^\sigma 2^{j\sigma} \Vert\Delta_j f\Vert_{L^p(\T^n)}. 
\end{aligned}
\end{equation*}
Since $\widehat f(0)=0$, the low-frequency block $\Delta_{-1}f$ contains only finitely many nonzero Fourier modes and can therefore be included in the preceding estimate by enlarging the constant. 
In summary, the estimate \eqref{eq:est} holds for all $p\in (1,\infty]$. 
\medskip

\noindent 
\underline{\it Case 3. $p=1$.} 
Using \eqref{eq:reciprocalSump1} we obtain
\begin{equation*}
\begin{aligned}
	\frac{1}{\pi} \sum_{\xi\in \mathcal{A}_j} |\widehat{\Delta_j f}(\xi)|\cdot \frac{1}{|\xi\cdot \omega|} 
	& \leq 
    \frac{1}{\pi} \Vert \widehat{\Delta_j f}\Vert_{L^\infty(\T^n)} \sum_{\xi\in \mathcal{A}_j} \frac{1}{|\xi\cdot \omega|} \\
    & \leq 
    \frac{1}{\pi} \Vert \Delta_j f\Vert_{L^1(\T^n)} \cdot C(n,\sigma)\left((1+j)2^{j(n-1)} + \frac{1}{C_\omega}2^{j\sigma}\right). 
\end{aligned}
\end{equation*}
Here we also treat the low-frequency block $\Delta_{-1}f$ similarly. 
\end{proof}

\begin{rmk} 
Let $f\in B^s_{p,q}(\T^n)$ with $1\leq p\leq \infty, 1\leq q\leq \infty$. 
The assumptions in Theorem \ref{thm:BesovHolderCounting} on the range of $s$, namely,
\begin{equation*}
    s>\frac{n}{p}\quad\text{when}\; q>1,
    \qquad\text{and}\qquad
    s\geq \frac{n}{p}\quad\text{when}\; q=1,
\end{equation*}
ensure that $f\in C(\T^n)$ and that its Littlewood--Paley decomposition $f=\sum_{j=-1}^{\infty}\Delta_j f$ converges absolutely and uniformly. 
\begin{itemize}
    \item For \(1\leq q\leq\infty\) and \(s>n/p\), Bernstein's inequality (Lemma \ref{lem:BernsteinPaley}) and H\"older's inequality yield 
    $B^s_{p,q}(\T^n) \hookrightarrow B^0_{\infty,1}(\T^n) \hookrightarrow C(\T^n)$. 

    \item  For the endpoint case $q=1$, $s\geq n/p$ is enough to guarantee continuity of $f$, since $B^s_{p,1}(\T^n) \hookrightarrow B^{s-n/p}_{\infty,1}(\T^n) \hookrightarrow B^0_{\infty,1}(\T^n) \hookrightarrow C(\T^n)$, including the endpoint case $p=1$.
\end{itemize}
\end{rmk}

We now give the proof of Theorem \ref{thm:BesovHolderCounting}. 

\begin{proof}[Proof of Theorem \ref{thm:BesovHolderCounting}]
We assume $x=0$ and $\int_{\T^n} f(y)\,dy=0$. 
Fix \(N\in \mathbb{N}\). We split the sum of the blocks in $f=\sum_{j=-1}^\infty \Delta_j f$ into the low-frequency part \(j\leq N\) and the high-frequency part \(j>N\). Apparently, 
\begin{equation*}
 \left|\frac{1}{T}\int_0^T f(x+\omega t)\;dt - \int_{\T^n} f(y)\;dy\right|  
     \leq \sum_{j=-1}^\infty \left|\int_0^T \Delta_j f(\omega t)\;dt\right|,
\end{equation*}
so it suffices to get an upper bound for the right-hand side. 
We first treat the case $p>1$.

\medskip 

\noindent 
\underline{\it The low-frequency part $j\leq N$.}
The sum of blocks over $-1\leq j\leq N$ can be estimated by Lemma \ref{lem:EstBlock},
\begin{equation}\label{eq:continuousestNNN}
\begin{aligned}
    \sum_{j=-1}^N \left|\int_0^T \Delta_j f(\omega t)\;dt\right|
    &\leq 
    C \sum_{j=-1}^N 2^{j\sigma} \Vert \Delta_j f\Vert_{L^p} 
    = 
	C \sum_{j=-1}^N 2^{j(\sigma - s)} \cdot 2^{js} \Vert \Delta_j f\Vert_{L^p} \\
    &\leq C\cdot  D_N \left(\sum_{j=-1}^N \left( 2^{j s}\Vert \Delta_j f\Vert_{L^p} \right) ^q\right)^{1/q} \leq C \cdot D_N \cdot \Vert f\Vert_{B^s_{p,q}}
\end{aligned}
\end{equation}
thanks to H\"older inequality with $\frac{1}{q'} + \frac{1}{q} =1$, where $C=C(p,C_\omega, \sigma)$ as given in \eqref{eq:est}, and  
\begin{align}
    & D_N = \left(\sum_{j=-1}^N 2^{j(\sigma-s)q'}\right)^{1/q'} = 
    C(\sigma-s, q')
    \begin{cases}
    \begin{aligned}
        &2^{N(\sigma - s)}, && s<\sigma \\
        &(N+2)^{1/q'}, && s=\sigma \\[1mm]
        &1, && s>\sigma . 
    \end{aligned}    
    \end{cases} & 1\leq q' < \infty  \label{eq:qprimeNotInfity} \\
    & D_N = \sup_{-1\leq j\leq N} 2^{j(\sigma-s)}
    = 
    \begin{cases}
    \begin{aligned}
        &2^{N(\sigma - s)}, && s<\sigma \\[1mm]
        &1,                 && s=\sigma \\[1mm]
        &2^{s-\sigma},      && s>\sigma . 
    \end{aligned}    
    \end{cases} & q' = \infty  \label{eq:qprimeEqualInfity}. 
\end{align}

\medskip 

\noindent 
\underline{\it The high-frequency part $j>N$.}
The sum of blocks over $j > N$ is estimated by Lemma \ref{lem:BernsteinPaley} (Bernstein inequality)
\begin{align}\label{eq:continuousestNNNBigger}
    \sum_{j>N} \left|\int_0^T \Delta_j f(\omega t)\;dt\right|
    &\leq 
    T \sum_{j>N} \Vert \Delta_j f\Vert_{L^\infty}
    \leq 
    C\cdot T \sum_{j>N}  2^{j\frac{n}{p}} \Vert \Delta_j f\Vert_{L^p},
\end{align}
where $C= C(n)$ is the constant from Bernstein Lemma \ref{lem:BernsteinPaley}. By H\"older inequality with $\frac{1}{q'} + \frac{1}{q} = 1$ we have
\begin{equation*}
    \sum_{j>N} 2^{j\frac{n}{p}} \Vert \Delta_j f\Vert_{L^p} 
    \leq 
    E_N
    \left( \sum_{j>N} \left( 2^{js}\Vert \Delta_j f\Vert_{L^p} \right)^{q} \right)^{1/q} 
    \leq E_N \Vert f\Vert_{B^s_{p,q}}
\end{equation*}
where 
\begin{align}
    & E_N = \left(\sum_{j>N} 
    2^{j\left(\frac{n}{p} - s\right)q'}
    \right)^{1/q'}  
    = C\left(s-\tfrac{n}{p}, q'\right) 
    \begin{cases}
        \begin{aligned}
            & 2^{-N(s-\tfrac{n}{p})} & s>\tfrac{n}{p}, \\[1mm]
            & +\infty & s\leq \tfrac{n}{p}, 
        \end{aligned}
    \end{cases}
    & 1\leq q' < \infty  \label{eq:qprimeNotInfityNNNBigger} \\ 
    & E_N = \sup_{j>N} 2^{-j(s-\tfrac{n}{p})}
    = \begin{cases}
        \begin{aligned}
            &2^{-(N+1)(s-\tfrac{n}{p})},  & s>\tfrac{n}{p}, \\[1mm]
            &1, & s=\tfrac{n}{p}, \\[1mm]
            &+\infty, & s<\tfrac{n}{p}. 
        \end{aligned}
    \end{cases}
    & q'= \infty  \label{eq:qprimeEqualInfityNNNBigger} . 
\end{align}
From \eqref{eq:continuousestNNN} and \eqref{eq:continuousestNNNBigger}, we systematically deduce the following:
\begin{align*}
    \left|\frac{1}{T}\int_0^T f(\omega t)\;dt \right|
    \leq 
    \underbrace{\left( C\cdot \frac{D_N}{T} +  C(n)\cdot E_N \right)}_{\mathcal{R}} \Vert f\Vert_{B^s_{p,q}}.
\end{align*}
\begin{itemize}
    \item For $1\leq q'<\infty$ and $s>\frac{n}{p}$, we have 
    \begin{equation*}
        \mathcal{R}   \leq 
        C
        \begin{cases}
            \begin{aligned}
                & T^{-1}\cdot 2^{N(\sigma-s)} &+  &\quad 2^{-N(s-\tfrac{n}{p})} &&\qquad s<\sigma, \\
                & T^{-1}\cdot (N+2)^{1/q'} &+ &\quad 2^{-N(s-\tfrac{n}{p})} &&\qquad s=\sigma, \\
                & T^{-1} &+ &\quad 2^{-N(s-\tfrac{n}{p})} &&\qquad s>\sigma,
            \end{aligned}
        \end{cases}
    \end{equation*}
    where 
    \begin{equation*}
        C = C(p,C_\omega, \sigma, n)\cdot \Big(C(\sigma-s,q') + C(s-\tfrac{n}{p},q')\Big).
    \end{equation*}
    By choosing \(N\) accordingly, as follows, we deduce that:
    \begin{equation*}
        \begin{aligned}
            &N\sim \frac{\log_2T}{\sigma - n/p} &&\qquad \text{if}\qquad \frac{n}{p}<s<\sigma 
            &&\qquad \text{then}\qquad \mathcal{R} \leq C T^{-\frac{s-n/p}{\sigma-n/p}}, \\ 
            & N\sim \frac{\log_2T}{\sigma - n/p} &&\qquad \text{if}\qquad \frac{n}{p}<s=\sigma
            &&\qquad \text{then}\qquad \mathcal{R} \leq CT^{-1}(\log T)^{1/q'}, \\ 
            &N\geq \frac{\log_2T}{\sigma - n/p} &&\qquad \text{if}\qquad \max\{\tfrac{n}{p},\sigma\} < s
            &&\qquad \text{then}\qquad \mathcal{R} \leq CT^{-1}. 
        \end{aligned}
    \end{equation*}

    \item For $q'=\infty$ and $s>\frac{n}{p}$, we have
    \begin{equation*}
        \mathcal{R}   \leq 
        C
        \begin{cases}
            \begin{aligned}
                & T^{-1}\cdot 2^{N(\sigma-s)} &+ &\quad 2^{-(N+1)(s-\tfrac{n}{p})} &&\qquad \tfrac{n}{p}<s<\sigma, \\
                & T^{-1} &+ &\quad 2^{-(N+1)(s-\tfrac{n}{p})} &&\qquad \tfrac{n}{p}<s=\sigma, \\
                & T^{-1} \cdot 2^{s-\sigma} &+ &\quad 2^{-(N+1)(s-\tfrac{n}{p})} &&\qquad \max\{\tfrac{n}{p}, \sigma\} < s.
            \end{aligned}
        \end{cases}
    \end{equation*}
    By choosing \(N\) accordingly, as follows, we deduce that:
    \begin{equation*}
        \begin{aligned}
            &N\sim \frac{\log_2T}{\sigma - n/p} &&\qquad \text{if}\qquad \frac{n}{p}<s<\sigma 
            &&\qquad \text{then}\qquad \mathcal{R} \leq C T^{-\frac{s-n/p}{\sigma-n/p}}, \\ 
            & N\geq\frac{\log_2T}{\sigma - n/p}-1 &&\qquad \text{if}\qquad \frac{n}{p}<s=\sigma
            &&\qquad \text{then}\qquad \mathcal{R} \leq CT^{-1}, \\ 
            &N\to \infty &&\qquad \text{if}\qquad \max\{\tfrac{n}{p},\sigma\} < s
            &&\qquad \text{then}\qquad \mathcal{R} \leq CT^{-1}. 
        \end{aligned}
    \end{equation*}
\end{itemize}
We observe that the case $q'=\infty$ can be absorbed into the earlier case $1\leq q'<\infty$.
We deduce the conclusion \eqref{eq:BesovAllPQa}--\eqref{eq:BesovAllPQc}. 
\medskip 

We now treat the case $p=1$. For $N\geq1$, \eqref{eq:estp1} yields
\begin{equation}\label{eq:endpointcasep1}
\begin{aligned}
    \frac{1}{T}\sum_{j\leq N}
    \left|
        \int_0^T\Delta_jf(\omega t)\,dt
    \right|
    &\leq
        \frac{C(n,\sigma)}{T}
        \sum_{j\leq N}
        \left(
            (1+j)2^{-j(s-n+1)}
            +
            \frac{1}{C_\omega}2^{j(\sigma-s)}
        \right)\cdot 2^{js}\|\Delta_jf\|_{L^1(\T^n)} \\
    &\leq 
    \frac{C(n,\sigma)}{C_\omega}  \left(\frac{F_N + D_N}{T} \right)\Vert f\Vert_{B^s_{1,q}(\T^n)}
\end{aligned}
\end{equation}
by using H\"older's inequality with $\frac{1}{q'}+\frac{1}{q} = 1$, where $D_N$ is given in \eqref{eq:qprimeNotInfity}--\eqref{eq:qprimeEqualInfity}, and
\begin{align}
    F_N &= \left( \sum_{j\leq N} \left( (1+j)2^{-j(s-n+1)} \right)^{q'} \right)^{1/q'} \leq 4 & q'\in [1,\infty), \label{eq:Fa}\\
    F_N &= \sup_{j\leq N}  (1+j)2^{-j(s-n+1)} = 1  & q'=\infty,\label{eq:Fb}
\end{align}
thanks to $s\geq n$. We remark that \eqref{eq:Fa}--\eqref{eq:Fb} are true even if $s=n$. Combining with the estimates for $D_N$ from \eqref{eq:qprimeNotInfity}--\eqref{eq:qprimeEqualInfity}, we obtain the same conclusion \eqref{eq:BesovAllPQa}--\eqref{eq:BesovAllPQc} for $p=1$. 

\medskip
Finally, we treat the endpoint case \(q=1\). If \(p\in(1,\infty]\), we use \eqref{eq:continuousestNNN}; if \(p=1\), we instead use \eqref{eq:endpointcasep1}. In both cases, we may let \(N\to\infty\) directly. Indeed, by \eqref{eq:qprimeEqualInfity}, we have \(D_N\) is finite uniformly in \(N\) whenever \(s\geq \sigma\), which yields \eqref{eq:BesovALemmaP1}. The condition \(s\geq \frac{n}{p}\) for \(f\in B^s_{p,1}(\T^n)\) ensures that \(f\) is continuous.
\end{proof}

\begin{proof}[Proof of corollary \ref{cor:Holder}] The proof of \eqref{eq:corHolder} is straight forward from Theorem \ref{thm:BesovHolderCounting}, thanks to the fact that $C^{[s],\alpha}(\T^n) = B^s_{\infty,\infty}(\T^n)$ where $s = [s]+\alpha, \alpha\in ( 0,1)$. For $\alpha=1$, we know that $C^{k,1}(\T^n)\hookrightarrow B_{\infty,\infty}^{k+1}(\T^n)$ for all $k\in\Z$. Moreover, for any $f\in C^{k,1}(\T^n)$, there holds $\|f\|_{B^{k+1}_{\infty,\infty}(\T^n)}\leq C\|f\|_{C^{k,1}(\T^n)}$ for some constant $C:=C(n,f)>0$. So we are able to enlarge the term $\|f\|_{B^{s}_{\infty,q}(\T^n)}$ of Theorem \ref{thm:BesovHolderCounting} to $\|f\|_{C^{k,1}(\T^n)}$ for the special case $s=k+\alpha$ and $\alpha=1$ (as an implicit requirement, $q=\infty$).
\end{proof}

\subsection{Proof of Theorem \ref{thm2:examples}: Examples for the sharp rates}


We construct examples for function $f\in B^s_{\infty,\infty}(\T^n)$ for specific $\omega\in \mathcal{D}(\sigma, C_\omega,n)$ that realizes the optimal rate $T^{-s/\sigma}$ and $T^{-1}\log(T)$. The usage of $\omega$ with exact Diophantine index as in \eqref{eq:exactDio} plays a key role in the constructions. We note that such constructions have been used for other purposes; see \cite{moshchevitin_nonrecursiveness_1991, moshchevitin_example_2023} and the references therein.

\begin{lem}\label{lem:verificationBs} 
Let $c>0$ and $\{k_j\}_{j\geq 1} \subset \Z^n\backslash \{0\}$ be a sequence such that $|k_{j+1}| \geq c|k_j|$. For $s>0$, we have $f\in B^s_{\infty,\infty}(\mathbb T^n)$ where $f$ is given by
\begin{equation*}
    f(x)=\sum_{j=1}^{\infty}|k_j|^{-s}\cos(2\pi k_j\cdot x),
\qquad x\in\mathbb T^n.
\end{equation*}
\end{lem}

\begin{proof} Since \(|k_{j+1}|\geq |k_j|^2\) super-exponentially, we have $\sum_{j=1}^{\infty}|k_j|^{-s}<\infty$. Thus the series defining \(f\) converges absolutely and uniformly. 
For \(\ell \geq 0\), the Fourier support of \(\Delta_j \ell \) is contained in an annulus $\left\{\xi\in\mathbb Z^n:\frac34\,2^\ell \leq |\xi|\leq \frac83\,2^\ell\right\}$. Hence
\begin{equation*}
    \Delta_\ell f(x)
    =
    \sum_{k_j\in \mathcal{A}_\ell}
    |k_j|^{-s}\varphi_\ell(k_j)\cos(2\pi k_j\cdot x). 
\end{equation*}
By the condition $|k_{j+1}|\geq c|k_j|$, each dyadic annulus contains at most a uniformly bounded number of frequencies \(k_j\). Indeed, if \(k_a\) and \(k_b\) with \(b = a + j\) both satisfy $k_a,k_b\in \mathcal{A}_j$, then
\begin{equation*}
    |k_b|\geq c^j|k_a| \qquad\Longrightarrow\qquad 
    c^j \cdot \frac{3}{4}\cdot 2^\ell  \leq \frac{8}{3}\cdot 2^\ell 	
    \qquad\Longrightarrow\qquad 
    c^j \leq \frac{32}{9} 
    \qquad\Longrightarrow\qquad 
    j \leq \frac{\ln(4)}{\ln(c)}. 
\end{equation*}
Hence, $\#\{j:\varphi_\ell(k_j)\neq 0\}\leq \frac{\ln(4)}{\ln(c)}$.
Since $0\leq \varphi_j \leq 1$ for all $j\geq -1$ by  \eqref{eq:PaleyLittlewood-1}, we have
\begin{equation*}
    \Vert \Delta_\ell f \Vert _{L^\infty(\mathbb T^n)}
    \leq
    \sum_{ \{j \in \N:\varphi_\ell(k_j)\neq 0\} } |k_j|^{-s} \\
    \leq \frac{\ln(4)}{\ln(c)}\cdot 2^{-\ell s} 
    \qquad \Longrightarrow\qquad 
    2^{\ell s}\|\Delta_\ell f\|_{L^\infty(\mathbb T^n)} \leq \frac{\ln(4)}{\ln(c)}
\end{equation*}
uniformly for all \(\ell\geq 0\).
The low-frequency block \(\Delta_{-1}f\) is also bounded because \(f\in L^\infty(\mathbb T^n)\). Hence $f\in B^s_{\infty,\infty}(\mathbb T^n)$.
\end{proof}

\begin{prop}[Item (i) of Theorem \ref{thm2:examples}]\label{prop:Thm12Casesalpha}
Let $\omega\in \mathcal{D}(\sigma, C_\omega, n)$ be such that $\omega$ has an \emph{exact} Diophantine index $\sigma$, i.e., there is a sequence $k_j \in \Z^n\backslash \{0\}$ with $|k_j|\to \infty$ such that
\begin{equation}\label{eq:exactDiophantine}
    C_\omega|k_j|^{-\sigma}	\leq |k_j\cdot\omega| \leq \widehat{C}_\omega|k_j|^{-\sigma}.
\end{equation}
For every $0<s<\sigma$, there exists $f\in B^s_{\infty,\infty}(\T^n)$ and a sequence $T_N\to\infty$ as $N\to\infty$, and a constant $C=C(s,\sigma,C_\omega,\widehat{C}_\omega)$ such that 
\begin{align}\label{eq:exam-besovsalphawork}
\left|\frac{1}{T_N}\int_0^{T_N} f(\omega t)\,dt - \int_{\T^n} f(x)\;dx \right|
        \geq C\left(\frac{1}{T_N}\right)^{s/\sigma}.
\end{align}
\end{prop}

\begin{proof} 
Since $|k_j|\to \infty$, we extract a subsequence, calling again $k_j$, such that $|k_1|\geq 2$ and
\begin{align}\label{eq:sequencekj}
	|k_{j+1}| \geq c|k_j|,  \qquad j=1,2,\ldots,
\end{align}
where $c$ depends only on $C_\omega, \widehat{C}_\omega, s,\sigma$ to be chosen. 
We define
\begin{align}\label{eq:exampleseries}
	f(x) = \sum_{m=1}^\infty \frac{1}{|k_m|^s} \cos(2\pi k_m\cdot x), \qquad x\in \T^n.
\end{align}
It is clear that $\int_{\T^n} f(x)\;dx = 0$ and $f\in B^s_{\infty,\infty}(\T^n)$ by Lemma \ref{lem:verificationBs}. 
We construct the sequence $T_N\to \infty$ that verifies \eqref{eq:exam-besovsalphawork}. There are infinitely many $k_N$ such that 
\begin{equation}\label{eq:Tj}
    k_N\cdot \omega > 0, 
    \qquad 
    T_N = \frac{1}{8\, k_N\cdot \omega} \geq \frac{|k_N|^\sigma}{8\widehat{C}_\omega} \to\infty 
    \qquad\text{as}\qquad |k_N|\to \infty. 
\end{equation}
The idea of the construction is that:
\begin{align*}
	& I_2 = \frac{1}{T_N} \int_0^{T_N} \frac{\cos(2\pi k_N \cdot \omega t)}{|k_N|^s}  \;dt \geq \mathcal{O}(T_N^{-s/\sigma}) \\
	& I_1 = \sum_{m<N}\frac{1}{T_N} \int_0^{T_N} \frac{\cos(2\pi k_m \cdot \omega t)}{|k_m|^s}\;dt \leq \mathcal{O}(\kappa_c T_N^{-s/\sigma}) \\
	&I_3 = \sum_{m>N}\frac{1}{T_N} \int_0^{T_N} \frac{\cos(2\pi k_m\cdot \omega t)}{|k_m|^s}\;dt \leq \mathcal{O}(\kappa_c T_N^{-s/\sigma}),
\end{align*}
where $\kappa_c \to 0$ as $c\to \infty$. Then we obtain the desired result from
\begin{equation}\label{eq:fff}
	\left|\frac{1}{T_N}\int_0^{T_N}f(\omega t)\;dt\right| 
    =
    \left|I_1+I_2+I_3\right| 
    \geq 
    I_2 - |I_1| - |I_3|. 
\end{equation}
We verify these claims.
By the exact Diophantine condition \eqref{eq:exactDiophantine}, we have
\begin{align}\label{eq:estTij}
    C_\omega |k_N|^{-\sigma} \leq \frac{1}{8T_N} = |k_N\cdot \omega| \leq \widehat{C}_\omega|k_N|^{-\sigma} 
    \quad\Longrightarrow\quad 
    C_\omega ^{\frac{s}{\sigma}} |k_N|^{-s} \leq \left(\frac{1}{8T_N}\right)^{s/\sigma} \leq \widehat{C}_\omega^{\frac{s}{\sigma}} |k_N|^{-s}.
\end{align}
By the definition of $T_N$ in \eqref{eq:Tj} and the fact that $k_N\cdot \omega > 0$, we have $2\pi k_N\cdot \omega T_N = \frac{\pi}{4}$. Therefore
\begin{equation}\label{eq:I2}
\begin{aligned}
     & I_2 
     = \frac{1}{|k_N|^s} \frac{\sin(2\pi k_N\cdot\omega T_N)}{2\pi k_N\cdot\omega T_N} 
     = \frac{2\sqrt{2}}{\pi}	\frac{1}{|k_N|^s} \geq 
     \frac{2\sqrt{2}}{\pi} \left(\frac{1}{8 \widehat{C}_\omega}\right)^{s/\sigma} T_N^{-s/\sigma}.
\end{aligned}
\end{equation}
For $I_1$, we compute for each $m<N$ using the lower bound in \eqref{eq:exactDiophantine} that
\begin{align*}
	\frac{1}{|k_m|^s}\left| \frac{1}{T_N} \int_0^{T_N} \cos(2\pi k_m \cdot \omega t)\;dt\right| 
    =	\frac{1}{|k_m|^s} \left|\frac{\sin(2\pi k_m\cdot\omega T_N)}{2\pi k_m\cdot\omega T_N}\right| 
    \leq 
        \frac{|k_m|^{\sigma-s}}{2\pi C_\omega T_N}  .
\end{align*}
Since $|k_{m+1}|\geq c|k_m|$ with $c>1$, for $m<N$, $|k_m|\leq c^{-(N-m)}|k_N|$. Thus, since $\sigma-s>0$, 
\begin{align*}
    \sum_{m<j}|k_m|^{\sigma-s}
        &\leq |k_N|^{\sigma-s}
    \sum_{m=1}^{N-1}c^{-(N-m)(\sigma-s)}
        =|k_N|^{\sigma-s}
    \sum_{\ell=1}^{N-1}c^{-\ell(\sigma-s)}
        \leq
    \frac{1}{c^{\sigma-s}-1}|k_N|^{\sigma-s}.
\end{align*}
Therefore, by \eqref{eq:estTij} we have
\begin{align*}
    |I_1|
        \leq
        \frac{1}{2\pi C_\omega T_N}
        \sum_{m<N}|k_m|^{\sigma-s}
        &\leq
        \frac{4\widehat{C}_\omega}{\pi C_\omega(c^{\sigma-s}-1)}
|k_N|^{-s}
        \leq
            \frac{4\widehat{C}_\omega}{\pi C_\omega(c^{\sigma-s}-1)}
            \left(\frac{1}{8C_\omega}\right)^{s/\sigma}
            T_N^{-s/\sigma}.
\end{align*}
For $I_3$, also sing \eqref{eq:estTij} we have 
\begin{align*}
    	|I_3| 
        \leq 	\sum_{m>N} \frac{1}{|k_m|^s} 
        \leq 
        \frac{1}{c^s-1}|k_N|^{-s}
        \leq
    \frac{1}{c^s-1}\left(\frac{1}{8C_\omega}\right)^{s/\sigma}
        T_N^{-s/\sigma}.
\end{align*}
Therefore,
\begin{equation*}
    I_2-|I_1|-|I_3|
        \geq
    \left[
        \frac{2\sqrt{2}}{\pi}\left(\frac{1}{8\widehat{C}_\omega}\right)^{s/\sigma}
        -
        \frac{4\widehat{C}_\omega}
        {\pi C_\omega(c^{\sigma-s}-1)}\left(\frac{1}{8C_\omega}\right)^{s/\sigma}
        -
        \frac{1}{c^s-1}\left(\frac{1}{8C_\omega}\right)^{s/\sigma}
    \right]
    T_N^{-s/\sigma}.
\end{equation*}
As $\sigma-s>0$, we choose 
\begin{equation*}
	c \geq \max\left\{
	\left[
	1+4\sqrt{2}
	\left(\frac{\widehat C_\omega}{C_\omega}\right)^{1+s/\sigma}
	\right]^{\frac{1}{\sigma-s}},
	\;
	\left[
	1+\sqrt{2}\pi
	\left(\frac{\widehat C_\omega}{C_\omega}\right)^{s/\sigma}
	\right]^{\frac{1}{s}}
	\right\}		
\end{equation*}
so that
\begin{equation*}
    I_2-|I_1|-|I_3|
    \geq
	\frac{\sqrt{2}}{\pi}\left(\frac{1}{8\widehat{C}_\omega}\right)^{s/\sigma}
    T_N^{-s/\sigma}.
\end{equation*}
We obtain the conclusion. 
\end{proof}

\begin{prop}[Item (ii) of Theorem \ref{thm2:examples}]\label{prop:sharpRateND}
Let $n\geq 2$ and $\omega\in \mathcal{D}(n-1,C_\omega,n)$. 
There exist $f\in B^{n-1}_{\infty,\infty}(\T^n)$, a sequence $T_N\to\infty$, and a constant $C = C(\omega, C_\omega,n)>0$ such that
\begin{equation*}
    \left|\frac1{T_N}\int_0^{T_N}f(\omega t)\,dt-\int_{\T^n}f(x)\,dx\right|
    \geq
    C\cdot\frac{\log T_N}{T_N}. 
\end{equation*}
\end{prop}

\begin{proof} By Corollary \ref{cor:se}, there is a sequence $k_j \in \Z^n\backslash \{0\}$ with $|k_j|\to \infty$ such that
\begin{equation}\label{eq:sequencekkk}
    C_\omega |k_j|^{-(n-1)}	\leq \delta_j: = k_j\cdot\omega \leq \widehat{C}_\omega |k_j|^{-(n-1)},
    \qquad \delta_{j+1} < \delta_j, 
    \qquad c_1 |k_{j-1}| \leq |k_j| \leq c_2 |k_{j-1}|
\end{equation}
where $c_1,c_2 > 1$ depend only on $|\omega|_\infty, C_\omega$, $n$ and $\widehat{C}_\omega = |\omega|_\infty n^{n-1}$. 
We define
\begin{equation}\label{eq:fND}
	f(x)
	:=
	\sum_{j=1}^\infty \frac{1}{|k_j|^{n-1}} \sin(2\pi k_j\cdot x), \qquad x\in\T^n.
\end{equation}
Since \eqref{eq:sequencekkk} implies exponential growth of $|\xi_j|$, $\sum_{j=1}^\infty |\xi_j|^{-(n-1)}<\infty$ and the series in \eqref{eq:fND} converges absolutely and uniformly. Moreover, $\int_{\T^n}f(x)\,dx=0$. By a similar argument to Proposition \ref{prop:Thm12Casesalpha} (as in Lemma \ref{lem:verificationBs}), we have $f\in B^{n-1}_{\infty,\infty}(\T^n)$.
To obtain the lower bound along a sequence $T_N\to \infty$, set
\begin{equation*}
	\mathcal S(T) := \frac{1}{T}\int_0^T f(\omega t)\,dt = \frac{1}{2\pi T} \sum_{j=1}^\infty \frac{1-\cos(2\pi\delta_jT)} {|k_j|^{n-1}\delta_j}. 
\end{equation*}
The idea behind the construction is that the positivity of each term allows us to truncate the series at a suitable radius $R_N$, which is chosen via the mean value theorem so that $\mathcal{S}(T_N)\approx NT_N^{-1}$. 
The desired rate then follows from $N\gtrsim \log T_N$, as guaranteed by \eqref{eq:sequencekkk}.
By \eqref{eq:sequencekkk}, $C_\omega \leq  |k_j|^{n-1}\delta_j \leq \widehat{C}_\omega$, hence
\begin{equation}\label{eq:STlowerND}
	T\mathcal S(T)
	\geq
	\frac{1}{2\pi \widehat{C}_\omega}
	\sum_{j=1}^\infty \bigl(1-\cos(2\pi\delta_jT)\bigr).
\end{equation}
For $N\geq 1$, define
\begin{equation*}
	R_N:=\frac1{\pi\delta_N}.
\end{equation*}
Averaging \eqref{eq:STlowerND} over $(0,R_N)$ gives
\begin{align*}
	\frac{1}{R_N}\int_0^{R_N}T\mathcal S(T)\,dT
	&\geq
	\frac{1}{2\pi \widehat{C}_\omega}
	\sum_{j=1}^N
	\left(1- \frac{\sin(2\pi\delta_jR_N)} {2\pi\delta_jR_N} \right).
\end{align*}
Since $\delta_j\geq\delta_N$ for $j\leq N$,
\begin{equation*}
	2\pi\delta_jR_N
	\geq
	2\pi\delta_NR_N
	=2 
	\quad\Longrightarrow\quad 
	\left|
	\frac{\sin(2\pi\delta_jR_N)}
	{2\pi\delta_jR_N}
	\right|
	\leq \frac{1}{2}
	\quad\Longrightarrow\quad 
	\frac{1}{R_N}\int_0^{R_N}T\mathcal S(T)\,dT
	\geq
	\frac{N}{4\pi \widehat{C}_\omega}.
\end{equation*}
Hence there exists $T_N\in(0,R_N)$ such that
\begin{equation}\label{eq:TNlowerND}
	T_N\mathcal S(T_N) \geq \frac{N}{4\pi \widehat{C}_\omega} 
	\qquad\Longleftrightarrow\qquad 
	\mathcal S(T_N)
	\geq
	\frac{N}{4\pi \widehat{C}_\omega T_N}.
\end{equation}
We note that 
\begin{equation*}
	T_N\|f\|_{L^\infty(\T^n)} \geq
	\left|
		\int_0^{T_N}f(\omega t)\,dt
	\right|
	=
	T_N|\mathcal S(T_N)|
	\geq
	\frac{N}{4\pi \widehat{C}_\omega}
	\qquad\Longrightarrow\qquad 
	\lim_{N\to\infty} T_N = \infty. 
\end{equation*}
Using \eqref{eq:sequencekkk} we gave
\begin{equation*}
	T_N \leq R_N = \frac{1}{\pi\delta_N} 
	\leq 
		\frac{|k_N|^{n-1}}{\pi c_1}
	\leq 
		\frac{(c_2^N |k_0|)^{n-1}}{\pi c_1} 
		\quad\Longrightarrow\quad	\log(T_N) \leq (n-1) \left[ N\log(c_2) + \log\left(\frac{|k_0|}{\pi c_1}\right) \right].
\end{equation*}
Therefore
\begin{equation*}
	N 
	\geq 
		\frac{1}{(n-1)\log(c_2)}\log(T_N)  - \frac{1}{\log(c_2)} \log\left(\frac{|k_0|}{\pi c_1}\right) 
	\geq 
		\frac{\log(T_N)}{2(n-1)\log(c_2)} 
\end{equation*}
for all sufficiently large $N$, since $\lim_{N\to \infty} T_N = \infty$.
Combining this with \eqref{eq:TNlowerND} yields
\begin{equation*}
	\left| \frac{1}{T_N}\int_0^{T_N}f(\omega t)\,dt
		-
	\int_{\T^n}f(x)\,dx \right| 
	= 
	\mathcal S(T_N) 
		\geq
	\left(\frac{1}{8\pi(n-1)\log(c_2)\widehat{C}_\omega}\right) \frac{\log T_N}{T_N}, 
\end{equation*}
thanks to $\int_{\T^n}f\,dx=0$. 
The proof is complete. 
\end{proof}

\begin{rmk}\label{rmk:exceptional-frequencies}
We remark that, when \(n=2\), vectors \(\omega=(\omega_1,\omega_2)\in \mathcal{D}(1,C_\omega,2)\) correspond precisely to those for which \(\omega_1/\omega_2\) is a badly approximable irrational number; see \cite{cassels_1957,schmidt_diophantine_1980,hu_polynomial_2025}. Trigonometric series like \eqref{eq:fND} have been considered for different purposes by many works in the literature. 
Trigonometric series such as \eqref{eq:exampleseries} and \eqref{eq:fND} have been considered for various purposes in the literature (see \cite{moshchevitin_example_2023} and the references therein). To the best of our knowledge, Propositions \ref{prop:Thm12Casesalpha} and \ref{prop:sharpRateND} are the first to use such a trigonometric construction to obtain, in particular, the optimal convergence rate for Birkhoff averages in Besov spaces.
\end{rmk}

Next, for almost every $\omega\in\R^n$, we construct examples that realize the \(\mathcal{O}(T^{-1})\) rate. 
The construction below relies on the notion of a \emph{coboundary}, which is standard in the literature; see \cite{sorrentino_statistical_2026} and the references therein.

\begin{defn}[Rotational flow and co-boundary] \quad 
\begin{itemize}
\item[(i)] The rotational flow $\rho_\om^t:\T^n\rightarrow \T^n$ is defined by $\rho_\om^t(x)=x+t\omega \;({\rm mod}\ \Z^n)$
with $
\omega=(\omega_1,\dots,\omega_{n})\in\R^n$ non-resonant. 
\item[(ii)] Any function $f\in C(\T^n)$ could be called a {\bf co-boundary} with respect to $\rho_\om^t$ if there exists $\psi\in C^1(\T^n)$ satisfying the {\bf cohomological equation}
\begin{equation*}
		f(x)-\int_{\T^n}f\,dx	=	\nabla_\om \psi(x).
\end{equation*}
\end{itemize}
\end{defn}

\begin{lem}\label{lem:co-boundary}
	Let $f\in C(\T^n)$  admit a co-boundary $\psi\in C^1(\T^n)$. 
	There exists a sequence $T_k\to\infty$ such that
	\begin{align*}
		\frac{\operatorname{osc}(\psi)}{2T_k}
	\leq
	\sup_{x\in\T^n}
	\left|
	\frac1{T_k}
	\int_0^{T_k}f(\rho^t_\omega(x))\,dt
	-
	\int_{\T^n}f(y)\,dy
	\right|
	\leq
	\frac{\operatorname{osc}(\psi)}{T_k},	
    \qquad 
    \operatorname{osc}(\psi)=\max_{\T^n}\psi-\min_{\T^n}\psi. 
	\end{align*}
\end{lem}
\begin{proof}
For any $T>0$ we have
\begin{align*}
	&\sup_{x\in\T^n} \left| \frac{1}{T}\int_0^T
	f(\rho_\om^t(x))
	\,dt-\int_{\T^n}f\,dx \right|
	=
	\sup_{x\in\T^n} \left| \frac{\psi(\rho_\om^t(x))-\psi(x)}{T} \right| \\
	& \qquad\qquad \qquad \qquad\qquad 
    \leq 
	\sup_{x\in\T^n}
	   \left|
	       \frac{1}{T} \int_0^T f(\Phi_t(x))\,dt - \int_{\T^n}f\,dx
	   \right| 
	\leq
	\frac{\operatorname{osc}(\psi)}{T},\quad\forall\, T>0.
\end{align*}
		Recall that $\omega$ is non-resonant, so for any $x\in\T^n$, the orbit $\rho_\om^t(x)$ is always dense in $\T^n$. Suppose $\psi(x_+)=\max_{\T^n}\psi$ and $
	\psi(x_-)=\min_{\T^n}\psi$, 
		then there exists a sequence of pairs $(S_k,T_k)\to(+\infty,+\infty)$ as $k\rightarrow+\infty$ such that	
		\begin{equation*}
			\rho_\om^{S_k}(x)\to x_+,\quad \rho_\om^{-T_k}(x)\to x_-,\quad \text{as }k\rightarrow +\infty.
		\end{equation*}
	By continuity of $\psi$,
	\begin{align*}
	\sup_{x\in\T^n}
	\left|
	\frac1{S_k+T_k}
	\int_0^{S_k+T_k} f(\rho_\om^t(x))\,dt
	-
	\int_{\T^n}f\,dx
	\right|
	\geq
	\frac{\operatorname{osc}(\psi)}{2T_k}
	\end{align*}
	which completes the proof.
\end{proof}

\begin{prop}[Item (iii) of Theorem \ref{thm2:examples}]\label{prop:Thm12ii}
For almost every $\omega\in\R^n$, there exists a function $f\in B_{\infty,\infty}^{k+\alpha}(\T^n)$ with $k+\alpha>0$ ($k\in\N,\ \alpha\in(0,1]$) and a sequence $T_j\to \infty$ as $j\to \infty$ such that
        \begin{align}\label{eq:besovAexampleLem}
            \left|
                \frac{1}{T_j}\int_0^{T_j} f(\omega t)\;dt
                -
                \int_{\T^n} f(x)\;dx
            \right|
            \geq C(\omega)\cdot \frac{1}{T_j}. 
        \end{align}
\end{prop}

\begin{proof}[Proof of Proposition \ref{prop:Thm12ii}]
We construct an example showing that the rate \(O(T^{-1})\) can occur for observables in \(B_{\infty,\infty}^{\alpha}(\mathbb T^n)\), \(0<\alpha \leq 1)\). We observe that the set of non-resonant $\omega\in \R^n$ has full measure, and if $\omega\in\R^n$ is non-resonant then $\omega_i\neq 0$ for all $1\leq i\leq n$.
Define
\begin{equation}\label{eq:cobound-psi}
	\psi(x)=\sum_{j=1}^\infty 2^{-j(1+\alpha)}\cos(2\pi\cdot 2^j x_1),
\qquad x=(x_1,\cdots,x_n)\in\T^n,
\end{equation}
and 
\begin{equation*}
	f(x_1)=\nabla_\om\psi(x)
=
-2\pi\omega_1\sum_{j=1}^\infty 2^{-j \alpha }\sin(2\pi\cdot 2^j x_1).
\end{equation*}
We show that $f\in B_{\infty,\infty}^{\alpha}(\T^n)$ if $0<\alpha\leq 1$. Indeed, for any $h\in\R$, choose $j_0\in\N$ such that 
\begin{equation*}
\frac{2}{2^{j_0+1}}\leq 	2\pi |h| < \frac{2}{2^{j_0}}. 
\end{equation*}
We have $-1-\log_2(\pi|h|)\leq j_0 < -\log_2(\pi |h|)$. 
Since $|\sin(a+b)-\sin (a)|\leq \min\{2,|b|\}$, we have
\begin{align*}
|f(x_1+h)-f(x_1)|
	\leq
2\pi |\omega_1|\sum_{j=1}^\infty 2^{-\alpha j}\min\{2,2\pi\cdot 2^j|h|\} = I_1+I_2,
\end{align*}
where 
\begin{align*}
	I_1 &= 2\pi |\omega_1|\sum_{j=1}^{j_0} 2^{-\alpha j}\cdot 2\pi\cdot 2^j|h| = 4\pi^2 |\omega_1|\cdot |h| \sum_{j=1}^{j_0} 2^{j(1-\alpha)}
	\leq 
	4\pi^2 |\omega_1|
	\begin{cases}
	C(\alpha) |h|^\alpha &\quad \alpha \in (0,1)\\
	|h|\cdot |\log_2(\pi |h|)| &\quad \alpha =1,
	\end{cases} \\
	I_2 &= 2\pi |\omega_1|\sum_{j=j_0+1}^{\infty} 2^{-\alpha j}\cdot 2 = 4\pi |\omega_1| = \frac{4\pi|\omega_1|}{1-2^{-\alpha}}\cdot |h|^{\alpha},	
\end{align*}
where $C(\alpha)  = \frac{2^{1-\alpha}}{2^{1-\alpha}-1}$ for $\alpha \in (0,1)$.
Overall, we have 
\begin{equation*}
\begin{aligned}
	& f\in C^{0,\alpha}(\T^n)=B_{\infty,\infty}^\alpha(\T^n) &&\qquad  \alpha\in (0,1), \\
	& f(x+h) - f(x)| \leq C|h|\cdot|\log(h)|\; \quad x,h\in \R^n &&\qquad  \alpha=1. 
\end{aligned}
\end{equation*}
As a consequence, $f$ admits a $C^1$ co-boundary for all $0<\alpha\leq 1$.
More precisely,
\begin{equation*}
	\psi\in
	\begin{cases}
	C^{1,\alpha}(\mathbb{T}^n), &\qquad 0<\alpha<1, \\
	B_{\infty,\infty}^2(\mathbb{T}^n)\ \text{with }D\psi=f\ \in B_{\infty,\infty}^1(\T^n), &\qquad 		\alpha=1.
\end{cases} 
\end{equation*}
Due to Lemma~\ref{lem:co-boundary}, there exists a sequence $T_k\to\infty$ such that
\begin{align*}
	\frac{\operatorname{osc}(\psi)}{2T_k}
\leq
\sup_{x\in\T^n}
\left|
\frac{1}{T_k}
\int_0^{T_k} f(\rho_\om^t(x))\,dt
-
\int_{\T^n}f(y)\,dy
\right|
\leq
\frac{\operatorname{osc}(\psi)}{T_k}
\end{align*}
so we get \eqref{eq:besovAexampleLem} for $f\in B_{\infty,\infty}^{\alpha}(\T^n)$ with $\alpha\in (0,1]$.

For $k\geq 1$, any function $f(x_1)\in B_{\infty,\infty}^{k+\alpha}(\T)$ leads to $f^{(k)}(x_1)\in B_{\infty,\infty}^{\alpha}(\T)$, vice versa. From the previous arguments, we can construct a function $\psi^{(k)}(x_1)\in B_{\infty,\infty}^{1+\alpha}(\T)$ in the form of \eqref{eq:cobound-psi} such that we can corresponding get $f^{(k)}(x_1)\in B_{\infty,\infty}^\alpha(\T)$ as the coboundary, such that $\int_{\T^n}f^{(k)} dx=0$ and $f^{(k)}(x)=\nabla_{\om}\psi^{(k)}(x)$ for all $x\in\T^n$. By $k-$times integral of $f^{(k)}(x_1)$, we can get a function $f(x_1)\in B_{\infty,\infty}^{k+\alpha}(\T)$ which is the coboundary of the $k-$times integral of $\psi^{(k)}(x_1)$. Once again applying Lemma \ref{lem:co-boundary} we get the result \eqref{eq:besovAexampleLem} for $f\in B_{\infty,\infty}^{k+\alpha}(\T^n)$ with $k\in\Z_+$ and $\alpha\in (0,1]$.
\end{proof}

\begin{prop}[Item (iv) of Theorem \ref{thm2:examples}]\label{prop:Thm12iii}
  For almost every  $\omega\in\R^n$, there exists $f\in C^{0,\alpha}(\T^n)$ with $\alpha\in (0,1)$ and a sequence $T_j\rightarrow \infty$ as $j\rightarrow+\infty$, such that 
\begin{equation}\label{eq:NearlyOptimal2DLem  }
    C_1(\om)\left(\frac{1}{T_j}\right)^\alpha \leq          \left|
            \frac{1}{T_j}\int_0^{T_j} f(\omega t)\;dt - \int_{\T^2} f(x)\;dx 
        \right| 
         \leq  C_2(\om,\eps)\left(\frac{1}{T_j}\right)^{\frac{\alpha}{n-1+\varepsilon}} 
\end{equation}
        for all $\eps>0$. Consequently, the subcritical rate $\mathcal{O}\big(T^{-\frac{k+\alpha}{\sigma}}\big)$ in \eqref{eq:corHolder} is nearly optimal for $n=2$ and $k=0$.
\end{prop}

The following lemma \ref{lem:fullmeasure-new} is used in the proof of Proposition \ref{prop:Thm12iii}. 

\begin{lem}\label{lem:fullmeasure-new}
	There exists a full measure subset $\Omega_0\subset(0,1)$
	such that for every irrational number $\vartheta\in\Omega$,
	there exists an increasing sequence of positive integers
	$\{m_k\}_{k\ge1}$ satisfying
	\begin{equation*}
			Q_{m_k+1}\geq c_0 m_k Q_{m_k},
	\qquad k\geq 1,
	\end{equation*}
	where $c_0>0$ is a uniform constant, and $\frac{P_k}{Q_k}$ is the Diophantine approximation of $\vartheta$. 
\end{lem}

\begin{proof}[Proof of Lemma \ref{lem:fullmeasure-new}]
	Let $c_0>0$ be any fixed constant. Applying the Borel--Bernstein Theorem (see Theorem \ref{thm:borelbernstein}) with $\varphi(j)=c_0 j$ for $j\in \N$, and using
	\begin{equation*}
		\sum_{j=1}^{\infty}\frac1{\phi(j)}
	=
	\frac1{c_0}\sum_{j=1}^{\infty}\frac1j
	=
	\infty,
	\end{equation*}
	we conclude that there exists a full-measure subset $	\Omega_0\subset(0,1)$
	such that for every irrational $\vartheta\in\Omega$, the corresponding continuous fraction satisfies $	a_j(\vartheta)\ge c_0 j$
	for infinitely many $j$.
	Let $\{n_k\}$ be the corresponding subsequence, namely 
\begin{equation*}
	a_{n_k} = a_{n_k}(\vartheta) \geq c_0 n_k \qquad\text{for all}\;k\in \N. 
\end{equation*}	
Let $m_k = n_k-1$, and without loss of generality $m_k\geq 1$.
As a consequence, by definition of $\frac{P_k}{Q_k}$ in \eqref{eq:DiophantinePQ} we have
	\begin{align*}
			Q_{m_k+1}
			=
			a_{m_k+1}Q_{m_k}+Q_{m_k-1} 
			\geq 
			a_{m_k+1}Q_{m_k} \geq c_0 (m_{k}+1)Q_{m_k} \geq c_0 m_k Q_{m_k}. 
	\end{align*}			
	This proves the lemma.
\end{proof}

\begin{lem}\label{lem:func} Let $\alpha\in(0,1)$ and $(P_j, Q_j)$ be the $j$-th Diophantione approximation of $\vartheta \in \Omega_0\subset (0,1)$ defined in Lemma \ref{lem:fullmeasure-new} for $j\geq 1$. We define
\begin{equation}\label{def:functionf-new}
	f(x)
	=
		\sum_{j=1}^{\infty}
	\frac{1}{Q_j^\alpha}
	\cos\left(
		2\pi (-P_j,  Q_j)\cdot (x_1,x_2)
	\right), 
	\qquad
	x = (x_1,x_2	)\in\mathbb T^2. 
\end{equation}
Then $f\in B_{\infty,\infty}^\alpha(\T^2)=C^{0,\alpha}(\mathbb T^2)$ for $\alpha\in(0,1)$.
\end{lem}

We omit the proof of this Lemma \ref{lem:func} as it is similar to Lemma \ref{lem:verificationBs}.

\begin{proof}[Proof of Proposition \ref{prop:Thm12iii}]
The argument for this target is essentially for $n=2$. 
	Without loss of generality, we can always assume $\omega = (1,\vartheta)$ with $\vartheta\in \Omega_0\subset (0,1)$, where $\Omega_0$ is the subset with full measure defined in Lemma \ref{lem:fullmeasure-new}. Then the function $f$ defined in \eqref{def:functionf-new} satisfies	
	\begin{equation*}
		\frac{1}{T}\int_0^T f(\omega t)\,dt
	=
	\frac{1}{T}
	\int_0^{T}
	f(t,\vartheta t)\,dt \qquad\text{for}\;\vartheta\in\R
	\backslash\Q. 
	\end{equation*}
Due to subsection \ref{subsection:Diophantine}, if we 	
denote by $(P_j,Q_j)$ for $j\geq 1$ the continued fraction expansion of $\vartheta$, then from \eqref{eq:Diophantineestimates} we have
\begin{equation}\label{eq:continued-fraction}
			\frac1{2Q_{j+1}}
			<
			|P_j - \vartheta Q_j|
			<
			\frac{1}{Q_{j+1}} 
			\qquad\text{for every}\; j\geq 1.
		\end{equation}
Furthermore, by \eqref{eq:DiophantinePQ} the sequence $\{Q_j\}$ grows at least exponentially:
		\begin{equation}\label{eq:q-growth}
			Q_{j+2\ell}\geq 2^\ell Q_j,
			\qquad \text{for all}\; j,\ell\geq 0.
		\end{equation}
Let $\vartheta\in\Omega$ as defined in \eqref{lem:fullmeasure-new}. By Lemma~\ref{lem:fullmeasure-new}, there exists a subsequence $\{m_k\}_{k\geq 1}$ such that
\begin{align*}
	Q_{m_k+1} \geq c_0 m_k Q_{m_k}.
\end{align*}
	To simplify notation, we write $m$ in place of $m_k$ and set $\delta_j = -P_j + Q_j \vartheta$. Then
	\begin{align*}
		(-P_j, Q_j)\cdot (t,\vartheta t) = (-P_j + Q_j\vartheta)t = \delta_j t. 
	\end{align*}		
	Since $\int_{\mathbb T^2}f=0$, we have
	\begin{align}\label{eq:Reest}
		\frac{1}{T}\int_0^T f(t,\vartheta t)\,dt
	=
	\mathrm{Re}
	\left(
	\sum_{j\geq 1}
		\frac{1}{Q_j^\alpha}
		\frac{e^{2\pi i\delta_j T}-1}
		{2\pi i\delta_jT}
		\right) = \mathrm{Re}\left(I_1 + I_2+I_3\right) \geq \mathrm{Re}(I_2) - |I_1| - |I_3|,
	\end{align}
	where we evaluate this expression at $T=Q_{m_k}$, and
	\begin{align*}
	I_1 = \sum_{j<m_k}\frac{1}{Q_j^\alpha}
		\frac{e^{2\pi i\delta_j T}-1}
		{2\pi i\delta_jT}, 
		\qquad
	I_2 = \frac{1}{Q_{m_k}^\alpha}
		\frac{e^{2\pi i Q_{m_k} \delta_j}-1}{2\pi i \delta_j Q_{m_k}}		,
		\qquad
	I_3 = \sum_{j>m_k}\frac{1}{Q_j^\alpha}
		\frac{e^{2\pi i\delta_j T}-1}
		{2\pi i\delta_jT}.
	\end{align*}
	\medskip
	
	\underline{{\it Step 1. Estimating $I_2$	. }}
	When $j=m_k$, \eqref{eq:continued-fraction} implies 
	\begin{align*}
	|\delta_{m_k}| < \frac{1}{Q_{m_k+1}}
	\qquad\Longrightarrow\qquad
	|Q_{m_k} \delta_{m_k}| < \frac{Q_{m_k}}{Q_{m_k+1}} 
		\leq \frac{1}{c_0m_k}. 
	\end{align*}
	Hence $Q_{m_k} \delta_{m_k} \to 0$ as $k\to \infty$, and 	therefore, at $T=Q_{m_k}$, for all sufficiently large $k$ we have 
	\begin{align}
		\mathrm{Re} \left(
		\frac{e^{2\pi i\delta_{m_k} Q_{m_k}}-1}
		{2\pi i\delta_{m_k} Q_{m_k}}
	\right) 
	= \frac{\sin(2\pi \delta_{m_k}Q_{m_k})}{2\pi \delta_{m_k}Q_{m_k}}
	\geq \frac{1}{2}
	\qquad\Longrightarrow\qquad 
	\mathrm{Re}(I_2) \geq \frac{1}{2 Q_{m_k}^\alpha}. 
	\label{eq:mainmode-final2}
	\end{align}			
	\medskip
	
		\underline{{\it Step 2. Estimating $I_1$	. }}
	For $j<m_k$, using $\delta_j = -P_j+Q_j\vartheta $ we have
	\begin{align*}
		Q_{m_k} \delta_j = -Q_{m_k}P_{m_k} +		 Q_{m_k} Q_j \vartheta 		
		\qquad\Longrightarrow\qquad 
		e^{2\pi i Q_{m_k} \delta_j} = e^{2\pi i Q_j Q_{m_k}\vartheta}. 
	\end{align*}
	We compute
	\begin{align*}
		\left|
			e^{2\pi i Q_{m_k} \delta_j} - 1		
		\right|
		&=
		\left|
			e^{2\pi i Q_{m_k} Q_j \vartheta} - e^{2\pi i P_{m_k} Q_j}
		\right|\\
		&= 
		\left|
			e^{2\pi i Q_j \cdot (Q_{m_k} \vartheta - P_{m_k})} - 1
		\right|
		\leq 2\pi |Q_j|\cdot |\delta_{m_k}| 
		\leq 2\pi \cdot 		\frac{Q_j}{Q_{m_k+1}}. 		
	\end{align*}
	Due to \eqref{eq:continued-fraction} we have $|\delta_j|\geq \frac{1}{2Q_{j+1}}$, hence
	\begin{align*}
		\left|
		\frac{e^{2\pi i Q_{m_k}\delta_j} - 1}{2\pi i Q_{m_k} \delta_j} 
		\right| \leq 2\cdot \frac{Q_j Q_{j+1}}{Q_{m_k}Q_{m_k+1}}
	\end{align*}

	Hence
	\begin{equation}\label{eq:anqm<}
		\left |
			\sum_{j=1}^{m_k-1} \frac{1}{Q_j^\alpha}\cdot 
		\frac{e^{2\pi i Q_{m_k}\delta_j} - 1}{2\pi i Q_{m_k} \delta_j} 
			\right| 
			\leq 			
			\sum_{j=1}^{m_k-1} \frac{2}{Q_j^\alpha}\cdot \frac{Q_jQ_{j+1}}{Q_{m_k}Q_{m_{k}+1}}
			=
			\frac{2}{Q_{m_k}Q_{m_{k}+1}}
			\sum_{j=1}^{m_k-1} Q_j^{1-\alpha} Q_{j+1}. 
	\end{equation}

We recall that $Q_{j+2}\geq 2Q_j$ due to \eqref{eq:q-growth}. For $j<m_k$, there are two cases:
\begin{itemize}
	\item If \(m_k=j+1+2\ell\) for some integer $\ell \geq 0$, then
	\begin{equation*}
			Q_j^{1-\alpha}Q_{j+1}
	\leq
	2^{-\frac{2-\alpha}{2}(m_k-1-j)}
	Q_{m_k-1}^{1-\alpha}Q_{m_k}
	=
	2^{-(2-\alpha)\ell}
	Q_{m_k-1}^{1-\alpha}Q_{m_k}.
	\end{equation*}
	\item If \(m_k\neq j+1+2\ell\) for all integer $\ell\geq 0$, since \(0<\alpha<1\) and \(Q_j\leq Q_{j+1}\leq Q_{j+2}\), we have
	\begin{equation*}
	Q_j^{1-\alpha}Q_{j+1}
	\leq
	Q_{j+1}^{1-\alpha}Q_{j+2}
	\end{equation*}
	and $j+1$ is an index of the previous form: $m_k = (j+1) + 1 +2\ell$. 
\end{itemize}	
Therefore, grouping consecutive terms in pairs, we obtain
\begin{align*}
	\sum_{j=1}^{m_k-1}Q_j^{1-\alpha}Q_{j+1}
	&\leq
	2Q_{m_k-1}^{1-\alpha}Q_{m_k}
	\sum_{\ell=0}^{\infty}2^{-(2-\alpha)\ell}=
	\frac{2}{1-2^{-(2-\alpha)}}
	Q_{m_k-1}^{1-\alpha}Q_{m_k}.
\end{align*}	
	Taking this estimate into the inequality \eqref{eq:anqm<} and the fact that $Q_{m_k+1} \geq c_0 m_k Q_{m_k}$, we obtain
	\begin{align}\label{eq:small-indices-final}
		& \left |
			\sum_{j=1}^{m_k-1} \frac{1}{Q_j^\alpha}\cdot 
		\frac{e^{2\pi i Q_{m_k}\delta_j} - 1}{2\pi i Q_{m_k} \delta_j} 
			\right| 
			\leq 
			\frac{4}{1-2^{2-\alpha}} \cdot \frac{Q_{m_k-1}^\alpha}{Q_{m_k+1}} \nonumber \\
			&\qquad\qquad
			 \leq 
			\frac{4}{1-2^{2-\alpha}} \cdot \left(\frac{Q_{m_k-1}}{Q_{m_k+1}}			\right)^\alpha \cdot \frac{1}{Q_{m_k}^\alpha} \leq \frac{4}{1-2^{2-\alpha}} \left(\frac{1}{c_0^2 m_k(m_k-1)}\right)^\alpha\frac{1}{Q_{m_k}^\alpha} \leq \frac{1}{8}\cdot\frac{1}{Q_{m_k}^{\alpha}}
	\end{align}
	if $m_k$ is large enough. 
	
	\medskip
	\underline{{\it Step 3. Estimating $I_3$	. }}
	For $j>m_k$, we have
\begin{align*}
		\left|
	\frac{e^{2\pi i\delta Q_{m_k}}-1}
	{2\pi i\delta_j Q_{m_k}}
	\right|
	\leq 1 
		\qquad\Longrightarrow\qquad
	\left |
		\sum_{j=m_k}^{\infty} \frac{1}{Q_j^\alpha}\cdot 
		\frac{e^{2\pi i Q_{m_k}\delta_j} - 1}{2\pi i Q_{m_k} \delta_j} 
			\right| 
			\leq \sum_{j=m_k}^{\infty} \frac{1}{Q_j^\alpha}.
\end{align*}
Using \(Q_{j+2}\geq 2Q_j\) due to \eqref{eq:q-growth} and $Q_{m_k+1}\geq c_0 m_k Q_{m_k}$, we split the sum into two terms
\begin{align}
	\sum_{j=m_k+1}^{\infty}\frac{1}{Q_j^\alpha}
	&=
	\sum_{\ell=0}^{\infty}
	\frac{1}{Q_{m_k+1+2\ell}^{\alpha}}
	+
	\sum_{\ell=0}^{\infty}
	\frac{1}{Q_{m_k+2+2\ell}^{\alpha}} \nonumber \\
	&\leq
	\frac{1}{1-2^{-\alpha}}
	\left(
		\frac{1}{Q_{m_k+1}^{\alpha}}
		+
		\frac{1}{Q_{m_k+2}^{\alpha}}
	\right)
	\leq
	\frac{2}{1-2^{-\alpha}}
	\frac{1}{Q_{m_k+1}^{\alpha}}
	\leq
	\frac{2(c_0m_k)^{-\alpha}}
	{1-2^{-\alpha}}
	\frac{1}{Q_{m_k}^{\alpha}} \leq \frac{1}{8}\cdot \frac{1}{Q_{m_k}^\alpha} \label{eq:large-indices-final}
\end{align}
if $m_k$ is large enough. 
\medskip 

	Combining
	\eqref{eq:mainmode-final2},
	\eqref{eq:small-indices-final},
	and
	\eqref{eq:large-indices-final} into \eqref{eq:Reest},
	we obtain
	\begin{align*}
	\frac{1}{Q_{m_k}}\int_0^{Q_{m_k}}f(t,\vartheta t)\,dt
	\geq
	\mathrm{Re}(I_2) - |I_1| - |I_3| 
		\geq 
	\frac{1}{2Q_{m_k}^\alpha}
	-\frac{1}{8Q_{m_k}^\alpha}
	-\frac{1}{8Q_{m_k}^\alpha}
		\geq
	\frac{1}{4Q_{m_k	}^\alpha}.
	\end{align*}
	Taking $T_k=Q_{m_k}$, it follows that
	\begin{align*}
	\frac{1}{T_k}\int_0^{T_k}f(t,\vartheta t)\,dt
	\geq
	\frac{1}{4T_{k}^\alpha},
	\end{align*}
	which proves \eqref{eq:NearlyOptimal2D} for the case $n=2$.
	
For $n\geq 3$, we can always choose $\om=(\om_1,\cdots,
\om_n)\in\R^n$ nonresonant, of which the previous two components $(\om_1,\om_2)\in\R^2$ satisfy the arithmetic property given in Lemma \ref{lem:fullmeasure-new} (by taking $\vartheta=\om_2/\om_1$). As we can see, such  frequency $\om$ exists almost everywhere in $\R^n$. As for the observable $f(x)\in C^{0,\alpha}(\T^n)$, we can always make it depends only on $(x_1,x_2)\in\R^2$ and of a form as in \eqref{def:functionf-new}. Previous argument is still available and get the lower bound as in \eqref{eq:NearlyOptimal2D}. 
\end{proof}

\begin{proof}[Proof of Theorem \ref{thm2:examples}]
The proofs of \eqref{eq:exam-besovsalpha} and \eqref{eq:exam-besov} are given in Propositions \ref{prop:Thm12Casesalpha} and \ref{prop:sharpRateND}, respectively, whereas \eqref{eq:besovAexample} and \eqref{eq:NearlyOptimal2D} are established in Propositions \ref{prop:Thm12ii} and \ref{prop:Thm12iii}, respectively.
\end{proof}

\subsection{Discrete Birkhoff average and its convergence rate}\label{a1}

For any $\omega\in\R^n$ of which $\widehat{\omega}:=(\omega,1)\in\mathcal{D}(\sigma,C_\omega,n+1)$ for some $C_\omega>0$, then
\begin{equation}\label{eq:badphi}
\begin{aligned}
	\Vert \xi\cdot \omega\Vert_{\mathbb{Z}} 
        &:= \min_{\xi_0\in \mathbb{Z}} |\xi_0 + \xi\cdot \omega| 
        = \min_{\xi_0\in \mathbb{Z}} |(\xi,\xi_0)\cdot (\omega,1)| \\
        & \geq \frac{C_\omega}{|(\xi,\xi_0)|_2^\sigma}  
        \geq \frac{C_\omega}{2|\omega|_2}|\xi|^{-\sigma}
        \qquad\text{for all}\;\xi\in \Z^n\backslash \{0\},
\end{aligned}
\end{equation}
Here we use the fact that $|\xi_0+\xi\cdot \omega| \leq \frac{1}{2}$, thus $|\xi_0+\xi\cdot\omega|_2 \leq |\xi|\sqrt{3|\omega|_2 + 1} \leq 2|\omega|_2 |\xi|$. 
The Birkhoff ergodic Theorem yields the uniform convergence of 
\begin{equation*}
    \lim_{N\to \infty}	\frac{1}{N}\sum_{j=0}^{N-1} f(x+j\omega) = \int_{\T^n} f(y)\;dy, \qquad    f\in C(\T^n).
\end{equation*}
We record the following elementary lemma, whose proof is omitted.

\begin{lem}\label{lem:estimateZ}
Let $\omega\in \R^n$ be such that $(\omega,1) \in \mathcal{D}(\sigma, C_\omega, n+1)$.
\begin{itemize}
\item[(i)] Suppose $\Vert z\Vert_{\mathbb{Z}} := \min_{m\in \Z} |z-m|$ is the quotient norm of $z\in\R$ on $\T$, then
\begin{equation}\label{eq:estimateSince}
	\pi \Vert \xi\cdot \omega\Vert_{\Z}  \leq |1-e^{2\pi i \xi \cdot \omega}| \leq 2\pi \Vert \xi\cdot \omega\Vert_{\Z} \qquad\text{for all}\;\xi\in \Z^n. 
\end{equation}

\item[(ii)] For $N\in \N$ and $x\in \T^n$, we have
\begin{align*}
	\left|\sum_{j=0}^{N-1} e^{2\pi \xi\cdot (x+j\omega)}\right| = \left|\frac{1-e^{2\pi i N\xi\cdot \omega}}{1-e^{2\pi i\xi \cdot \omega}}\right| \leq \frac{2}{\pi}\cdot \frac{1}{\Vert \xi\cdot\omega\Vert_{\Z}}.
\end{align*}
\end{itemize} 
\end{lem}

The next lemma is the discrete analog of Lemma \ref{lem:keyThm1}. Here, uniform distribution on the circle requires additional care before deriving the geometric sum used in its proof.

\begin{lem}\label{lem:keyEstimateDiscrete} 
Let $n\geq 1$, $\omega\in \R^n$ be such that $(\omega,1) \in \mathcal{D}(\sigma, C_\omega, n+1)$, and let $\mathcal{A}_j$ be defined as in \eqref{eq:annulusZ}. 
\begin{itemize}
    \item[(i)] If $p\in (1,\infty)$, we have
        \begin{equation}\label{eq:discp}
            \left(	\sum_{\xi\in \mathcal{A}_j} \frac{1}{\Vert \xi\cdot \omega\Vert_{\Z} ^p } \right)^{1/p} 
            \leq 
            \frac{C_p}{C_\omega}\left(\frac{16}{3}\right)^\sigma 2^{j\sigma}, 
            \qquad 
            C_p=\left(2 \sum_{m=1}^\infty \frac{1}{m^p}\right)^{1/p}.  
        \end{equation}
    \item[(ii)] If $p=1$, we have for $j\in \N$ that
        \begin{equation}\label{eq:dispEqual1}
            \sum_{ \xi\in \mathcal{A}_j }
            \frac{1}{\Vert \xi\cdot \omega \Vert _\Z}
            \leq 
            C(n,\sigma)\left( j \cdot 2^{jn} + \frac{2|\omega|_2}{C_\omega}\cdot 2^{j\sigma}\right).
        \end{equation}
\end{itemize}
\end{lem}
\begin{proof} Let $S_j = \{\xi\cdot \omega\;(\mathrm{mod}\;1): \xi\in \mathcal{A}_j\} \subset \mathbb{T}$. Let 
\begin{equation*}
    \kappa_j = C_\omega \cdot \left(\frac{3}{16}\right)^\sigma\cdot  2^{-j\sigma},
\end{equation*}
We show that elements of $S_j$ are uniformly distributed on the circle $\T$. Indeed, if $x,y\in S_j$ then $x=\xi_1\cdot \omega \;(\mathrm{mod}\;1)$ and $y=\xi_2\cdot \omega \;(\mathrm{mod}\;1)$ for $\xi_1,\xi_2 \in \mathcal{A}_j$, then $|\xi_1-\xi_2|\leq \frac{16}{3}\cdot 2^{j}$, thus
\begin{equation}\label{eq:equi-distri}
\begin{aligned}
	&\Vert x-y\Vert_{\Z} 
	= 
	\Vert (\xi_1-\xi_2)\cdot \omega\Vert_{\Z} \geq C_\omega|\xi_1-\xi_2|^{-\sigma} \geq  C_\omega \cdot \left(\frac{3}{16}\right)^\sigma\cdot  2^{-j\sigma} \geq \kappa_j\\
	& \Vert x\Vert_\Z 
	= \Vert \xi_1\cdot \omega\Vert_\Z \geq C_\omega|\xi_1|^{-\sigma}	 \geq C_\omega\cdot\left(\frac{3}{8}\right)^\sigma\cdot 2^{-j\sigma} \geq \kappa_j\\
	& \Vert y\Vert_\Z 
	= \Vert \xi_2\cdot \omega\Vert_\Z \geq C_\omega|\xi_2|^{-\sigma}	 \geq C_\omega\cdot\left(\frac{3}{8}\right)^\sigma\cdot 2^{-j\sigma} \geq \kappa_j. 
\end{aligned}
\end{equation}
For $m=1,2,\ldots$, we define $E_m = \{x\in \T: m\kappa_j \leq \Vert x\Vert _\Z < (m+1)\kappa_j\}$.
From \eqref{eq:equi-distri} we obtain that $\#(E_m\cap S_j) \leq 2$. Thus
\begin{align*}
	\left(\sum_{x\in \mathcal{S}_j} \frac{1}{\Vert x\Vert_\Z^p}\right)^{1/p} 
	\leq 
	\left(
		\sum_{m=1}^{\lfloor (2\kappa_j)^{-1	} \rfloor}  
        \frac{\#(E_m\cap S_j)}{(m\delta_j)^p	}
	\right)^{1/p}
	\leq 	
	\left( 2\sum_{m=1}^\infty \frac{1}{(m\delta_j)^p} \right)^{1/p}
	= \frac{C_p}{\kappa_j},
\end{align*}
where $\left\lfloor x \right\rfloor$ is the integer part of $x$, namely $\left\lfloor x \right\rfloor \leq x < \left\lfloor x \right\rfloor  + 1$. 
We obtain the conclusion \eqref{eq:discp}.

For $p=1$, by \eqref{eq:badphi} we have 
\begin{equation*}
    \Vert \xi\cdot \omega\Vert_{\mathbb{Z}} 
       \geq 
    \frac{\phi(|\xi|)}{|\xi|^{n}}
        \qquad\text{for all}\;\xi\in \Z^n\backslash \{0\}, \qquad \phi(t) = \left(\frac{C_\omega}{2|\omega|_2}\right) t^{n-\sigma}. 
\end{equation*}
Since $\sigma \geq n$, the function $t\mapsto \phi(t)$ is non-decreasing on $(0,\infty)\to (0,\infty)$. By Definition \ref{defn:badly-approx}, $\omega$ is $\phi$-badly approximable. Theorem \ref{thm:freg2019thm12} with $m=n$ yields 
\begin{equation*}
        \sum_{ \xi\in \mathcal{A}_j }
            \frac{1}{\Vert \xi\cdot \omega \Vert _\Z}
            \leq 
        \sum_{ \xi \in \Z^n\backslash \{0\}, |\xi|\leq \frac{8}{3}\cdot 2^j}
            \frac{1}{\Vert \xi\cdot \omega \Vert _\Z}
            \leq 
            C(n) 
            \left( 
                Q^n\log Q + \frac{Q^n}{\phi(Q)}
            \right),
\end{equation*}
where $Q = \frac{8}{3}\cdot 2^j$. This reduces to the conclusion \eqref{eq:dispEqual1}. 
\end{proof}

Now we give a proof of the discrete case, an analog of Theorem \ref{thm:BesovHolderCounting}.

\begin{prop}[Discrete Besov setting]\label{prop:DiscreteBesov}
Let $n\geq 1$, $\omega \in \R^n$ be such that $(\omega,1) \in \mathcal{D}(\sigma, C_\omega, n+1)$, $x\in \T^n$ and $s>0$. Let 
\begin{equation*}
    f\in B^s_{p,q}(\T^n), \qquad 1\leq p,q \leq \infty, \qquad s>\frac{n}{p}.
\end{equation*}
    There exists a positive constant $C=C(\sigma, s, p,q,n)$ such that
    \begin{subnumcases}{\displaystyle
        \left|
            \frac{1}{N}\sum_{\ell=0}^{N-1}  f(x+\ell\omega )
            -
            \int_{\T^n}f(y)\;dy
        \right|
    \leq
    \frac{C}{C_\omega}
    \Vert f\Vert_{B^s_{p,q}(\T^n)}}
        N^{-\frac{s-n/p}{\sigma-n/p}},
            & $\displaystyle \tfrac{n}{p}<s<\sigma$,
            \label{eq:BesovAllPQaDiscrete}\\
        N^{-1}(\log N)^{1-1/q},
            & $\displaystyle \tfrac{n}{p}<s=\sigma$,
            \label{eq:BesovAllPQbDiscrete}\\
        N^{-1},
            & $\displaystyle s>\max\{\sigma, \tfrac{n}{p}\}$.
            \label{eq:BesovAllPQcDiscrete}
    \end{subnumcases}
     Furthermore, for the endpoint $q=1$, the supercritical case \eqref{eq:BesovAllPQcDiscrete} can be relaxed to $s\geq \max \{\sigma, \frac{n}{p}\}$. 
    \begin{equation}
    \begin{aligned}
     &  
            \left|\frac{1}{N}\sum_{\ell=0}^{N-1}  f(x+\ell\omega )- \int_{\T^n} f(y)\;dy  \right| 
            \leq 
            \frac{C(\sigma,p)}{C_\omega}
            \Vert f\Vert_{{B^s_{p,1}}(\T^n)} \cdot \frac{1}{N},              && s\geq \max \{\sigma, \tfrac{n}{p}\}. 
            \label{eq:BesovALemmaP1Discrete}
	   \end{aligned}
     \end{equation}
\end{prop}

\begin{proof}
Without loss of generality, we assume $\int_{\T^n} f(y)\;dy = 0$. 
Let $\mathcal{A}_j$ be defined in \eqref{eq:annulusZ} and $\kappa_j = C_\omega \left( \frac{3}{16}\right)^\sigma 2^{-j\sigma}$. For each fixed $j$, by Lemma \ref{lem:estimateZ} we have
\begin{equation*}
    \left| \sum_{\ell=0}^{N-1} \Delta_j f(x+\ell \omega) \right|
	= 
	\left| \sum_{\xi\in \mathcal{A}_j} \widehat{\Delta_j f}(\xi)\sum_{\ell=0}^{N-1} e^{2\pi i\xi \cdot(x+\ell \omega)} \right|
    \leq 
    \frac{2}{\pi}\sum_{\xi\in \mathcal{A}_j} |\widehat{\Delta_j f}(\xi)| \cdot \frac{1}{\Vert \xi\cdot\omega\Vert_{\Z	}}
\end{equation*}
\medskip

\noindent 
\underline{\it Case 1. $p\in (1,2]$.}
Using H\"older's inequality with conjugate exponents $p\in (1,2]$ and $p'\in [2,\infty)$, along with the Hausdorff--Young inequality and Lemma \ref{lem:keyEstimateDiscrete}-(i), we deduce that
\begin{equation}\label{eq:keyDisEstimateA}
\begin{aligned}
    & \frac{1}{\pi}\sum_{\xi\in \mathcal{A}_j} |\widehat{\Delta_j f}(\xi)|\cdot \frac{1}{|\xi\cdot \omega|} 
    \leq 
        \frac{1}{\pi } \left(\sum_{\xi\in \mathcal{A}_j} |\widehat{\Delta_j f}(\xi)|^{p'} \right)^{1/p'}  \left( \sum_{\xi\in \mathcal{A}_j} \frac{1}{|\xi\cdot\omega|^p} \right)^{1/p}\\
    &\qquad \qquad 
        \leq
        \frac{1}{\pi} \Vert \widehat{\Delta_j f} \Vert_{\ell^{p'}(\Z^n)} \frac{C_p}{\kappa_j} \leq \frac{C_p}{\pi } \Vert \Delta_j f\Vert_{L^p(\T^n)}  \frac{1}{\kappa_j}   
        = 
        \frac{C_p}{\pi C_\omega} \left(\frac{16}{3}\right)^\sigma  2^{j\sigma} \Vert \Delta_j f\Vert_{L^p(\T^n)}. 
\end{aligned}
\end{equation}
\noindent 
\underline{\it Case 2. $p\in(2,\infty]$.} By \eqref{eq:keyDisEstimateA} with $p=2$ and the fact that $L^p(\T^n)\hookrightarrow L^2(\T^n)$ for $p>2$ we obtain 
\begin{equation*}
\begin{aligned}
	\frac{1}{\pi} \sum_{\xi\in \mathcal{A}_j} |\widehat{\Delta_j f}(\xi)|\cdot \frac{1}{|\xi\cdot \omega|} 
	& \leq 
	\frac{C_2}{\pi C_\omega}\left(\frac{16}{3}\right)^\sigma 2^{j\sigma} \Vert\Delta_j f\Vert_{L^2(\T^n)} 
	\leq 
		\frac{C_2}{\pi C_\omega}\left(\frac{16}{3}\right)^\sigma 2^{j\sigma} \Vert\Delta_j f\Vert_{L^p(\T^n)}. 
\end{aligned}
\end{equation*}
Since $\widehat f(0)=0$, the low-frequency block $\Delta_{-1}f$ contains only finitely many nonzero Fourier modes and can therefore be included in the preceding estimate by enlarging the constant. 
In summary, we have
\begin{equation}\label{eq:allp}
    \left| \sum_{\ell=0}^{N-1} \Delta_j f(x+\ell \omega) \right|
	\leq 
        C \cdot 2^{j\sigma}\Vert \Delta_j f\Vert_{L^p(\T^n)},
        \qquad \text{where}\qquad 
        C = \frac{C_p}{\pi C_\omega}\left(\frac{16}{3}\right)^\sigma, p\in (1,\infty]. 
\end{equation}

\noindent 
\underline{\it Case 3. $p=1$.} 
Using \eqref{eq:dispEqual1} we obtain
\begin{equation*}
\begin{aligned}
	\frac{1}{\pi} \sum_{\xi\in \mathcal{A}_j} |\widehat{\Delta_j f}(\xi)|\cdot \frac{1}{|\xi\cdot \omega|} 
	& \leq 
    \frac{1}{\pi} \Vert \widehat{\Delta_j f}\Vert_{L^\infty(\T^n)} \sum_{\xi\in \mathcal{A}_j} \frac{1}{\Vert \xi\cdot \omega\Vert_\Z} \\
    & \leq 
    \frac{1}{\pi} \Vert \Delta_j f\Vert_{L^1(\T^n)} \cdot C(n,\sigma)\left( j\cdot 2^{j(n-1)} + \frac{2|\omega|_2}{C_\omega}\cdot 2^{j\sigma}\right). 
\end{aligned}
\end{equation*}
Here we also treat the low-frequency block $\Delta_{-1}f$ similarly. The proof then follows similarly to that of Theorem \ref{thm:BesovHolderCounting}.
\end{proof}

\section{Application: Homogenization Rate of Hamilton--Jacobi equations}
\label{Sec:AppHomogenization}

We consider $H(x,\xi) = \frac{1}{2}|\xi|^2 - V(x)$ for $(x,\xi)\in \R^2$, where $ V(x) = f(\omega x)$ for $\quad x\in \R$, with $f\in C(\T^n)$ such that $\min_{\T^n} f = 0$. We briefly recall the setting from \cite{hu_polynomial_2025}. 
We can assume $x=0$ without loss of generality. 
The solution to \eqref{eq:HJ-intro} can be written as (see \cite{bardi_optimal_1997,le_dynamical_2017, tran_hamilton-jacobi_2021})
\begin{equation*}
    u^\varepsilon(0,t) = \inf 
    \left\lbrace 
    \varepsilon\int_0^{\varepsilon^{-1}t } 
    \left(
        \frac{|\dot{\eta}(s)|^2}{2} + V(\eta(s))
    \right)\;ds
     + 
     u_0(\varepsilon\eta(\varepsilon^{-1}t)):
    \varepsilon\eta(0) = 0, 
    \dot{\eta}\in 
    L^1([0,\varepsilon^{-1}t])
    \right\rbrace.
\end{equation*}
Let $\mathcal{A} =\big\lbrace \eta\in \mathrm{AC}([0,\varepsilon^{-1}t]): \eta(0) = 0\big\rbrace$ where $\mathrm{AC}([a,b])$ denotes the set of absolutely continuous functions from $[a,b]$ to $\mathbb{R}$, and 
\begin{equation}\label{eq:A^eps}
    A^\varepsilon[\eta] = \varepsilon\int_0^{\varepsilon^{-1}t} \left(\frac{|\dot{\eta}(s)|^2}{2} +  V(\eta(s))\right)\;ds + u_0\big(\varepsilon \eta(\varepsilon^{-1}t)\big), \qquad \eta \in \mathcal{A}. 
\end{equation}
We have $u^\varepsilon(0,t) = \inf_{\eta\in \mathcal{A}} A^\varepsilon[\eta]$. 
Conservation of energy implies that, if $\eta$ is a minimizer then there exists $r\in [r_{\textrm{min}},+\infty)$ where $r_\textrm{min} = \min_{\mathbb{R}}V$ such that
\begin{equation}\label{eq:rate-conservation-energy}
    H\left(\eta(s),\dot{\eta}(s)\right) = \frac{|\dot{\eta}(s)|^2}{2} - V(\eta(s)) = r \qquad\text{for all}\; s\in (0,\varepsilon^{-1}t).
\end{equation}
The corresponding minimizer solves
\begin{equation}\label{eq:ode}
    \begin{cases}
    \begin{aligned}
        |\dot{\eta}(s)| &= \sqrt{2(r+V(\eta(s)))}, & & s\in (0,\varepsilon^{-1}t),\\
        \eta(0) &= 0. & &
    \end{aligned}
    \end{cases}
\end{equation}
For $r\in [r_{\text{min}},+\infty)$ we define 
\begin{equation*}
    \mathcal{A}_r = \left\lbrace \eta\in \mathcal{A}\;\text{is a minimizer of}\;u^\varepsilon(0,t)\;\text{with}\; H(\eta(s),\dot{\eta}(s)) = r\;\text{in}\;(0,\varepsilon^{-1}t)\right\rbrace.
\end{equation*}
Then 
\begin{equation}\label{eq:u-eps-01}
    u^\varepsilon(0,t) = \inf_{r} 
    \left\lbrace 
    \inf_{\eta \in \mathcal{A}_r} A^\varepsilon[\eta] 
    \right\rbrace.
\end{equation}
Furthermore, by \cite[Lemma 4.9]{hu_polynomial_2025} we can ignore the value of $r\geq r_0$ in \eqref{eq:u-eps-01} for some $r_0>0$ depending only on $\|f\|_{L^\infty(\T^n)}$ and $\|u_0\|_{C^{0,1}(\R)}$. 
Denote 
\begin{equation}\label{eq:p0}
    p_r := \int_{\T^n} \sqrt{2(r + f(x))}\;dx, \qquad r\geq 0.
\end{equation}
If $|p|\geq |p_0|$ (corresponding to $\overline{H}(p)\geq 0$) then the cell problem \eqref{eq:CPdelta} has an exact sublinear corrector $v_p$, defined by
\begin{align}\label{eq:vp}
    v_p(t) = \int_0^t \sqrt{2(\mu + f(\omega t))}\;dt - pt, \qquad t\in \R 
\end{align}
where $\mu = \overline{H}(p)$. The region $[-p_0,p_0]$ where $\overline{H} = 0$ is called the \emph{flat part} \cite{le_dynamical_2017,tran_hamilton-jacobi_2021}. If $\overline{H}$ is differentiable at $p$, it is known that \cite[Lemma 3.2]{hu_polynomial_2025}
\begin{align*}
	\frac{1}{\overline{H}'(p)} = \int_{\T^n} \Big( 2\left(\overline{H}(p) + f(x)\right)\Big)^{-1/2}\;dx. 
\end{align*}
Let $\overline L(v):=\max_{p\in\R} \big( v\cdot p -\overline H(p)\big)$.
Due to the Hopf-Lax fomula, we get 
\begin{equation*}
	u(0,t)=\inf_{x\in\R}\Big(\overline L\left(-\tfrac{x}{t}\right)+u_0(x)\Big). 
\end{equation*}

The lower bound for
\(u^\varepsilon-u\) depends on the decay of $v_p(s)/s$ as $s\to \infty$, where \(v_p\) is defined in \eqref{eq:vp}. The upper bound additionally
requires the convergence of \(\eta(s)/s\) to its rotation vector; see
\eqref{eq:bound-master-1}. Related
subsequential convergence results for $\eta(s)/s$ appear in
\cite{EWeinanAubryMatherCPAM1999,EvansGomesARMAPart12001,fathi_book,
gomes_CalVarPDE2002,tran_yu_WeakKAM2022}. Let us recap the idea developed in \cite{tu_2018_rate_asymptotic} and \cite{hu_polynomial_2025}. Let $\eta_r$ be a minimizer for $r = \overline{H}(p)$.
\begin{itemize}
\item[(i)] (Lower bound \cite[Proposition 4.10]{hu_polynomial_2025}) Using, in part, the Hopf--Lax formula, we obtain
\begin{equation}\label{eq:Action}
    A^\varepsilon[\eta_r] \geq u(0,t) + \inf_{|p|\geq |p_0|} \varepsilon v_p(\eta_r(\varepsilon^{-1}t)).
\end{equation}
Together with \eqref{eq:u-eps-01} we deduce a lower bound estimate fro $u^\varepsilon(0,t)- u(0,t)$, which depends on the decay rate of $v_p(s)/s$. 

\item[(ii)] (Upper bound \cite[Proposition 4.11]{hu_polynomial_2025}) We have $u^\varepsilon(0,t) \to u(0,t)$ (see \cite{Ishii_almost_periodic_1999, tran_hamilton-jacobi_2021}) as $\varepsilon\to 0^+$, and also from \eqref{eq:u-eps-01}:
\begin{equation}\label{eq:uepsfor}
	u^\varepsilon(0,t) = \min
    \left\lbrace 
        \inf_{r\geq 0}A^\varepsilon[\eta_r],  
        \inf_{r\leq 0}A^\varepsilon[\eta_r]
    \right\rbrace. 
\end{equation}
If we can find a quantity \(I_r\), independent of \(\varepsilon\), such that
\begin{equation*}
\left|
\min\left\{
\inf_{r\geq 0} A^\varepsilon[\eta_r],
\inf_{r\leq 0} A^\varepsilon[\eta_r]
\right\}
-I_r
\right|
\leq C\varepsilon^\theta,
\end{equation*}
then $|u^\varepsilon(0,t)-I_r|\leq C\varepsilon^\theta$. Letting \(\varepsilon\to0^+\), we obtain \(u(0,t)=I_r\), and hence the
upper-bound rate is attained. We will show that $I_r=\min\{I_r^-,I_r^+\}$ where \(I_r^-\) and \(I_r^+\) correspond to the large time average of the minimization problem over
\(r\leq 0\) and \(r\geq 0\), respectively.

\end{itemize}

\begin{proof}[Proof of Theorem \ref{thm:homogenization}] 
We track the dependence on \(t\), rather than
normalizing \(t=1\) as in \cite{hu_polynomial_2025}. This yields a sharper upper bound in some cases where the estimate remain uniform as $t\to \infty$.

\medskip

\noindent 
\underline{{\bf Part 1. The lower bound of $u^\varepsilon-u$.}}
We note that 
\begin{align*}
    \frac{v_p(t)}{t} = \frac{1}{t}\int_0^t \sqrt{2(\mu + f(\omega t))}\;dt - \int_{\T^n}\sqrt{2(\mu + f(x))}\;dx.
\end{align*}
By Theorem \ref{thm:BesovHolderCounting}, the decay rate of $\frac{v_p(t)}{t}$ depends on the regularity in Besov norm of 
\begin{equation}\label{eq:fmu}
    f_\mu(x) = \big( 2(\mu + f(x))\big)^{1/2}, \qquad x\in \T^n, \mu \geq 0. 
\end{equation}
\begin{itemize}
\item If $f\in C^{0,1}(\T^n)$, then $f_\mu \in C^{0,\frac{1}{2}}(\T^n)$ uniformly in $0\leq \mu \leq \mu_0$. By Theorem \ref{thm:BesovHolderCounting} we obtain
\begin{align}\label{eq:ratevp}
    \left|\frac{v_p(s)}{s}\right| \leq \frac{C}{s^{\frac{1}{2\sigma}}}. 
\end{align}
There exists a constant $C_0$ (corresponding to $\mu_0$) such that, for $|p_0|\leq |p|\leq C_0$ then
    \begin{align}\label{eq:S1}
        |\varepsilon v_p(\eta(\tfrac{t}{\varepsilon}))|
        & =
        |\varepsilon \eta(\tfrac{t}{\varepsilon})|\cdot \left|
            \frac{v_p(\eta(\tfrac{t}{\varepsilon}))}{\eta(\tfrac{t}{\varepsilon})}
        \right| \leq |\varepsilon \eta(\tfrac{t}{\varepsilon})|\cdot 
        \frac{C}{|\eta(\tfrac{t}{\varepsilon})|^{\frac{1}{2\sigma}}} \leq C \cdot t^{1-\frac{1}{2\sigma}}\cdot \varepsilon^\frac{1}{2\sigma}.
    \end{align}
Here we use the fact that $|s\eta(s^{-1})|$ is bounded if $\eta$ is a minimizer for $0\leq r\leq C$, which follows from \eqref{eq:ode}. 

\item Similarly, suppose that \(f\in C^2(\T^n)\) has only non-degenerate minima. By choosing suitable coordinates on $\T^n$ and using the following  Lemma~\ref{lem:keyBgamma}, we may reduce to the case of a unique non-degenerate minimum \(x_0\), with \(f(x_0)=0\). 
By the Morse Lemma, locally we have $f(x)\approx C|x-x_0|^2$ where $x_0$ is the minimum point, thus by Lemma \ref{lem:keyBgamma} we obtain $f_\mu \in B^1_{\infty,\infty}(\T^n)$ uniformly for $0\leq \mu\leq \mu_0$. By Theorem \ref{thm:BesovHolderCounting} we 
obtain
\begin{align}\label{eq:ratevplog}
    \left|\frac{v_p(s)}{s}\right| \leq \frac{C\;\log(s)}{s^{\frac{1}{\sigma}}}. 
\end{align}
Therefore, for $|p_0|\leq |p|\leq C$ we deduce that
\begin{equation}\label{eq:S2}
    \begin{aligned}
        |\varepsilon v_p(\eta(\tfrac{t}{\varepsilon})|
        &=
        |\varepsilon \eta(\tfrac{t}{\varepsilon})|\cdot \left|
            \frac{v_p(\eta(\tfrac{t}{\varepsilon}))}{\eta(\tfrac{t}{\varepsilon})}
        \right|  \\
        & \leq  |\varepsilon \eta(\tfrac{t}{\varepsilon})|\cdot 
        \frac{C\;\log(\eta(\tfrac{t}{\varepsilon}))}{|\eta(\tfrac{t}{\varepsilon})|^{\frac{1}{\sigma}}} 
        \leq C \cdot t^{1-\frac{1}{\sigma}}\cdot \varepsilon^\frac{1}{\sigma}\cdot \big(\log t + |\log \varepsilon|\big) . 
    \end{aligned}
\end{equation}
\end{itemize}
From \eqref{eq:S1}, \eqref{eq:S2}, and \eqref{eq:Action}, by taking the infimum in view of \eqref{eq:u-eps-01} we obtain 
\begin{align*}
    u^\varepsilon(0,t) - u(0,t)\geq -C
    \begin{cases}
    \begin{aligned}
        &t^{1-\frac{1}{2\sigma}} \varepsilon^{\frac{1}{2\sigma}}   && f\in C^{0,1}(\T^n)\\ 
        &t^{1-\frac{1}{\sigma}} \varepsilon^{\frac{1}{\sigma}} \big( \log t + |\log(\varepsilon)| \big)   &&  f\in C^2(\T^n)\;\text{has only non-degenerate minima}. 
    \end{aligned}
    \end{cases}
\end{align*}
\medskip

\noindent 
\underline{{\bf Part 2. The upper bound of $u^\varepsilon-u$.}}
For $|p_0|\leq |p|\leq C_0$, let $v_p$ be defined by \eqref{eq:vp} and $\eta:[0,\infty) \to \mathbb{R}$ be a characteristic that corresponds to $p$. 
Let $\tilde{p}, p \geq p_0$ and $\tilde{\mu} = \overline{H}(\tilde{p}), \mu = \overline{H}(p) \geq 0 = \overline{H}(p_0)$. 
We have
\begin{equation}\label{eq:bound-eta}
    \left|\frac{\eta(t)}{t}\right| \leq \max_{s\in (0,\infty)} |\dot{\eta}(s)| \leq C, \qquad\text{where}\qquad 
    C=\sqrt{2}\left(\max\{\mu, \tilde{\mu}\} + \Vert f\Vert_{L^\infty}\right)^{1/2}. 
\end{equation}
Let $v_p$ and $v_{\tilde{p}}$ be the correctors to the cell problem for $p$ and $\tilde{p}$, respectively with $v_p(0) = v_{\tilde{p}}(0) = 0$. 
Since $p,\tilde{p} > p_0$, we can select 
\begin{equation}\label{eq:choicepp}
	\tilde{p} = p + \omega(t)\;\mathrm{sign}\left(\frac{\eta(t)}{t} - \overline{H}'(p)\right)
\end{equation}
where $\omega(t) \to 0^+$ as $t\to +\infty$, to be chosen. 
We have $\dot{\eta}(s) = D_pH(\eta(s), p+v'_p(\eta(s)))$, therefore, the equality in Fenchel-Young inequality holds:
\begin{align*}
    \int_0^{t} \Big( L(\eta(s), \dot{\eta}(s)) + \overline{H}(\tilde{p})\Big)\;ds &\geq \int_0^t\dot{\eta}(s)\big(\tilde{p} + Dv_{\tilde{p}}(\eta(s))\big)ds = \tilde{p} \eta(t) + v_{\tilde{p}}(\eta(t)),\\
    \int_0^{t} \Big( L(\eta(s), \dot{\eta}(s)) + \overline{H}(p)\Big)\;ds &= \int_0^t \dot{\eta}(s)\big(p + Dv_p(\eta(s))\big)ds = p \eta(t) + v_{p}(\eta(t)).
\end{align*}
Subtracting these equations with $\overline{H}'(p)(\tilde{p} - p)$, using the fact that $\overline{H}$ is $C^{1,\beta}(\R)$, convex, and also \eqref{eq:choicepp}, we have
\begin{align}\label{eq:estimate-p-tilde-p}
	C|\tilde{p}-p|^{1+\beta} 
		&\geq 
    \overline{H}(\tilde{p}) - \overline{H}(p) - \overline{H}'(p)(\tilde{p} - p) \nonumber \\
    		&\geq \left(\frac{\eta(t)}{t}-\overline{H}'(p)\right)(\tilde{p}-p) + \frac{v_{\tilde{p}}(\eta(t))}{t} - \frac{v_{p}(\eta(t))}{t}.
\end{align}
If $f\in C^{0,1}(\T^n)$ then thanks to \eqref{eq:ratevp} and \eqref{eq:bound-eta} we have
\begin{equation*}
\begin{aligned}
    \left|\frac{v_{p}(\eta(t))}{t}\right| 
    +
    \left|\frac{v_{\tilde{p}}(\eta(t))}{t}\right|  
    & \leq 
    \frac{\eta(t)}{t} 
    \left(\frac{(C_p+C_{\tilde{p}})C_0}{\eta(t)^{\frac{1}{2\sigma}}}\right) 
    \leq \frac{C}{t^{\frac{1}{2\sigma}}}. 
\end{aligned}
\end{equation*}
By \eqref{eq:estimate-p-tilde-p} we have
\begin{equation}\label{eq:bound-master-1}
    \left|\frac{\eta(t)}{t} - \overline{H}'(p)\right| 
    \leq C|\tilde{p} - p|^{1+\beta} + \frac{C}{t^{\frac{1}{2\sigma}}} \leq C 
      \left(\frac{1}{|t|}\right)^{\frac{1}{2\sigma}\cdot\frac{\beta}{1+\beta}} 
\end{equation}
by choosing $|\tilde{p} - p| = \omega(t)$ appropriately. 

Since the constant $C$ in the bound is independent of $p$ as long as $|p_0|\leq |p|\leq C_0$, and $\overline{H}'(p)$ is continuous as $p\to p_0^+$, we can deduce the same results \eqref{eq:bound-master-1} for $|p|\geq |p_0|$. 
We utilize \eqref{eq:bound-master-1} to obtain the upper bound for $u^\varepsilon(0,t) - u(0,t)$.

\smallskip
\underline{\textit{Step 1. Positive energies}} If $\eta_r \in \mathcal{A}_r$, then for $r\geq 0$ $\eta_r$ can only solve either 
\begin{equation}\label{eq:ode-positive}
\begin{cases}
\begin{aligned}
    \dot{\eta}(s) &= +\sqrt{2(r-V(\eta(s)))},  \qquad s\in (0,\infty),\\
    \eta(0) &= 0, 
\end{aligned}
\end{cases}
\end{equation}
or 
\begin{equation}\label{eq:ode-negative}
\begin{cases}
\begin{aligned}
    \dot{\eta}(s) &= -\sqrt{2(r-V(\eta(s)))},   \qquad s\in (0,\infty),\\
    \eta(0) &= 0,
\end{aligned}
\end{cases}
\end{equation} 
It suffices to consider $\eta_r \geq 0$ (the other case is similar). For $r\geq 0$ we have $r = \overline{H}(p_r)$ where $p_r$ is defined by \eqref{eq:p0}. Similar to \cite[Lemma 4.8]{hu_polynomial_2025}, the action \eqref{eq:A^eps} can be written as
\begin{align}\label{eq:Ar}
	A^\varepsilon[\eta_r] 
	= -rt + 
	\varepsilon \int_0^{\eta\left(\tfrac{t}{\varepsilon}\right)} \sqrt{2(r+V(x))}\;dx
	+ u_0\left(\eta_r\left(\tfrac{t}{\varepsilon}\right)\right).
\end{align}
We define
\begin{align*}
	    A^{\pm}_r: &= -rt + t\overline{H}'(\pm p_r) \cdot (\pm p_r) + u_0\left(t\overline{H}'(\pm p_r)\right). 
\end{align*}
For $r\geq 0$ and $\eta_r\in \mathcal{A}_r$ with $\eta_r\geq 0$, we compare $A^\varepsilon[\eta_r]$ and $A^+_r$ using \eqref{eq:Ar} that
\begin{align*}
    \left|A^\varepsilon[\eta_r] - A^{+}_r\right| 
        &\leq |\varepsilon\eta_r\left(\tfrac{t}{\varepsilon}\right)|\cdot \left|\frac{1}{\eta_r(\varepsilon^{-1}t)}\int_0^{\eta_r(\varepsilon^{-1}t)}f_r(\omega x)\;dx - \int_{\T^n}f_r(x)\;dx \right|  \\
        & \qquad\qquad \qquad \qquad\qquad\qquad + \left|\varepsilon \eta_r\left(\tfrac{t}{\varepsilon}\right)  - t\overline{H}'(p_r)\right| \cdot \left(|p_r| + \mathrm{Lip}(u_0)\right).
\end{align*}
If $f\in C^{0,1}(\T^n)$ then by \eqref{eq:bound-master-1} we have
\begin{align*}
	    \left|
        \varepsilon \eta_r(\tfrac{t}{\varepsilon}) - 	
        	t\overline{H}'(p_r)
    \right| \leq Ct^{1-\frac{\beta}{2\sigma(1+\beta)}} \varepsilon ^\frac{\beta}{2\sigma(1+\beta)}.
\end{align*}
Thanks to \eqref{eq:S1} we conclude that
\begin{align*}
	\left|A^\varepsilon[\eta_r] - A^{+}_r\right|  
	\leq 
	C t^{1-\frac{1}{2\sigma}}\cdot \varepsilon^\frac{1}{2\sigma} + Ct^{1-\frac{\beta}{2\sigma(1+\beta)}} \varepsilon ^\frac{\beta}{2\sigma(1+\beta)}.
\end{align*}
Here $C$ also depends on $r_0$, due to the restriction $|r|\leq r_0$. 
Similarly, if $r\geq 0$ and $\eta_r \in \mathcal{A}_r$ with $\eta_r\leq 0$, we compare $A^\varepsilon[\eta_r]$ with $A_r^-$ and we can obtain the same rates. \medskip

\underline{\textit{Step 2. Negative energies}}
If $r<0$, from \eqref{eq:ode-positive} and \eqref{eq:ode-negative} we have 
\begin{equation*}
	\eta_0^-(s) \leq \eta_r(s) \leq \eta_0^+(s), \qquad s\in \left(0,\tfrac{t}{\varepsilon}\right),
\end{equation*}
where $\eta_0^\pm$ are solutions to \eqref{eq:ode-positive}, \eqref{eq:ode-negative} with $r=0$, respectively. In particular, we deduce that 
\begin{equation*}
    0 \leq 
    |\varepsilon \eta_r\left(\tfrac{t}{\varepsilon}\right)|
    \leq 
    \max 
    \left\lbrace
    |\varepsilon \eta^-_0\left(\tfrac{t}{\varepsilon}\right) - 
    		t\overline{H}'(-p_0)|,
    |\varepsilon \eta^+_0\left(\tfrac{t}{\varepsilon}\right) - 
	    t\overline{H}'(+p_0)|
    \right\rbrace.
\end{equation*}
We have $\overline{H}'(\pm p_0) = 0$. From \eqref{eq:A^eps}, \eqref{eq:rate-conservation-energy} and $L(x,v)\geq 0$ for all $(x,v)$ due to the assumption, we have 
\begin{equation*}
    \inf_{r\leq 0} A^\varepsilon[\eta_r] 
    \geq 
    u_0(\varepsilon \eta_r\left(\tfrac{t}{\varepsilon}\right)) 
    \geq 
    u_0(0) - C|\varepsilon\eta_0^-\left(\tfrac{t}{\varepsilon}\right)|.
\end{equation*}
On the other hand, we have
\begin{align*}
     \inf_{r\leq 0} A^\varepsilon[\eta_r] 
        &\leq A^\varepsilon[\eta_0^+] = \varepsilon \int_0^{\eta_0^+(\varepsilon^{-1}t)} \sqrt{-2V(x)}\;dx + u_0(\varepsilon\eta_0^+(\varepsilon^{-1})) 
        \leq u_0(0) + C|\varepsilon \eta^+_0\left(\tfrac{t}{\varepsilon}\right)| .
\end{align*}
We conclude that if $f\in C^{0,1}(\T^n)$ then 
\begin{equation*}
    \left|
    \inf_{r\leq 0} A^\varepsilon[\eta_r] - u_0(0)
    \right| 
    \leq C|\varepsilon\eta_0\left(\tfrac{t}{\varepsilon}\right)|
    \leq 
    C \left(
    		t^{1-\frac{1}{2\sigma}} \varepsilon^\frac{1}{2\sigma} 
    		+ 
    		t^{1-\frac{\beta}{2\sigma(1+\beta)}} \varepsilon ^\frac{\beta}{2\sigma(1+\beta)}
    		\right).
\end{equation*}

From step 1 and step 2, we obtain that, if $f\in C^{0,1}(\T^n)$ then 
\begin{align*}
    & \left|
    \min
    \left\lbrace 
        \inf_{r\geq 0}A^\varepsilon[\eta_r],  
        \inf_{r\leq 0}A^\varepsilon[\eta_r]
    \right\rbrace 
    - 
    \min 
    \left\lbrace  
        \inf_{r\geq 0} A^{\pm}_r,\inf_{r\leq 0} u_0(0)
    \right\rbrace 
    \right|  \\
    &\qquad \qquad  \leq
    C \left(
    		t^{1-\frac{1}{2\sigma}} \varepsilon^\frac{1}{2\sigma} 
    		+ 
    		t^{1-\frac{\beta}{2\sigma(1+\beta)}} \varepsilon ^\frac{\beta}{2\sigma(1+\beta)}
    		\right) 
    		\leq
    	    C  
    	    \begin{cases}
    	    	t              & \qquad 0 < t \leq \varepsilon, \\
    	    	t^{1-\frac{\beta}{2\sigma(1+\beta)}} \varepsilon ^\frac{\beta}{2\sigma(1+\beta)}
			      	    	& \qquad t>\varepsilon.     	    	
    	    \end{cases}
\end{align*}
From \eqref{eq:uepsfor} and $u^\varepsilon\to u$ as $\varepsilon\to 0^+$, we obtain the desired conclusion.

Finally, for every \(\delta>0\), the set of \(\omega\in\mathbb R^n\) with Diophantine index \(n-1+\delta\) has full measure; see \cite{schmidt_diophantine_1980} for the proof. Therefore, \eqref{eq:generalQuasi} yields the desired conclusion for a.e. \(\omega\in\mathbb{R}^n\), up to an arbitrary \(\delta\)-loss.
\end{proof}

\begin{lem}
\label{lem:keyBgamma}
Let $\gamma, \mu_0>0$ and $\eta\in C_c^\infty(\R^n)$. The function $u(x)=\eta(x)(\mu+|x|^2)^{\gamma/2} \in B^\gamma_{\infty, \infty}(\R^n)$ uniformly for $0\leq \mu\leq \mu_0$. 
\end{lem}

\begin{proof}[Proof of Lemma \ref{lem:keyBgamma}] It suffices to consider $\mu=0$. The case $\mu>0$ can be deduced by either a similar argument or by direction from $\mu=0$ by a H\"older estimate when $0<\gamma \leq 1$. 
It is clear that $|\Delta_j u(x)|\leq C(\eta, \gamma)$ for all $x\in \R^n$. 
We recall that
\begin{align}\label{eq:derivative-xalpha}
    K_j(x) = 2^{nj}K_0(2^jx) 
    \qquad\text{and}\qquad 
    \int_{\R^n} x^\alpha K_0(x)\;dx = 0
\end{align}
for any multi-index $\alpha\in \Z^n_{\geq 0}$. For all $j\geq 0$, we have
\begin{align*}
    \Delta_j u(x) &= \int_{\R^2} K_j(w)u(x-w)\;dw 
    = \int_{\R^n} K_0(w)u(x-2^{-j}w)\;dw.
\end{align*}
Let $\rho=2^{-j}$ and $m\in \mathbb{Z}_\geq 0$ be such that $m-1 \leq \gamma < m$. 
\medskip

\noindent
{\it \underline{Case 1}.} If $|x|\leq 2\rho$ then using $|u(x)|\leq C|x|^\gamma$, we have
    \begin{align*}
^{\prime}        |\Delta_ju(x)| &\leq C\int_{\R^n} K_0(y)|x-\rho y|^\gamma\;dy 
        \leq C|\rho|^\gamma\int_{\R^n} K_0(y) (2+|y|)^\gamma\;dy \leq C_\gamma \cdot 2^{-j\gamma}
    \end{align*}
    since $K_0\in \mathcal{S}(\R^n)$. Therefore $2^{j\gamma }|\Delta_j u(x)|\leq C_\gamma$ for all $|x|\leq 2\rho$.

\medskip

\noindent
{\it \underline{Case 2}.} If $|x|> 2\rho$, then $u\in C^\infty(\overline{B_{2r}(x))}$ where $r=\frac{1}{2}|x|$. We have
    \begin{align*}
        u(y) = \sum_{|\alpha|\leq m-1} \frac{D^\alpha u(x)}{\alpha!}(y-x)^\alpha + \sum_{|\beta| = m} R_\beta(y)(y-x)^\beta, \qquad y\in \overline{B_r(x)},
    \end{align*}
    where the first sum is the be the Taylor expansion polynomial of degree $m-1$ of $u$ at $x$, denoted by $T_{m-1}$, and the remainders satisfy 
    \begin{align*}
        R_\beta(y) = \frac{m}{\beta!}\int_0^1 (1-t)^{m-1}D^\beta u(x + t(y-x)) \;dt, \qquad |\beta| = m. 
    \end{align*}
    If $u(y) = \eta(y)|y|^\gamma$ for $\eta\in \mathrm{C}_c(\R^n)$, then for all $y\neq 0$ we have
    \begin{equation*}
        \max_{|\alpha| = m}|D^\alpha u(y)| \leq C_{n,m}(\eta)|y|^{\gamma-m}, \qquad m>\gamma. 
    \end{equation*}
    Hence, for every multi-index $\beta$ with $|\beta|=m$ we have
    \begin{align}\label{eq:TaylorRemainder}
        |R_\beta(y)| \leq \frac{|\beta|}{\beta!}\int_0^1 (1-t)^{m-1} C_{n,m}(\eta)|x+t(y-x)|^{\gamma-m}\;dt. 
    \end{align}
    Thanks to \eqref{eq:derivative-xalpha}, we have
    \begin{equation*}
        \int_{\R^n} T_{m-1}(x-\rho w) \;dw = 0.
    \end{equation*}
    Therefore
    \begin{align*}
        \Delta_j u(x)
        &=\int_{\R^n} K_0(w) \left( u(x-\rho w) - T_{m-1}(x-\rho w) \right)\;dw  = I_1+I_2,
    \end{align*}   
    where
    \begin{align*}
        I_1 &= \int_{\rho |w|\leq r} K_0(w) \left( u(x-\rho w) - T_{m-1}(x-\rho w) \right)\;dw \\
        I_2 &= \int_{\rho |w| > r} K_0(w) \left( u(x-\rho w) - T_{m-1}(x-\rho w) \right)\;dw.
    \end{align*}

    \medskip
    
    \noindent
    {\it Estimate $I_1$.} Let $y = x-\rho w$. If
    $y\in B_r(x)$ with $r=\frac{|x|}{2}$, then $|\rho w| \leq r= \frac{|x|}{2}$ and $y\in \overline{B}_r(x)$, and for all $t\in [0,1]$ we have
        \begin{align*}
            |x+t(y-x)| = |x+t\rho w| \geq |x| - |\rho w|  \geq |x| - \frac{|x|}{2} = \frac{|x|}{2}. 
        \end{align*}
        In view of \eqref{eq:TaylorRemainder}, for every multi-index $\beta$ with $|\beta|=m)$, we obtain that
        \begin{align*}
            |R_\beta(y)| 
                &\leq C_{n,m}(\eta) \left(\frac{|x|}{2}\right)^{\gamma-m} 
                 \leq C_{n,m}(\eta) \rho^{\gamma-m}
        \end{align*}
        since $|x|\geq 2\rho$ and $\gamma<m$.
        Hence, since $|y-x| = |\rho w|$, we have
        \begin{align*}
            \sum_{|\beta|=m} |R_\beta(y)|\cdot|y-x|^\beta 
                \leq 
            C_{n,m}(\eta)\cdot \rho^{\gamma-m} \rho^m |w|^m  
                = 
            C_{n,m}(\eta)\cdot \rho^\gamma |w|^m  .
        \end{align*}
        Therefore, 
        \begin{align*}
            I_1 \leq \int_{\rho|w|\leq r} K_0(w) \sum_{|\beta|=m} |R_\beta(y)|\cdot |y-x|^\beta \;dw 
            \leq 
            \left(C_{n,m}(\eta)   \int_{\R^n} K_0(w)|w|^m\right) \rho^\gamma. 
        \end{align*}

    \medskip

    \noindent
    {\it Estimate $I_2$.} If $|\rho w| > r = \frac{|x|}{2}$ then $|x|\leq 2\rho |w|$. We have
    \begin{align*}
        |u(x-\rho w)| \leq C(\eta)|x-\rho w|^\gamma \leq 3^\gamma C(\eta)  \cdot \rho^\gamma |w|^\gamma. 
    \end{align*}
    Similarly, we have
    \begin{align*}
        T_{m-1}(x-\rho w) 
            &= 
            \sum_{|\alpha|\leq m-1} \frac{|D^\alpha u(x)|}{\alpha!} \cdot |\rho w|^\alpha \\
            &\leq 
            \sum_{|\alpha|\leq m-1} \frac{1}{\alpha!} C_{n,|\alpha|}(\eta)|x|^{\gamma-|\alpha|} \cdot \rho^{|\alpha|}\cdot |w|^{|\alpha|}\\
            &\leq   
            \sum_{|\alpha|\leq m-1} \frac{C_{n,|\alpha|}(\eta)}{\alpha!} |2\rho w|^{\gamma-|\alpha|}\cdot \rho^{|\alpha|}\cdot |w|^{|\alpha|} \leq 2^{\gamma}C_{n,m}(\eta) \rho^\gamma |w|^\gamma. 
    \end{align*}
    Therefore
    \begin{align*}
        I_2\leq \Big(3^\gamma C(\eta) + 2^\gamma C_{n,m}(\eta)\Big)\left(\int_{\R^n} K_0(w)|w|^\gamma\;dw \right)\rho^\gamma . 
    \end{align*}
We obtain the conclusion. 
\end{proof}

To conclude this section, we prove the result in Remark \ref{rmk:prototypeCPDE}. We recall that $\omega=(\omega_1,\omega_2)\in \mathcal{D}(1,C_\omega,2)$ iff $\omega_2/\omega_1$ is badly approximable.

\begin{lem} \label{lem:prototype}
Let $n=2$, $\omega\in \mathcal{D}(1,C_\omega,2)$, and  
\begin{equation*}
    f({\bf y}) = (2-\sin(2\pi y_1)-\sin(2\pi y_2))^\gamma, \quad {\bf y}=(y_1,y_2)\in \T^2,
    \qquad V(x) = f(\omega x)\in C(\R). 
\end{equation*}
There exists $C>0$ independent of $\varepsilon$ such that, for $(x,t)\in \mathbb{R}\times (0,\infty)$ we have
\begin{equation}\label{eq:lower}
    u^\varepsilon(0,1) - u(0,1)  \geq -C
    \begin{cases}
        \varepsilon  &\gamma > 1, \\
        \varepsilon |\log(\varepsilon)| &\gamma =1, \\
        \varepsilon^{\gamma} &\gamma <1. 
    \end{cases}
\end{equation}
\end{lem}

\begin{proof}
For $0<\gamma < \infty$ and $0\leq \mu\leq \mu_0$, we define $f_{\mu}(x) = \big(\mu + f(x)\big)^{1/2}$ for $x\in \T^2$. Then by Lemma \ref{lem:keyBgamma}, 
    \(f_{\mu}\in B^\gamma_{\infty,\infty}(\T^2)\) uniformly for
    \(0\leq \mu\leq \mu_0\). Theorem~\ref{thm:BesovHolderCounting} yields
    \begin{align}\label{eq:sharpT}
        \left| 
            \frac{1}{T}\int_0^T (\mu + F(\omega_1 t, \omega_2t))^{1/2}\;dt - \int_{\T^2} (\mu+F(x))^{1/2}\;dx
        \right|  
        \leq C\begin{cases}
            \begin{aligned}
                &T^{-1} &&\quad  \gamma>1, \\
                &T^{-1}\log(T) &&\quad  \gamma=1, \\
                &T^{-\gamma} &&\quad  \gamma<1,
            \end{aligned}
        \end{cases}
    \end{align}
    where $C$ is a constant independent of $T$. By an argument similar to Theorem \ref{thm:homogenization} with $\sigma=1$, we deduce \eqref{eq:lower} from \eqref{eq:sharpT}. 
\end{proof}

The rates \eqref{eq:sharpT} provide a sharp improvement of
\cite[Corollary~4.3]{hu_polynomial_2025}. The latter gives only the rates
$T^{-1}$ for $\gamma>2$, $T^{-1/2}$ for $\gamma\in[1,2]$, and
$T^{-\frac{\gamma}{\gamma+1}}$ for $\gamma\in(0,1)$.

\section{Application: Statistical regularity of invariant measures with respect to perturbations}\label{Sec:AppStabilityMeasures}

Consider the family of differential equations
\begin{equation}\label{eq:ODE}
	\dot x = V(x,\dt),
	\qquad x\in\T^n,
\end{equation}
where $V:\T^n\times[-1,1]\to\R^n$ is a family of Lipschitz vector fields satisfying
\begin{equation}\label{ass:vectorfield}
	\|V(\cdot,\dt)-\omega\|_{L^\infty(\T^n)}
	\le
	|\delta|,
\end{equation}
for some $\omega\in \mathcal{D}(\sigma,C_\omega,n)$. Let $\phi_\delta^t$ denote the flow generated by \eqref{eq:ODE}.

\begin{defn}[Invariant measure of a flow]
A probability measure $\mu_\delta\in\mathbb P(\T^n,\R)$ is said to be invariant with respect to the flow $\phi_\delta^t$ if the push-forward by the flow satisfies $(\phi_\delta^t)_\#\mu_\delta=\mu_\delta$ for all $t\in \R$, i.e.,
\begin{equation*}
	\int_{\T^n} f\circ \phi^t_\delta(x)\;d \mu_\delta(x) = \int_{\T^n} f(x)\;d\mu_\delta(x) \qquad\text{for all}\; f\in C(\T^n), t\in \R. 
\end{equation*}
\end{defn}

\begin{defn}\label{defn:W1}
	For any two measures $\mu,\nu\in\mathbb P(\T^n,\R)$, the {\bf  1-Wasserstein distance} between them is defined by the {\bf Kantorovich--Rubinstein} dual formula
\begin{align*}
	\mathcal W_1(\mu,\nu)
	:=
	\sup
	\left\lbrace 
	\left|
	\int_{\T^n}f\,d\mu
	-
	\int_{\T^n}f\,d\nu
	\right|: f\in {\rm Lip}(\T^n), \| \nabla f\|_{L^\infty}\leq 1
	\right\rbrace.
\end{align*}	
\end{defn}

\begin{proof}[Proof of Theorem~\ref{thm:st-regu}]
Since $\omega$ is non-resonant, the linear flow $\phi_0^t(x)=x+\omega t$ is uniquely ergodic. Let $\mu_0$ denote its unique invariant probability measure. Then $\mu_0$ coincides with the normalized Lebesgue measure on $\T^n$, that is $d\mu_0(x) = dx$.
In what follows, we take $f\in \mathrm{Lip}(\T^n)$ with $\Vert \nabla f\Vert_{L^\infty(\T^n)} \leq 1$.
Since $\nu_\delta$ is invariant under $\phi_\delta^t$, we have
\begin{align*}
	\int_{\T^n}f\,d\mu_\delta
	=
	\frac{1}{T}
	\int_0^T
	\int_{\T^n}
	f(\phi_\delta^t(x))
	\,d\mu_\delta(x)\,dt \qquad\text{for every}\;T>0.
\end{align*}
Hence
\begin{align*}
	&\left|
	\int_{\T^n}f\,d\mu_\delta
	-
	\int_{\T^n}f\,d\mu_0
	\right| 
	= 
	\left|
		\frac{1}{T}\int_0^T f(\phi^t_\delta(x))\;d\mu_\delta(x)\;dt - \int_{\T^n} f(x)\;d\mu_0(x)
	\right| \leq I_1 + I_2,
\end{align*}
where 
\begin{align*}
	I_1 &=
		 \left|
		 \frac{1}{T}
		\int_0^T f(\phi^t_\delta(x))\;d\mu_\delta(x)\;dt - \int_{\T^n} f(x+\omega t)\;d\mu_\delta		 (x)
	\right| \\
	I_2 &=	
	\left|
		\frac{1}{T}\int_0^T f(x+\omega t	)\;d\mu_\delta(x)\;dt - \int_{\T^n} f(x)\;d\mu_0(x)
	\right|. 
\end{align*}
Applying Theorem~\ref{thm:BesovHolderCounting}-(iii) with $f\in C^{0,1}(\T^n)$, we obtain
\begin{align*}
	\left|
	\frac{1}{T}\int_0^T f(x+\omega t)\;dt - \int_{\T^n} f(y)\;dy
	\right|	
	\leq 
	\begin{cases}
	CT^{-1/\sigma}, & \qquad \sigma>1,\\	 
	CT^{-1}\log T, & \qquad \sigma=1. 
\end{cases}
\end{align*}
Integrating the above estimate with respect to $\mu_\delta$ yields
\begin{align*}
	I_2 = \left| 
	\frac{1}{T} \int_0^T\int_{\T^n}
	f(x+\omega t) \,d\mu_\delta(x)\,dt
	-
	\int_{\T^n}f\,d\mu_0 \right| 
	\leq
\begin{cases}
	CT^{-1/\sigma}, & \qquad \sigma>1,\\
	CT^{-1}\log T, & \qquad \sigma=1.
\end{cases}
\end{align*}
Let $\wt{\phi}_\delta^t$ be the lift of $\phi_\delta^t$ to $\mathbb{R}^n$. 
Since $V(\cdot, \delta)$ is $\mathbb{Z}^n$-periodic, the lifted flow satisfies
\begin{align*}
	\left| 	
	\widetilde{\phi}_\delta^t(x)
	- x \right|
	= \left| 
	\int_0^t
	V\left(\widetilde{\phi}_\delta^s(x), \delta\right)\,ds \right| .
\end{align*}
Therefore, since $V\in \mathrm{Lip}(\T^n\times [-1,1];\R^n)$ and $V(x,0) =\omega$, we have
\begin{align*}
\left| 	
	\widetilde{\phi}_\delta^t(x)
	- (x+\omega t) \right| \leq 
	\int_0^t \left|
		V\left(\widetilde{\phi}_\delta^s(x), \delta\right) 
		-
		V\left(\widetilde{\phi}_\delta^s(x), 0\right)
	\right|\;ds \leq \delta t. 
\end{align*}
Since $\Vert \nabla f\Vert_{L^\infty(\T^n)} \leq 1$, we obtain
\begin{align*}
I_1 = 	\frac{1}{T}\int_0^T  \int_{\T^n}
		\left|
		f\bigl(\phi_\delta^t(x)\bigr)
			-
		f(x+\omega t)
		\right|d\mu_\delta(x) dt \leq \frac{\delta T}{2}. 
\end{align*}
Combining the above estimates yields
\begin{equation*}
	\mathcal{W}_1(\mu_\delta,\mu_0) 
	\leq 
	\begin{cases}
	T^{-1/\sigma}+\delta T \leq \delta^{\frac{1}{1+\sigma}},
	& \sigma>1,
	\\[2mm]
	T^{-1}\log T+\delta T \leq \delta^{1/2}\log(1/\delta),
	& \sigma=1, 
\end{cases}
\end{equation*}
by choosing $T$ optimally in each case and that completes the proof of the upper bound. The lower bound part can be obtained directly from the following Lemma.
\end{proof}

\begin{lem}\label{Thm2}
	Let $n=2$ and $\omega=(1,\vartheta)\in\cD(\sigma, C_\omega,2)$. There exist a sequence $\{\delta_j\}_{j\in\mathbb{N}}$ with $\delta_j>0$, $\delta_j\to0$ as $j\to\infty$, and a corresponding sequence of vector fields $V_{\delta_j}(x) = V(x,\delta_j)$ such that
	\[
	\sup_{x\in\mathbb{T}^2}\bigl|V_{\delta_j}(x)-V_0\bigr|=\delta_j,
	\]
	and
	\begin{equation}\label{E9}
		\mathcal W_1\bigl(\nu_0,\nu_{\delta_j}\bigr)
		\ge C'
		\begin{cases}
			\,\delta_j^{1/(r+1)}, & \qquad 1<r<\sigma_*(\om),\\[2mm]
			\,\delta_j^{1/2}, & \qquad \sigma_*(\om)=1,
		\end{cases}
	\end{equation}
	for all $j\in\mathbb{N}$.
\end{lem}

\begin{proof} 
We consider two cases depending on the value of $\sigma_*(\omega)$:
\begin{itemize}
\item If $\sigma_*(\omega)=1$, then by \eqref{eq:Diophantineestimates}, there exist infinitely many pairs $(P_j,Q_j)\in\mathbb{Z}\times\mathbb{N}$ such that
\begin{equation}\label{eq:case1sigma1}
	\left|\vartheta - \frac{P_j}{Q_j}\right|\leq \frac{1}{Q_j^{2}}.
\end{equation}
\item If $\sigma_*(\omega)>1$, then due to the definition of $\sigma_*(\omega)$ in \eqref{eq:defn-diophantine-type} (see also \cite[Definition 5]{galatolo_quantitative_2022}, it follows that for any $1<r<\sigma_*(\omega)\leq\sigma$, there exist infinitely many pairs $(P_j,Q_j)\in\mathbb Z\times\mathbb N$ such that 
\begin{align}\label{eq:case1sigmaless1}
	\left|\vartheta-\frac{P_j}{Q_j}\right|\leq \frac{1}{Q_j^{r+1}}. 
\end{align}
\end{itemize}
In both cases, for each \(j\geq1\), we introduce the following constants and the corresponding constant vector fields:
\begin{equation*}
    \delta_j := \left|\vartheta-\frac{P_j}{Q_j}\right|,\qquad
    V_{\delta_j}(x) = V(x,\delta_j)\equiv \left(1,\frac{P_j}{Q_j}\right), \qquad V_0(x) = \left(1,\vartheta\right) \qquad x\in \T^2.
\end{equation*}
Then for all $j\in \N$, we have 
\begin{equation*}
	\|V_{\delta_j}-V_0\|_{L^\infty(\mathbb{T}^2)}=\delta_j.
\end{equation*}
For each $j\in \N$, the rational flow $\phi^t_{\delta_j}$ associated with $\dot{x} = V_{\delta_j}(x)$ is periodic, and its corresponding orbit starting at $(0,0)$ can be defined by 
\begin{align*}
	\ell_j = \left\lbrace \left(t, \frac{P_j}{Q_j}t\right) \quad (\mathrm{mod}\;\Z^2) : t\in \Z^2\right\rbrace.
\end{align*}
Let $\nu_{\delta_j}$ be the normalized one-dimensional Lebesgue measure on $\ell_j$. We have $\mathrm{supp}(\nu_{\delta_j}) = \ell_j$ and $\nu_{\delta_j}(\ell_j) = 1$. Furthermore, we can see that $\mu_{\delta_j}$ is an invariant measure with respect to the flow $\phi^t_{\delta_j}$. 

We estimate $\mathcal{W}_1(\nu_{\delta_j}, \nu_0)$ using a specific test function thanks to Definition \ref{defn:W1}. Let us define
\begin{align*}
	\psi(x) = d_{\mathbb{T}^2} (x,\ell_j), \qquad x\in \T^2,
\end{align*}
where $d_{\mathbb{T}^2}(x,\ell_j)$ is the distance function on the torus. It is clear that $\psi$ is Lipschitz and $|\nabla \psi(x)| \leq 1$. Since $\psi = 0$ on $\mathrm{supp}(\nu_{\delta_j}) = \ell_j$, we deduce that (Kantorovich--Rubinstein duality)
\begin{align*}
	\mathcal{W}_1(\nu_{\delta_j}, \nu_0)  \geq \left|\int_{\T^2} \psi\;d\nu_{\delta_j} - \int_{\T^2} \psi\;d\nu_{0}\right| = \int_{\T^2} \psi\;d\nu_{0} = \int_{\T^2} \psi(x)\;dx,
\end{align*}
since $\nu_0$ is the Lebesgue measure on $\T^2$. 
Let $\pi:\R^2\to \T^2$ be the quotient map by the relation $\T^2 = \R^2/\Z^2$.
For each $m\in \Z$, let us define 
\begin{equation*}
	d_m=\{(x_1,x_2)\in \R^2: x_1P_j -x_2Q_j = m\} \qquad \text{then}\qquad \pi^{-1}(\ell_j) = \bigcup_{m\in \Z} d_m. 
\end{equation*}
Let $x\in \T^2$, we say $\tilde{x}\in\R^2$ is a lift of $x$ if $\tilde{x}-x \in \Z^2$.
We compute
\begin{align*}
	\psi(x) &= \inf_{m\in \Z} \left\lbrace \mathrm{dist}\left(\tilde{x}, d_m\right):m\in \Z, \tilde{x}-x\in \Z^2 \right\rbrace 
	= \inf_{m\in \Z	}\frac{|\tilde{x}_1P_j - \tilde{x}_2Q_j-m|}{(P_j^2+Q_j^2)^{1/2}}.
\end{align*}
Let $s(x) = \{\tilde{x}_1P_j - \tilde{x}_2Q_j\}$ the the fractional part of $\tilde{x}_1P_j - \tilde{x}_2Q_j$. 
\begin{itemize}
\item We note that $s(x)$ is well-defined, i.e., it is independent of the lift $\tilde{x}$ of $x$. 
Consequently, 
\begin{align*}
	\psi(x) = \inf_{m\in \Z	}\frac{|\tilde{x}_1P_j - \tilde{x}_2Q_j-m|}{(P_j^2+Q_j^2)^{1/2}} = \frac{\min\{s(x), 1-s(x)\}}{(P_j^2+Q_j^2)^{1/2}}.
\end{align*}
\item If $g\in C([0,1);\R)$ then it is standard that
\begin{align*}
	\int_0^1\int_0^1 g(\{P_jx_1 - Q_jx_2\})\;dx_2 \;dx_1 = \int_0^1 g(s)\;ds. 
\end{align*}
Applying this identity with $g(s) = \min\{s,1-s\}$ we deduce that
\begin{align*}
	\int_{\T^2} \psi(x)\;dx = \frac{1}{(P_j^2+Q_j^2)^{1/2}} \int_{0}^1 \min\{s,1-s\}\;ds 
	= \frac{1}{4(P_j^2+Q_j^2)^{1/2}}. 
\end{align*}
\end{itemize}
Using the fact that if $j$ is large enough then $C_1(\vartheta)	\leq \left|\frac{P_j}{Q_j}\right| \leq C_2(\vartheta)$, we deduce that 
\begin{align*}
	\mathcal{W}_1(\nu_{\delta_j}, \nu_{0}) 
	&\geq 
	\int_{\T^2} \psi(x)\;dx = \frac{1}{4(P_j^2+Q_j^2)^{1/2}} \geq \frac{1}{4\sqrt{1+C(\vartheta)^2}} \cdot \frac{1}{Q_j} \\
	&\geq 
	\frac{1}{4\sqrt{1+C(\vartheta)^2}}
	\begin{cases}
	\delta_j^{1/2} &\quad \text{if}\; \sigma_*(\omega)=1,\\
	\delta_j^{1/(1+r)} &\quad \text{if}\; \sigma_*(\omega)>r>1,
	\end{cases}
\end{align*}
thanks to \eqref{eq:case1sigma1} and \eqref{eq:case1sigmaless1}, respectively. 
\end{proof}

\subsection*{Acknowledgement}  
S. Tu is grateful to Dorina Mitrea for fruitful discussions on Bessel potentials, which helped motivate the development of this work and Proposition \ref{prop:BirkhoffRateBesselWiener}. He also thanks Tuoc Phan and Artur Andrade for helpful discussions on Besov spaces. 
The authors acknowledge the hospitality of the Vietnam Institute for Advanced Study in Mathematics (VIASM), where part of this work was carried out during the Intensive Research Collaboration Program (IRCP), June 1--9, 2026.
The work of S. Tu is supported in part by the AMS-Simons travel grant.
The work of J. Zhang and S. Zhu are supported by the National Key R\&D Program of China (No. 2022YFA1007500) and the National Natural Science Foundation of China (No. 12231010, 12571207).

\subsection*{Statements and Declarations} The authors declare no competing interests, and no data were generated or analyzed in this study.

\subsection*{Author contributions} Son Tu and Jianlu Zhang contributed equally to the conception and overall development of the project, including its main theoretical results and applications, and to the preparation of the manuscript. Siyao Zhu participated in an early-stage exploration of a Fourier-based Jackson-kernel approach and drafted preliminary versions of examples associated with Theorem 1.2(i), (iv), and Theorem 1.5. All authors reviewed and approved the final manuscript.

\bigskip

\appendix

\section{Convergence rate of Birkhoff average under Bessel Potentials and Wiener Algebra}\label{a0}

\begin{defn}[Wiener algebra and Bessel Potential] \quad 
\begin{itemize}
\item[(i)] For $s\in \R$, the weighted Wiener algebra $A^s(\T^n)$ consists of functions $f\in L^1(\T^n)$ such that $(1+|\kappa|^2)^{\frac{s}{2}}\widehat{f}(\kappa)\in \ell^1(\Z^n)$, i.e., 
\begin{equation}\label{eq:AsWienerAlgebra}
     \Vert f\Vert_{A^s(\T^n)} = \sum_{\kappa\in \Z^n} (1+|\kappa|^2)^{\frac{s}{2}} |\widehat{f}(\kappa)| < \infty. 
\end{equation}

\item[(ii)] Let $\alpha>0$. We define the Bessel kernel by 
\begin{equation}\label{eq:GBessel}
    G_s(x) = \mathcal{F}^{-1}(1+4\pi^2|\xi|^2)^{-s/2}.     
\end{equation}
For $1\leq p\leq \infty$, the Bessel Potential spaces $L^{\alpha,p}(\T^n)$ is defined by
\begin{align*}
    L^{\alpha,p}(\T^n) = 
    \left\lbrace
        f = G_s*g:  g\in L^p(\R^n) 
    \right\rbrace, 
    \qquad \qquad 
    \Vert f\Vert_{L^{s,p}(\T^n)} := \Vert g\Vert_{L^p(\T^n)}. 
\end{align*}
It is standard that $L^{\alpha,2}(\T^n) \equiv H^{\alpha}(\T^n)$.

\end{itemize}
\end{defn}

\begin{lem}[Some Basic Embedding]\label{lem:besovWeinerAs}\quad 
\begin{enumerate}[label=(\alph*)] 

    \item (Sobolev Embedding) We have $H^s(\T^n)\subset C^{k,\alpha}(\T^n)$ where $k\in \N$, $\alpha\in(0,1)$, and $0<\alpha < \min \left\lbrace s-\frac{n}{2} - k, 1 \right\rbrace$. 	
    In particular, for $\alpha\in (0,1)$ and $\varepsilon>0$ small then $H^{\alpha+\frac{n}{2}+\varepsilon}(\T^n) \hookrightarrow C^{0,\alpha}(\T^n)$.
    
    \item (H\"older space and Wiener Algebra) If $f\in C^{0,\alpha}(\T^n)$ for $\alpha\in (0,1]$ then    
    \begin{equation*}
        |\kappa|^\alpha \cdot |\widehat{F}(\kappa)| \leq 2^{-(1+\alpha)} [F]_{C^{0,\alpha}(\T^n)} \qquad\text{for all}\;\kappa\in \Z^n\backslash\{0\}. 
    \end{equation*}
    As a consequence, if $0<\alpha\leq 1$ then $ C^{k,\alpha}(\T^n)\hookrightarrow A^{s}(\T^n)$ if $k+\alpha > s + n$. 

    	\item Let \(s>0\) and \(\varepsilon>0\). Then, for \(1\leq p\leq 2\), $L^{s+\frac np+\varepsilon,p}(\T^n)\subset A^s(\T^n)$. 
In particular, taking \(p=2\) gives $H^{s+\frac n2+\varepsilon}(\T^n)\subset A^s(\T^n)$. 
    	
\end{enumerate}
\end{lem}

We omit the proof of Lemma \ref{lem:besovWeinerAs}. 
The main result of this section is the following.

\begin{prop}[Wiener Algebra and Bessel Potentials]\label{prop:BirkhoffRateBesselWiener}
     Assume that $\omega\in \mathcal{D}(\sigma, C_\omega,n)$ and $x\in \T^n$.
\begin{itemize}
    \item[(i)] 
    If $f\in A^s(\T^n)$ for $s>0$, then
    \begin{equation}\label{eq:RateW}
         \left|
        \frac{1}{T}\int_0^T f(x + \omega t)\;dt - \int_{\T^n} f(y)\;dy  
        \right| \leq 
        \left(1+\frac{1}{\pi C_\omega}\right)\Vert f\Vert_{A^s(\T^n)} \left(\frac{1}{T} \right)^{\min \{{\frac{s}{\sigma}},1\}}. 
    \end{equation}
    
    \item[(ii)] Let $1<p\leq 2$ and $s>\frac{n}{p}$. If $f\in L^{s,p}(\T^n)$ then there exists $C=C(n,p,s,C_\omega)$ such that 
    \begin{align}\label{eq:BesselRateA}
        \mbox{\qquad}
        \left|\frac{1}{T}\int_{0}^T f(x+\omega t)\;dt  - \int_{\T^n} f(y)\;dy \right| 
        \leq 
        C\Vert f\Vert_{L^{s,p}(\T^n)}
        \begin{cases}
        \begin{aligned}
            & T^{-1} 
                && s > \sigma + n/p \\[1mm]
            &T^{-1}(\log T)^{1/p}
                && s = \sigma + n/p \\[1mm]
            &T^{-\frac{s-n/p}{\sigma}}     && s<\sigma + n/p. 
        \end{aligned}
        \end{cases}
    \end{align}
\end{itemize}
\end{prop}

The condition \(s>\frac{n}{p}\) ensures that \(L^{s,p}(\T^n)\subset C(\T^n)\) and that the Fourier series of \(f\) converges uniformly and absolutely due to Lemma~\ref{lem:besovWeinerAs}.

\begin{proof}[Proof of Proposition \ref{prop:BirkhoffRateBesselWiener}] Without loss of generality, we assume $x=0$. 
With $S_Nf(x) = \sum_{|\xi|\leq N} \widehat{f}(\xi) e^{2\pi i \xi\cdot x}$ we have 
\begin{align}\label{eq:maines}
     &\left|
        \frac{1}{T}\int_0^T f(\omega t)\;dt - \int_{\T^n} f(x)\;dx
    \right|\nonumber \\
    &\qquad \leq 
    \left|\frac{1}{T}\int_{0}^T S_Nf(\omega t)\;dt  - \int_{\T^n} f(x)\;dx \right|
    +
    \frac{1}{T}\left| \int_0^T \sum_{|\xi|> N} \widehat{f}(\xi) e^{2\pi i\xi \omega x} \;dx\right| . 
\end{align}
For the lower frequencies part, we estimate
\begin{align}\label{eq:e1}
    \left|\frac{1}{T}\int_{0}^T S_Nf(\omega t)\;dt  - \int_{\T^n} f(x)\;dx \right|
    &= 
    \left| \sum_{0<|\xi|\leq N} \widehat{f}(\xi) \frac{e^{2\pi i \xi\cdot \omega T} - 1}{2\pi i \xi \cdot \omega T} \right| \leq 
    \frac{1}{\pi C_\omega T}
    \sum_{0<|\xi|\leq N}  |\xi|^\sigma |\widehat{f}(\xi)|. 
\end{align}
\medskip

{\it (i).} 
    Since $f\in A^s(\T^n)$, we have $f\in C^0(\T^n)$ and $\widehat{f}\in \ell^1(\Z^n)$, we have $S_Nf(x) \to f$ uniformly in $C^0(\T^n)$.
    For $\xi\in \Z^n$ with $|\xi|\leq N$, we have
    \begin{equation*}
        \frac{1}{T}\sum_{0<|\xi|\leq N} |\xi|^\sigma |\widehat{f}(\xi)| \leq \frac{1}{T} \cdot N^{\sigma-s} \cdot\sum_{0<|\xi| \leq N} |\xi|^s |\widehat{f}(\xi)| 
        \leq \frac{N^{\sigma-s}}{T} \Vert f\Vert_{A^s(\T^n)}. 
    \end{equation*}
    For $\xi\in \Z^n$ with $|\xi| > N$, we have
    \begin{align*}
        \frac{1}{T}
        \left| 
            \int_0^T \sum_{|\xi|> N} \widehat{f}(\xi) e^{2\pi i\xi \omega t} \;dt\right| 
        \leq \sum_{|\xi|> N}  |\widehat{f}(\xi)| \leq N^{-s} \sum_{|\xi|>N} |\xi|^s|\widehat{f}(\xi)| \leq N^{-s} \Vert f\Vert_{A^s(\T^n)}. 
    \end{align*}
    From \eqref{eq:maines} and \eqref{eq:e1} we obtain
    \begin{align*}
        \left|
        \frac{1}{T}\int_0^T f(\omega t)\;dt - \int_{\T^n} f(x)\;dx
        \right|
        &\leq \frac{1}{\pi C_\omega} \frac{1}{T} N^{\sigma-s}\Vert f\Vert_{A^s(\T^n)} + N^{-s} \Vert f\Vert_{A^s(\T^n)}.
    \end{align*}
    If $s>\sigma$ then the we simply estimate $N^{\sigma-s} \leq 1$, then we obtain \eqref{eq:RateW} with the rate $T^{-1}$. Otherwise, if $0<s\leq \sigma$,
    choosing $N = T^{1/\sigma}$ we obtain \eqref{eq:RateW}. 

\bigskip

{\it (ii).} If $f\in L^{s,p}(\T^n)$then $f = G_s*g$ for $g\in L^p(\T^n)$. 
By the Hausdorff-Young inequality with $1< p\leq 2$ we have 
\begin{align*}
    \left\Vert (1+|\xi|^2)^\frac{s}{2}\widehat{f}(\xi)\right\Vert_{\ell^q(\Z^n)} 
    = 
    \Vert \widehat{g}\Vert_{\ell^q(\Z^n)} \leq \Vert g\Vert_{L^p(\T^n)} = \Vert f\Vert_{L^{s,p}(\T^n)}. 
\end{align*} 
By H\"older inequality for $p\in(1,2]$ and $q\in [2,\infty)$ with $\frac{1}{p}+\frac{1}{q} = 1$, we have
\begin{align}
    \sum_{0<|\xi|\leq N}  |\xi|^\sigma |\widehat{f}(\xi)|
    &\leq
    \sum_{0<|\xi|\leq N}  (1+|\xi|^2)^{\frac{\sigma-s}{2}}
    (1+|\xi|^2)^{\frac{s}{2}} |\widehat{f}(\xi)| \nonumber \\
    &\leq 
    \left(\sum_{0<|\xi|\leq N}  (1+|\xi|^2)^{\frac{\sigma-s}{2}\cdot p}\right)^{1/p}
    \left\Vert (1+|\xi|^2)^\frac{s}{2}\widehat{f}(\xi)\right\Vert_{\ell^q(\Z^n)} . \label{eq:e2}
\end{align}
\medskip

{\it \underline{Case 1.}}
If $s > \sigma+\frac{n}{p}$ then $s= \sigma+\frac{n}{p}+\varepsilon$ for some $\varepsilon>0$. Then $-(\sigma-s)p =n+p\varepsilon > n$, and thus 
\begin{equation*}
    \sum_{0<|\xi|\leq N}  (1+|\xi|^2)^{\frac{\sigma-s}{2}\cdot p}
    \leq 
    \sum_{\xi \in \Z^n\backslash \{0\}} (1+|\xi|^2)^{-\frac{n/p+\varepsilon}{2}\cdot p} < \infty . 
\end{equation*}
Combining \eqref{eq:e1} and \eqref{eq:e2}, then passing \(N\to\infty\) by Lemma~\ref{lem:besovWeinerAs}, yields the desired result.\medskip

{\it \underline{Case 2.}} If $s = \sigma+\frac{n}{p}$, then 
\begin{align*}
    \left(\sum_{0<|\xi|\leq N}  (1+|\xi|^2)^{\frac{\sigma-s}{2}\cdot p}\right)^{1/p}
    &\leq 
    \left(C_n\int_1^N r^{(\sigma-s)p} r^{n-1}\;dr\right)^{1/p}= C_{n,p}\log(N)^{1/p}, \\
    \left(\sum_{|\xi|> N}  (1+|\xi|^2)^{-\frac{sp}{2}}\right)^{1/p} 
    &\leq \left(C_n\int_N^\infty |r|^{-sp}r^{n-1}\;dr \right)^{1/p}
    = \frac{C_{n,p}}{(p\sigma)^{1/p}} N^{-\sigma}.
\end{align*}
Therefore
\begin{align*}
    \left|
        \frac{1}{T}\int_0^T f(\omega t)\;dt - \int_{\T^n} f(x)\;dx
    \right|
        &\leq 
        C_{n,p,\sigma}
        \Vert f\Vert_{L^{s,p}(\T^n)}
        \frac{(\log T)^{1/p}}{T}
\end{align*}
by choosing $N = \left\lfloor T^{1/\sigma}\right\rfloor$. \medskip 

{\it \underline{Case 3.}} If $\frac{n}{p} < s<\sigma+\frac{n}{p}$ then
\begin{align*}
    \left(\sum_{0<|\xi|\leq N}  (1+|\xi|^2)^{\frac{\sigma-s}{2}\cdot p}\right)^{1/p}
    \leq 
    \left(\sum_{0<|\xi|\leq N}  N^{(\sigma-s)p}\right)^{1/p} \leq C_{n,p}N^{\sigma+\frac{n}{p}-s},
\end{align*}
thanks to the fact that, the number of integer vectors $\xi$ with $|\xi|\leq N$ is bounded by $C_nN^n$.  For $\xi\in \Z^n$ with $|\xi| > N$, we have
    \begin{align*}
        \frac{1}{T}\left| \int_0^T \sum_{|\xi|> N} \widehat{f}(\xi) e^{2\pi i\xi \omega x} \;dx\right| 
        &\leq 
        \sum_{|\xi|>N} |\xi|^{-s}(1+|\xi|^2)^{\frac{s}{2}}|\widehat{f}(\xi)| \\
        & \leq 
        \left(\sum_{|\xi|>N} |\xi|^{-sp}\right)^{1/p}
        \Vert f\Vert_{L^{s,p}(\T^n)} 
        \leq 
        \frac{C_{n,p}}{sp-n}\cdot  N^{-(s-\frac{n}{p})}\Vert f\Vert_{L^{s,p}(\T^n)} 
    \end{align*}
since $s>\frac{n}{p}$. 
From \eqref{eq:maines}, \eqref{eq:e1}, \eqref{eq:e2} we have
\begin{align*}
    \left|
        \frac{1}{T}\int_0^T f(\omega t)\;dt - \int_{\T^n} f(x)\;dx
    \right|
        &\leq 
        \left(
            \frac{C_{n,p}}{\pi C_\omega} \frac{N^{\sigma+\frac{n}{p}-s}}{T} + \frac{p C_{n,p}}{sp-n} N^{-(s-\frac{n}{p})} 
        \right)\Vert f\Vert_{L^{s,p}(\T^n)}\\
        &\leq 
        \frac{C(n,p,s-\frac{n}{p})}{C_\omega} \Vert f\Vert_{L^{s,p}(\T^n)} \left(\frac{1}{T}\right)^{\frac{s-n/p}{\sigma}}.  
\end{align*}
We obtain the desired conclusion.
\end{proof}

%

\end{document}